\documentclass[11pt]{article}
\usepackage{amsmath,amsthm,mathrsfs,amsfonts,amssymb,stmaryrd,xy,pslatex}
\usepackage[a4paper,textwidth=140mm,textheight=195mm]{geometry} 
\usepackage[dvips]{hyperref}
\theoremstyle{definition} 
\newtheorem{remark}{Remark}[section]
\newtheorem{remarks}[remark]{Remarks}
\newtheorem{example}[remark]{Example}
\newtheorem{examples}[remark]{Examples} \theoremstyle{plain}
\newtheorem{definition}[remark]{Definition}
\newtheorem{theorem}[remark]{Theorem}
\newtheorem{proposition}[remark]{Proposition}

\newtheorem{lemma}[remark]{Lemma}

\newtheorem{notation}[remark]{Notation}
\newcommand{\intd}{\,\mathrm{d}}

\newcommand{\Fix}{\mathrm{Fix}}
\newcommand{\Cofix}{\mathrm{Cofix}}
\newcommand{\triv}{\mathbf{1}}

\newcommand{\wclose}[1]{\llbracket{#1}\rrbracket}
\newcommand{\complex}{\mathbb{C}}

\newcommand{\naturals}{\mathbb{N}}

\DeclareMathOperator{\Id}{id}

\DeclareMathOperator{\supp}{supp}

\DeclareMathOperator{\Ad}{Ad}

\DeclareMathOperator{\End}{End}

\newcommand{\smalldiagram}{}

\newcommand{\bfF}{\mathbf{F}}
\newcommand{\bfG}{\mathbf{G}}

\newcommand{\bfcs}{\mathbf{C^{*}\!\textrm{-}}}
\newcommand{\bfws}{\mathbf{W^{*}\!\textrm{-}}}

\newcommand{\bfrep}{\mathbf{rep}}
\newcommand{\bfcorep}{\mathbf{corep}}

\newcommand{\trivrep}{\mathbf{1}_{V}}
\newcommand{\bKbo}{{_{\frakB}\frakK_{\frakBo}}}

\newcommand{\hDelta}{\widehat{\Delta}}
\newcommand{\hepsilon}{\widehat{\epsilon}}

\newcommand{\frakA}{\mathfrak{A}}
\newcommand{\frakB}{\mathfrak{B}}
\newcommand{\frakC}{\mathfrak{C}}

\newcommand{\frakAo}{\mathfrak{A}^{\dag}}
\newcommand{\frakBo}{\mathfrak{B}^{\dag}}
\newcommand{\frakCo}{\mathfrak{C}^{\dag}}

\newcommand{\fraka}{\mathfrak{a}}
\newcommand{\frakb}{\mathfrak{b}}
\newcommand{\frakc}{\mathfrak{c}}
\newcommand{\frakd}{\mathfrak{d}}
\newcommand{\frake}{\mathfrak{e}}

\newcommand{\frakbo}{\mathfrak{b}^{\dag}}
\newcommand{\frakco}{\mathfrak{c}^{\dag}}
\newcommand{\frakdo}{\mathfrak{d}^{\dag}}

\newcommand{\frakH}{\mathfrak{H}}
\newcommand{\frakK}{\mathfrak{K}}
\newcommand{\frakL}{\mathfrak{L}}

\newcommand{\cbasel}[2]{(\mathfrak{#2},\mathfrak{#1},
\mathfrak{#1}^{\dag})}
\newcommand{\cbases}[2]{\cbasel{#1}{#2}}
\newcommand{\cbaseos}[2]{(\mathfrak{#2},\mathfrak{#1}^{\dag},
\mathfrak{#1})} \newcommand{\cbasesa}{\cbases{A}{H}}
\newcommand{\cbasesb}{\cbases{B}{K}}

\newcommand{\cbaseosb}{\cbaseos{B}{K}}

\newcommand{\hdelta}{\widehat\delta}
\newcommand{\hphi}{\widehat\phi}

\newcommand{\aHb}{{_{\alpha}H_{\beta}}}
\newcommand{\aHhb}{{_{\alpha}H_{\hbeta}}}
\newcommand{\cKd}{{_{\gamma}K_{\delta}}}
\newcommand{\cKhd}{{_{\gamma}K_{\hdelta}}}
\newcommand{\eLf}{{_{\epsilon}L_{\phi}}}
\newcommand{\eLhf}{{_{\epsilon}L_{\widehat\phi}}}

\newcommand{\lt}{\triangleleft}
\newcommand{\rt}{\triangleright}

\newcommand{\Lt}[1]{\triangleleft}
\newcommand{\Rt}[1]{\triangleright}

\newcommand{\hpi}{\widehat{\pi}}

\newcommand{\hbeta}{\widehat{\beta}}

\newcommand{\hsigma}{\widehat{\sigma}}

\newcommand{\hA}{\widehat{A}} \newcommand{\ha}{\widehat{a}}
\newcommand{\hB}{\widehat{B}} \newcommand{\hb}{\widehat{b}}

\newcommand{\frei}{\,\cdot\,}
\newcommand{\mycong}{\xrightarrow{\cong}}

\newcommand{\vntimes}{\,{\bar{\otimes}}\,}

\newcommand{\rtensor}[3]{ {_{#1}\! \underset{#2}{\otimes}\!
{}_{#3}}} \newcommand{\vtensor}[3]{ {_{#1}
\underset{#2}{\bar{\otimes}} {}_{#3}}}

\newcommand{\tensor}[1]{\underset{#1}{\otimes}}

\newcommand{\btensor}{\underset{\frakb}{\otimes}}
\newcommand{\botensor}{\underset{\frakbo}{\otimes}}
\newcommand{\ctensor}{\underset{\frakc}{\otimes}}
\newcommand{\rtensorh}{\underset{\frakb}{\otimes}}

\newcommand{\rbtensor}[2]{\rtensor{#1}{\frakb}{#2}}
\newcommand{\rbotensor}[2]{\rtensor{#1}{\frakbo}{#2}}

\newcommand{\HtensorK}{H \rbtensor{\beta}{\gamma} K}
\newcommand{\KtensorH}{K \rbotensor{\gamma}{\beta} H}

\newcommand{\fibre}[3]{ {_{#1}\! \underset{#2}{\ast}\!
{}_{#3}}}  

\newcommand{\bfibre}{\underset{\frakb}{\ast}}
\newcommand{\bofibre}{\underset{\frakb^{\dag}}{\ast}}

\newcommand{\fsource}{\rtensor{\hbeta}{\frakbo}{\alpha}}
\newcommand{\frange}{\rtensor{\alpha}{\frakb}{\beta}}

\newcommand{\tl}{\ensuremath \olessthan}
\newcommand{\tr}{\ensuremath \ogreaterthan}

\newcommand{\Hsource}{H \fsource H} \newcommand{\Hrange}{H
\frange H}

\newcommand{\sHsource}{H \fsource H} \newcommand{\sHrange}{H
\frange H}

\newcommand{\Hone}{H \fsource H \fsource H}
\newcommand{\Htwo}{H \frange H \fsource H}
\newcommand{\Hthree}{H \frange H \frange H}
\newcommand{\Hfour}{(\Hsource) \rtensor{(\alpha \lt
\alpha)}{\frakb}{\beta} H} \newcommand{\Hfive}{H
\rtensor{\hbeta}{\frakbo}{(\alpha \rt \alpha)} (\Hrange)}
\newcommand{\Hfourlt}{H \fsource H
\rtensor{\beta}{\frakbo}{\alpha} H}
\newcommand{\Hfourrt}{\big(\Hrange\big) \rtensor{\hbeta \lt
\beta}{\frakbo}{\alpha} H}

\newcommand{\rfsource}{\rtensor{\hdelta}{\frakbo}{\alpha}}
\newcommand{\rfrange}{\rtensor{\gamma}{\frakb}{\beta}}

\newcommand{\rHsource}{K \rfsource H}
\newcommand{\rHrange}{K \rfrange H} 

\newcommand{\rHone}{K
\rfsource H \fsource H} \newcommand{\rHtwo}{K \rfrange H
\fsource H} \newcommand{\rHthree}{K \rfrange H \frange H}
\newcommand{\rHfour}{(\rHsource) \rtensor{(\gamma \lt
\alpha)}{\frakb}{\beta} H} \newcommand{\rHfourlt}{K \rfsource
H \rtensor{\beta}{\frakbo}{\alpha} H}
\newcommand{\rHfourrt}{\big(\rHrange\big) \rtensor{(\hdelta
\lt \beta)}{\frakbo}{\alpha} H} \newcommand{\rHfive}{K
\rtensor{\hdelta}{\frakbo}{(\alpha \rt \alpha)} (\Hrange)}

\newcommand{\crfsource}{\rtensor{\hbeta}{\frakbo}{\gamma}}
\newcommand{\crfrange}{\rtensor{\alpha}{\frakb}{\delta}}

\newcommand{\crHrange}{H
  \rtensor{\alpha}{\frakb}{\delta} K}
\newcommand{\crHsource}{H \rtensor{\hbeta}{\frakbo}{\gamma} K}
\newcommand{\wrHsource}{K {_{\hsigma}\tl} \alpha}
\newcommand{\wrHrange}{K {_{\sigma}\tl} \beta}
\newcommand{\wrHone}{K {_{\hsigma}\tl} \alpha
{_{\rho_{\hbeta}}\tl} \alpha} \newcommand{\wrHtwo}{K
{_{\sigma}\tl} \beta {_{\rho_{\hbeta}}\tl} \alpha}
\newcommand{\wrHthree}{K {_{\sigma}\tl} \beta
{_{\rho_{\alpha}}\tl} \beta} \newcommand{\wrHfour}{(K
{_{\hsigma}\tl} \alpha) {_{\sigma \tl \Id}\tl} \beta}
\newcommand{\wrHfourlt}{K {_{\hsigma}\tl} \alpha
{_{\rho_{\beta}}\tl} \alpha} \newcommand{\wrHfourrt}{(K
{_{\sigma}\tl} \beta) {_{\hsigma \tl \Id}\tl} \alpha}
\newcommand{\wrHfive}{K {_{\hsigma}\tl} (\alpha
\tr_{\rho_{\beta}} \alpha)}

\newcommand{\crHone}{H \fsource H \crfsource K}
\newcommand{\crHtwo}{H \frange H \crfsource K}
\newcommand{\crHthree}{H \frange H \crfrange K}
\newcommand{\crHfour}{(\Hsource) \rtensor{(\alpha \lt
    \alpha)}{\frakb}{\delta{}} K} \newcommand{\crHfive}{H
  \rtensor{\hbeta}{\frakbo}{(\alpha \rt \gamma)} (\crHrange)}

\newcommand{\fibreab}{\fibre{\alpha}{\frakb}{\beta}}
\newcommand{\fibrebc}{\fibre{\beta}{\frakb}{\gamma}}
\newcommand{\AfibreA}{A \fibreab A}

\newcommand{\AfibreB}{A \fibrebc B}

 \newcommand{\Cl}{C_{\lambda}}
\newcommand{\gB}{{_{\gamma}B}}

\newcommand{\kalpha}[1]{|\alpha\rangle_{{#1}}}
\newcommand{\balpha}[1]{\langle\alpha|_{{#1}}}
\newcommand{\kbeta}[1]{|\beta{}\rangle_{{#1}}}
\newcommand{\bbeta}[1]{\langle\beta|_{{#1}}}
\newcommand{\khbeta}[1]{|\hbeta{}\rangle_{{#1}}}
\newcommand{\bhbeta}[1]{\langle\hbeta|_{{#1}}}

\newcommand{\kgamma}[1]{|\gamma\rangle_{{#1}}}
\newcommand{\bgamma}[1]{\langle\gamma|_{{#1}}}

\newcommand{\pmuspace}{(H,\hbeta,\alpha,\beta)}
\newcommand{\pmuoperator}{V \colon
  \Hsource \to \Hrange}
\newcommand{\pmu}{(\frakb,H,\hat\beta,\alpha,\beta,V)}

\newcommand{\GrrG}{G {_{r}\times_{r}} G}
\newcommand{\GsrG}{G {_{s}\times_{r}} G}

\newcommand{\leg}[1]{#1}
\title{$C^{*}$-pseudo-multiplicative unitaries and Hopf $C^{*}$-bimodules}

\author{Thomas Timmermann\\[1ex]
\texttt{timmermt@math.uni-muenster.de}\\ SFB 478
``Geometrische Strukturen in der Mathematik''\\ Hittorfstr.\
27, 48149 M\"unster}

\date{\today}

\begin{document} \xyrequire{matrix} \xyrequire{arrow}
\xyrequire{curve} 

\maketitle

\abstract{We introduce $C^{*}$-pseudo-multiplicative
  unitaries and concrete Hopf $C^{*}$-bi\-modules for the
  study of quantum groupoids in the setting of
  $C^{*}$-algebras. These unitaries and Hopf
  $C^{*}$-bimodules generalize multiplicative unitaries and
  Hopf $C^{*}$-algebras and are analogues of the
  pseudo-multi\-plicative unitaries and Hopf--von
  Neumann-bimod\-ules studied by Enock, Lesieur and Vallin.
  To each $C^{*}$-pseudo-multiplicative unitary, we
  associate two Fourier algebras with a duality pairing, a
  $C^{*}$-tensor category of representations, and in the
  regular case two reduced and two universal
  Hopf $C^{*}$-bimodules.  The theory is illustrated by
  examples related to locally compact Hausdorff
  groupoids. In particular, we obtain a continuous Fourier
  algebra for a locally compact Hausdorff groupoid.}

\section{Introduction}
\label{section:introduction}

Multiplicative unitaries, which were first systematically
studied by Baaj and Skandalis \cite{baaj:2}, are fundamental
to the theory of quantum groups in the setting of operator
algebras and to generalizations of Pontrjagin duality
\cite{timmer:buch}. First, one can associate to every
locally compact quantum group a multiplicative unitary
\cite{vaes:1,vaes:30,masuda}.  Out of this unitary, one can
construct two Hopf $C^{*}$-algebras, where one coincides
with the initial quantum group, while the other is the
generalized Pontrjagin dual of the quantum group. The
duality manifests itself by a pairing on dense Fourier
subalgebras of the two Hopf $C^{*}$-algebras. These Hopf
$C^{*}$-algebras can be completed to Hopf--von Neumann
algebras and are reduced in the sense that they correspond
to the regular representations of the quantum group and of
its dual, respectively. Considering arbitrary
representations, one can also construct out of the
associated unitary two universal Hopf $C^{*}$-algebras with
morphisms onto the reduced ones.  In the study of coactions
of quantum groups on algebras, the unitary is an essential
tool for the construction of dual coactions on the reduced
crossed products and in the proof of biduality \cite{baaj:2}
which generalizes the Takesaki--Takai duality.

Much of the theory of quantum groups has been generalized
for quantum groupoids in a variety of settings, for example,
for finite quantum groupoids in the setting of
finite-dimensional $C^{*}$-algebras by B\"ohm, Szlach\'anyi,
Nikshych and others \cite{boehm:1,boehm:2,boehm:3,nikshych}
and for measurable quantum groupoids in the setting of von
Neumann algebras by Enock, Lesieur and Vallin
\cite{enock:10,enock:1,enock:2,lesieur}.  Fundamental for
the second theory are the Hopf--von Neumann bimodules and
pseudo-multiplicative unitaries introduced by Vallin
\cite{vallin:1,vallin:2}.

In this article, we introduce generalizations of
multiplicative unitaries and Hopf $C^{*}$-algebras that are
suited for the study of locally compact quantum groupoids in
the setting of $C^{*}$-algebras, and extend many of the
results on multiplicative unitaries that were obtained by
Baaj and Skandalis in \cite{baaj:2}. In particular, we
associate to every regular $C^{*}$-pseudo-multiplicative
unitary two Hopf $C^{*}$-bimodules and two Fourier algebras
with a duality pairing, and construct universal Hopf
$C^{*}$-bimodules from a $C^{*}$-tensor category of
representations of the unitary.  The theory presented here
was applied already in \cite{timmer:leiden} to the
definition and study of compact $C^{*}$-quantum groupoids,
and will be applied in a forthcoming article to the study of
reduced crossed products for coactions of Hopf
$C^{*}$-bimodules on $C^{*}$-algebras and to an extension of
the Baaj-Skandalis duality theorem; see also
\cite{timmer:ckac}.

Our concepts are related to their von~Neumann-algebraic
counterparts as follows.  In the theory of quantum groups,
one can use the multiplicative unitary to pass between the
setting of von Neumann algebras and the setting of
$C^{*}$-algebras and thus obtains a bijective
correspondence between measurable and locally compact
quantum groups. This correspondence breaks down for quantum
groupoids --- already for ordinary spaces, considered as
groupoids consisting entirely of units, a measure does not
determine a topology. In particular, one can not expect to
pass from a measurable quantum groupoid in the setting of
von~Neumann algebras to a locally compact quantum groupoid
in the setting of $C^{*}$-algebras in a canonical way. The
reverse passage, however, is possible, at least on the level
of the unitaries and the Hopf bimodules.

Fundamental to our approach is the framework of modules,
relative tensor products and fiber products in the setting
of $C^{*}$-algebras introduced in \cite{timmer:fiber}. That
article also explains in detail how the theory developed
here can be reformulated in the setting of von Neumann
algebras, where we recover Vallin's notions of a
pseudo-multiplicative unitary and a Hopf--von Neumann
bimodule, and how to pass from the level of $C^{*}$-algebras
to the setting of von Neumann algebras by means of various
functors.

The theory presented here overcomes several restrictions of
our former generalizations of multiplicative unitaries and
Hopf $C^{*}$-algebras \cite{timmermann:hopf}; see also
\cite{timmer:comparison}.

This work was supported by the SFB 478 ``Geometrische
Strukturen in der Mathematik''\footnote{funded by the
Deutsche Forschungsgemeinschaft (DFG)} and initiated during
a stay at the ``Special Programme on Operator Algebras'' at
the Fields Institute in Toronto, Canada, in July 2007.

\paragraph{Organization} This article is organized as
follows.
We start with preliminaries, summarizing notation,
terminology and some background on Hilbert $C^{*}$-modules.

In Section \ref{section:pmu}, we recall the notion of a
multiplicative unitary and define
$C^{*}$-pseudo-multiplicative unitaries. This definition involves $C^{*}$-modules over $C^{*}$-bases and their
relative tensor product, which were introduced in
\cite{timmer:fiber} and which we briefly recall. As an
example, we construct the $C^{*}$-pseudo-multiplicative
unitary of a locally compact Hausdorff groupoid. We shall
come back to this example frequently.

In Section \ref{section:legs}, we associate to every
well-behaved $C^{*}$-pseudo-multiplicative unitary two Hopf
$C^{*}$-bimodules. These Hopf $C^{*}$-bimodules are
generalized Hopf $C^{*}$-algebras, where the target of the
comultiplication is no longer a tensor product but a fiber
product that is taken relative to an underlying
$C^{*}$-base.  Inside these Hopf $C^{*}$-bimodules, we
identify dense convolution subalgebras which can be
considered as generalized Fourier algebras, and construct a
dual pairing on these subalgebras. To illustrate the theory,
we apply all constructions to the unitary associated to a
groupoid $G$, where one recovers the reduced groupoid
$C^{*}$-algebra of $G$ on one side and the function algebra
of $G$ on the other side.

In Section \ref{section:reps}, we study representations and
corepresentations of $C^{*}$-pseudo-multi\-plicative
unitaries. These (co)representations form a $C^{*}$-tensor
category and lead to the construction of universal variants
of the Hopf $C^{*}$-bimodules introduced in Section
\ref{section:legs}.  For the unitary associated to a
groupoid, we establish a categorical equivalence between
corepresentations of the unitary and representations of the
groupoid.

In Section \ref{section:props}, we show that every
$C^{*}$-pseudo-multiplicative unitary satisfying a certain
regularity condition is well-behaved. This condition is
satisfied, for example, by the unitaries associated to
groupoids and by the unitaries associated to compact quantum
groupoids. Furthermore, we collect some results on proper
and \'etale $C^{*}$-pseudo-multiplicative unitaries.

\paragraph{Terminology and notation }

Given a subset $Y$ of a normed space $X$, we denote by $[Y]
\subset X$ the closed linear span of $Y$.  We call a linear
map $\phi$ between normed spaces {\em contractive} or a {\em
  linear contraction} if $\|\phi\|\leq 1$.

All sesquilinear maps like inner products of Hilbert spaces
are assumed to be conjugate-linear in the first component
and linear in the second one.  Let $H,K$ be Hilbert spaces.
We canonically identify ${\cal L}(H,K)$ with a subspace of
${\cal L}(H \oplus K)$. Given subsets $X \subseteq {\cal
L}(H)$ and $Y \subseteq {\cal L}(H,K)$, we denote by $X'$
the commutant of $X$ and by $\wclose{Y}$ the $\sigma$-weak
closure of $Y$.

Given a $C^{*}$-subalgebra $A \subseteq {\cal L}(H)$ and a
$*$-homomorphism $\pi \colon A \to {\cal L}(K)$, we put
 \begin{align} \label{eq:intertwiner} {\cal L}^{\pi}(H,K) :=
\{ T \in {\cal L}(H,K) \mid Ta = \pi(a)T \text{ for all } a
\in A\}.
 \end{align}

 We use the ket-bra notation and define for each $\xi \in H$
operators $|\xi\rangle \colon \complex \to H$, $\lambda
\mapsto \lambda \xi$, and $\langle\xi|=|\xi\rangle^{*}
\colon H \to \complex$, $\xi' \mapsto \langle
\xi|\xi'\rangle$.

We shall use some theory of groupoids; for background, see
\cite{renault} or \cite{paterson}. Given a groupoid $G$, we
denote its unit space by $G^{0}$, its range map by $r$, its
source map by $s$, and let  $\GrrG =
\{(x,y) \in G \times G \mid r(x)=r(y)\}$, $\GsrG = \{(x,y)
\in G \times G \mid s(x)=r(y)\}$ and $G^{u}=r^{-1}(u)$,
$G_{u}=s^{-1}(u)$ for each $u \in G^{0}$.

We shall make extensive use of (right) Hilbert
$C^{*}$-modules and the internal tensor product; a standard
reference is \cite{lance}.  Let $A$ and $B$ be
$C^{*}$-algebras.  Given Hilbert $C^{*}$-modules $E$ and $F$
over $B$, we denote by $\mathcal{L}_{B}(E,F)$ the space of
all adjointable operators from $E$ to $F$.  Let $E$ and $F$
be $C^{*}$-modules over $A$ and $B$, respectively, and let
$\pi \colon A \to \mathcal{L}_{B}(F)$ be a
$*$-homomorphism. Recall that the internal tensor product $E
\otimes_{\pi} F$ is the Hilbert $C^{*}$-module over $B$
which is the closed linear span of elements $\eta
\otimes_{\pi} \xi$, where $\eta \in E$ and $\xi \in F$ are
arbitrary and $\langle \eta \otimes_{\pi} \xi|\eta'
\otimes_{\pi} \xi'\rangle = \langle
\xi|\pi(\langle\eta|\eta'\rangle)\xi'\rangle$ and $(\eta
\otimes_{\pi} \xi)b=\eta \otimes_{\pi} \xi b$ for all
$\eta,\eta' \in E$, $\xi,\xi' \in F$,  $b \in B$
\cite[\S 4]{lance}.  We denote the internal tensor product
by ``$\tr$'' and drop the index $\pi$ if the representation
is understood; thus, for example, $E \tr F=E \tr_{\pi} F=E
\otimes_{\pi} F$.


We also define a {\em flipped internal tensor product} $F
{_{\pi}\tl} E$ as follows. We equip the algebraic tensor
product $F \odot E$ with the structure maps $\langle \xi
\odot \eta | \xi' \odot \eta'\rangle := \langle \xi|
\pi(\langle \eta|\eta'\rangle) \xi'\rangle$, $(\xi \odot
\eta) b := \xi b \odot \eta$, form the separated completion,
and obtain a Hilbert $C^{*}$-module $F {_{\pi}\tl} E$ over
$B$ which is the closed linear span of elements $\xi
{_{\pi}\tl} \eta$, where $\eta \in E$ and $\xi \in F$ are
arbitrary and $\langle \xi {_{\pi}\tl} \eta|\xi'
{_{\pi}\tl} \eta'\rangle = \langle
\xi|\pi(\langle\eta|\eta'\rangle)\xi'\rangle$ and $(\xi
{_{\pi}\tl} \eta)b=\xi b {_{\pi}\tl} \eta$ for all
$\eta,\eta' \in E$, $\xi,\xi' \in F$,  $b\in B$. As
above, we drop the index $\pi$ and simply write ``$\tl$''
instead of ``${_{\pi}\tl}$'' if the representation $\pi$ is
understood.  Evidently, the usual and the flipped internal
tensor product are related by a unitary map $\Sigma \colon F
\tr E \mycong E \tl F$, $\eta \tr \xi \mapsto \xi \tl \eta$.

For each $\xi \in E$, the maps
\begin{align} \label{eq:ketbra} l^{\pi}_{F} (\xi) &\colon F \to E \tr F, \
\eta \mapsto \xi \tr \eta, & r^{\pi}_{F}(\xi) &\colon F \to
F \tl E, \ \eta \mapsto \eta \tl \xi,
\end{align} are adjointable operators, and for all $\eta
\in F$, $\xi' \in E$,
\begin{align*} l^{\pi}_{F}(\xi)^{*}(\xi' \tr \eta) &=
\pi(\langle \xi|\xi'\rangle) \eta =
r^{\pi}_{F}(\xi)^{*}(\eta \tl \xi').
\end{align*} Again, we drop the supscript $\pi$ in
$l^{\pi}_{F}(\xi)$ and $r^{\pi}_{F}(\xi)$ if this
representation is understood.

Finally, let $E_{1},E_{2}$ be Hilbert $C^{*}$-modules over
$A$, let $F_{1}$, $F_{2}$ be Hilbert $C^{*}$-modules over
$B$ with representations $\pi_{i} \colon A \to
\mathcal{ L}_{B}(F_{i})$ ($i=1,2$), and let $S \in \mathcal{
  L}_{A}(E_{1},E_{2})$, $T \in \mathcal{
  L}_{B}(F_{1},F_{2})$ such that $T\pi_{1}(a)=\pi_{2}(a)T$
for all $a \in A$. Then there exists a unique operator $S
\tr T \in \mathcal{ L}_{B}(E_{1} \tr F_{1}, E_{2} \tr
F_{2})$ such that $(S \tr T)(\eta \tr \xi)= S\eta \tr T\xi$
for all $\eta \in E_{1}$, $\xi \in F_{1}$, and $(S\tr
T)^{*}=S^{*} \tr T^{*}$ \cite[Proposition 1.34]{echter}.

\section{$C^{*}$-pseudo-multiplicative unitaries}
\label{section:pmu}

Recall that a multiplicative unitary on a Hilbert space $H$
is a unitary $V \colon H \otimes H \to H \otimes H$ that
satisfies the {\em pentagon equation}
$V_{12}V_{13}V_{23}=V_{23}V_{12}$ (see \cite{baaj:2}).
Here,  $V_{12},V_{13},V_{23}$  are operators on $H \otimes
H \otimes H$ defined by $V_{12} = V \otimes \Id$, $V_{23} =
\Id \otimes V$, $V_{13} = (\Sigma \otimes \Id)V_{23}(\Sigma
\otimes \Id) = (\Id \otimes \Sigma) V_{12}(\Id \otimes
\Sigma)$, where $\Sigma \in {\cal L}(H \otimes H)$ denotes
the flip $\eta \otimes \xi \mapsto \xi \otimes \eta$.
As an example, consider a locally compact group $G$ with left
Haar measure $\lambda$.  Then the formula 
\begin{align} \label{eq:pmu-group}
  (V f)(x,y) &= f(x,x^{-1}y) 
\end{align}
defines a linear bijection of $C_{c}(G \times G)$ that
extends to a unitary on $L^{2}(G \times G,\lambda
\otimes \lambda) \cong L^{2}(G,\lambda) \otimes
L^{2}(G,\lambda)$. This unitary is multiplicative,
and the pentagon equation amounts to associativity of the
multiplication in $G$.

In this section, we generalize the notion of a
multiplicative unitary so that it covers the example above
if we replace the group $G$ by a locally compact Hausdorff
groupoid $G$. In that case, formula \eqref{eq:pmu-group}
only makes sense for $(x,y) \in \GrrG$ and defines a linear
bijection from $C_{c}(\GsrG)$ to $C_{c}(\GrrG)$. If the
groupoid $G$ is finite, that bijection is a unitary from
$l^{2}(\GsrG)$ to $l^{2}(\GrrG)$, and these two Hilbert
spaces can be identified with tensor products of $l^{2}(G)$
with $l^{2}(G)$, considered as a module over the algebra
$C(G^{0})$ with respect to representations that are
naturally induced by the maps $s,r\colon G \to G^{0}$.  For
a general groupoid, the simple algebraic tensor product of
modules has to be replaced by a refined version. In the
setting of von Neumann algebras, Vallin used the relative
tensor product of Hilbert modules introduced by Connes, also
known as Connes' fusion of correspondences, to define
pseudo-multiplicative unitaries \cite{vallin:2} which
include as a main example the unitary of a measurable
groupoid. To take the topology of $G$ into account, we shall
work in the setting of $C^{*}$-algebras and use the relative
tensor product of $C^{*}$-modules over $C^{*}$-bases
introduced in \cite{timmer:fiber}.

\subsection{The relative tensor product of $C^{*}$-modules over $C^{*}$-bases}

\label{subsection:pmu-rtp}

In this subsection, we recall the relative tensor product of
$C^{*}$-modules over $C^{*}$-bases which is fundamental to
the definition of a $C^{*}$-pseudo-multiplicative unitary,
and generalize the theory presented in \cite[Section
2]{timmer:fiber} in two respects. First, we introduce the
notion of a semi-morphism between $C^{*}$-modules which will
be important in subsection
\ref{subsection:rep-universal}. Second, the definition of a
$C^{*}$-pseudo-multiplicative unitaries forces us to
consider $C^{*}$-$n$-modules for $n \geq 2$ and not only
$C^{*}$-bimodules. We shall not give separate proofs of
statements that are only mild generalizations of statements
found in \cite{timmer:fiber}.
For additional motivation and details, we refer to
\cite{timmer:fiber}; an extended example can be found in
subsection \ref{subsection:pmu-groupoid}.

\paragraph{$C^{*}$-modules over $C^{*}$-bases}
A {\em $C^{*}$-base} is a triple $\cbasesb$ consisting of a
Hilbert space $\frakK$ and two commuting nondegenerate
$C^{*}$-algebras $\frakB,\frakBo \subseteq
\mathcal{L}(\frakK)$.  A $C^{*}$-base should be thought of
as a $C^{*}$-algebraic counterpart to  pairs consisting
of a von Neumann algebra and its commutant. As an example,
one can associate to every faithful KMS-state $\mu$ on a
$C^{*}$-algebra $B$ the $C^{*}$-base $(H_{\mu},B,B^{op})$,
where $H_{\mu}$ is the GNS-space for $\mu$ and $B$ and
$B^{op}$ act on $H_{\mu}=H_{\mu^{op}}$ via the
GNS-representations \cite[Example 2.9]{timmer:fiber}.  If
$\frakb=\cbasesb$ is a $C^{*}$-base, then so is
$\frakbo:=\cbaseosb$ and
$M(\frakb):=(\frakK,M(\frakB),M(\frakBo))$, where 
 $M(\frakB)$ and $M(\frakBo)$ are
naturally represented of $\frakK$.

From now on, let $\frakb=\cbasesb$ be a $C^{*}$-base.  We
shall use the following notion of a $C^{*}$-module. A {\em
  $C^{*}$-$\frakb$-module} is a pair
$H_{\alpha}=(H,\alpha)$, where $H$ is a Hilbert space and
$\alpha \subseteq \mathcal{L}(\frakK,H)$ is a closed
subspace satisfying $[\alpha \frakK]=H$, $[\alpha \frakB]=
\alpha$, and $[\alpha^{*}\alpha] = \frakB \subseteq
\mathcal{L}(\frakK)$.  If $H_{\alpha}$ is a
$C^{*}$-$\frakb$-module, then $\alpha$ is a Hilbert
$C^{*}$-module over $B$ with inner product $(\xi,\xi')
\mapsto \xi^{*}\xi'$ and there exist isomorphisms
\begin{align} \label{eq:rtp-iso}
  \alpha \tr \frakK &\to H, \ \xi \tr \zeta
  \mapsto \xi \zeta, & \frakK \tl \alpha &\to H, \
  \zeta \tl \xi \mapsto \xi\zeta,
\end{align}
and a nondegenerate representation
\begin{align*}
  \rho_{\alpha} \colon \frakBo \to \mathcal{L}(H), \quad
  \rho_{\alpha}(b^{\dagger})(\xi \zeta)= \xi b^{\dagger} \zeta
  \quad \text{for all } b^{\dagger} \in \frakBo, \xi \in
  \alpha, \zeta \in \frakK.
\end{align*}
A {\em semi-morphism} between $C^{*}$-$\frakb$-modules
$H_{\alpha}$ and $K_{\beta}$ is an operator $T \in
\mathcal{L}(H,K)$ satisfying $T\alpha \subseteq \beta$. If
additionally $T^{*}\beta \subseteq \alpha$, we call $T$ a
{\em morphism}. We denote the set of all (semi-)morphisms by
$\mathcal{ L}_{(s)}(H_{\alpha},K_{\beta})$. If $T \in
\mathcal{ L}_{s}(H_{\alpha},K_{\beta})$, then
$T\rho_{\alpha}(b^{\dagger}) = \rho_{\beta}(b^{\dagger})T$
for all $b^{\dagger} \in \frakBo$, and if additionally $T
\in \mathcal{ L}(H_{\alpha},K_{\beta})$, then left
multiplication by $T$ defines an operator in $\mathcal{
  L}_{\frakB}(\alpha,\beta)$ which we again denote by $T$.

We shall use the following notion of $C^{*}$-bi- and
$C^{*}$-$n$-modules.  Let $\frakb_{1},\ldots,\frakb_{n}$ be
$C^{*}$-bases, where
$\frakb_{i}=(\frakK_{i},\frakB_{i},\frakBo_{i})$ for each
$i$.  A {\em
  $C^{*}$-$(\frakb_{1},\ldots,\frakb_{n})$-module} is a
tuple $(H,\alpha_{1},\ldots,\alpha_{n})$, where $H$ is a
Hilbert space and $(H,\alpha_{i})$ is a
$C^{*}$-$\frakb_{i}$-module for each $i$ such that
$[\rho_{\alpha_{i}}(\frakBo_{i})\alpha_{j}]=\alpha_{j}$
whenever $i\neq j$.  In the case $n=2$, we abbreviate
$\aHb:=(H,\alpha,\beta)$.  We note that if
$(H,\alpha_{1},\ldots,\alpha_{n})$ is a
$C^{*}$-$(\frakb_{1},\ldots,\frakb_{n})$-module, then
$[\rho_{\alpha_{i}}(\frakBo_{i}),\rho_{\alpha_{j}}(\frakBo_{j})]=0$
whenever $i \neq j$.  The set of {\em (semi-)morphisms}
between $C^{*}$-$(\frakb_{1},\ldots,\frakb_{n})$-modules
$\mathcal{H}=(H,\alpha_{1},\ldots,\alpha_{n})$,
$\mathcal{K}=(K,\gamma_{1},\ldots,\gamma_{n})$ is $\mathcal{
  L}_{(s)}(\mathcal{H}, \mathcal{K}) := \bigcap_{i=1}^{n}
\mathcal{ L}_{(s)}(H_{\alpha_{i}},K_{\gamma_{i}}) \subseteq
\mathcal{L}(H,K)$.

\paragraph{The relative tensor product}
Let $\frakb=\cbasesb$ be a $C^{*}$-base, $H_{\beta}$
$C^{*}$-$\frakb$-module, and $K_{\gamma}$ a
$C^{*}$-$\frakbo$-module. The {\em relative tensor product}
of $H_{\beta}$ and $K_{\gamma}$ is the Hilbert space
\begin{align*}
  \HtensorK:=\beta \tr \frakK \tl \gamma.
\end{align*}
It is spanned by elements $\xi \tr \zeta \tl \eta$, where
$\xi \in \beta$, $\zeta \in \frakK$, $\eta \in \gamma$, and
\begin{align*}
  \langle \xi \tr \zeta \tl
\eta|\xi' \tr \zeta' \tl \eta'\rangle = \langle \zeta |
\xi^{*}\xi' \eta^{*}\eta' \zeta'\rangle = \langle
\zeta|\eta^{*}\eta' \xi^{*}\xi' \zeta'\rangle 
\end{align*}
for all $\xi,\xi' \in \beta$, $\zeta,\zeta' \in \frakK$,
$\eta,\eta' \in \gamma$.  Obviously, there exists a unitary
{\em flip} 
\begin{align*}
  \Sigma \colon \HtensorK \to \KtensorH, \quad \xi \tr \zeta
  \tl \eta \mapsto \eta \tr \zeta \tl \xi.
\end{align*}
Using the unitaries in \eqref{eq:rtp-iso} on $H_{\beta}$ and
$K_{\gamma}$, respectively, we shall make the following
identifications without further notice:
\begin{gather*}
  H {_{\rho_{\beta}} \tl} \gamma \cong \HtensorK
  \cong \beta \tr_{\rho_{\gamma}} K, \quad \xi \zeta \tl \eta
  \equiv \xi \tr \zeta \tl \eta \equiv \xi \tr \eta\zeta.
\end{gather*}

For all $S \in \rho_{\beta}(\frakBo)'$ and $T \in
\rho_{\gamma}(\frakB)'$, we have
operators
\begin{align*}
  S \tl \Id &\in \mathcal{L}(H
  {_{\rho_{\beta}}\tl} \gamma) = \mathcal{L}(\HtensorK), &
  \Id \tr T &\in \mathcal{L}(\beta
  \tr_{\rho_{\gamma}} K) = \mathcal{L}(\HtensorK).
\end{align*}
If $S \in {\cal L}_{s}(H_{\beta})$ or $T \in {\cal
  L}_{s}(K_{\gamma})$, then $(S \tl \Id)(\xi
\tr \eta\zeta) = S\xi \tr
\eta\zeta$ or $(\Id \tr T)(\xi\zeta \tl
\eta) = \xi\zeta \tl T\eta$, respectively,
for all $\xi \in \beta$, $\zeta \in
\frakK$, $\eta \in \gamma$, so that we can define
\begin{multline*}
  S \btensor T :=(S \tl \Id) (\Id \tr T) = (\Id \tr T)(S \tl
  \Id) \in {\cal L}(\HtensorK)  \\
  \text{ for all } (S,T) \in \left({\cal L}_{s}(H_{\beta})
    \times \rho_{\gamma}(\frakB)'\right) \cup
  \left(\rho_{\beta}(\frakBo)' \times {\cal
      L}_{s}(K_{\gamma})\right).
\end{multline*}

For each $\xi \in \beta$ and $\eta \in \gamma$, there exist
bounded linear operators
\begin{align*}
  |\xi\rangle_{{1}} \colon K &\to \HtensorK, \ \omega
  \mapsto \xi \tr \omega, & 
  |\eta\rangle_{{2}} \colon H &\to \HtensorK, \ \omega
  \mapsto \omega \tl \eta,
\end{align*}
whose adjoints $ \langle \xi|_{1}:=|\xi\rangle_{1}^{*}$ and
$\langle\eta|_{{2}} := |\eta\rangle_{{2}}^{*}$ are given by
\begin{align*}
  \langle \xi|_{1}\colon \xi' \tr
  \omega &\mapsto \rho_{\gamma}(\xi^{*}\xi')\omega, &
  \langle\eta|_{{2}} \colon \omega
  \tl\eta' &\mapsto \rho_{\beta}( \eta^{*}\eta')\omega.
\end{align*}
We put $\kbeta{1} := \{ |\xi\rangle_{{1}} \,|\, \xi \in
\beta\} \subseteq {\cal L}(K,\HtensorK)$ and similarly
define $\bbeta{1}$, $\kgamma{2}$, $\bgamma{2}$.

Let $\mathcal{H}=(H,\alpha_{1},\ldots,\alpha_{m},\beta)$ be
a $C^{*}$-$(\fraka_{1},\ldots,\fraka_{m},\frakb)$-module and
$\mathcal{K}=(K,\gamma,\delta_{1},\ldots,\delta_{n})$ a
$C^{*}$-$(\frakbo,\frakc_{1},\ldots,\frakc_{n})$-module,
where $\fraka_{i}=(\frakH_{i},\frakA_{i},\frakAo_{i})$ and
$\frakc_{j}=(\frakL_{j},\frakC_{j},\frakCo_{j})$ are
$C^{*}$-bases for all $i,j$. We put
\begin{align*}
  \alpha_{i} \lt \gamma &:= [\kgamma{2}\alpha_{i}] \subseteq
  \mathcal{ L}(\frakH_{i}, \HtensorK), & \beta \rt
  \delta_{j} &:= [\kbeta{1}\delta_{j}] \subseteq \mathcal{
    L}(\frakL_{j}, \HtensorK)
\end{align*}
for all $i,j$.  Then $(\HtensorK,
\alpha_{1} \lt \gamma, \ldots, \alpha_{m} \lt \gamma, \beta
\rt \delta_{1},\ldots, \beta \rt \delta_{n})$ is a
$C^{*}$-$(\fraka_{1},\ldots,\fraka_{m}$, $
\frakc_{1},\ldots,\frakc_{n})$-module, called the {\em
  relative tensor product} of $\mathcal{H}$ and
$\mathcal{K}$ and denoted by $ \mathcal{H} \btensor
\mathcal{K}$. For all $i,j$ and  $a^{\dagger} \in \frakAo_{i}$, $c^{\dagger} \in \frakCo_{j}$,
\begin{align*}
  \rho_{(\alpha_{i} \lt \gamma)}(a^{\dagger}) &=
  \rho_{\alpha_{i}}(a^{\dagger}) \btensor \Id, &
  \rho_{(\beta \rt \delta_{j})}(c^{\dagger}) &= \Id \btensor
  \rho_{\delta_{j}}(c^{\dagger}).
\end{align*}
The relative tensor product has nice categorical properties:
\begin{description}
\item[Functoriality] Let $\tilde{\mathcal{H}}=(\tilde
  H,\tilde \alpha_{1},\ldots,\tilde \alpha_{m},\tilde
  \beta)$ be a
  $C^{*}$-$(\fraka_{1},\ldots,\fraka_{m},\frakb)$-module,
  $\tilde{\mathcal{K}}=(\tilde K,\tilde \gamma$, $\tilde
  \delta_{1},\ldots,\tilde \delta_{n})$ a
  $C^{*}$-$(\frakbo,\frakc_{1},\ldots,\frakc_{n})$-module,
  and $S \in
  \mathcal{L}_{(s)}(\mathcal{H},\tilde{\mathcal{H}})$, $T
  \in
  \mathcal{L}_{(s)}(\mathcal{K},\tilde{\mathcal{K}})$. Then
  there exists a unique operator $S \btensor T \in
  \mathcal{L}_{(s)}(\mathcal{H} \btensor \mathcal{K},
  \tilde{\mathcal{H}} \btensor \tilde{\mathcal{K}})$
  satisfying $(S \btensor T)(\xi \tr \zeta \tl \eta) = S\xi
  \tr \zeta \tl T\eta$ for all $\xi \in \beta, \, \zeta \in
  \frakK, \, \eta \in \gamma$. 
\item[Unitality] The triple
  $\mathcal{U}:=(\frakK,\frakBo,\frakB)$ is a
  $C^{*}$-$(\frakBo,\frakB)$-module and the maps
    \begin{align} \label{eq:rtp-unital}
      \begin{aligned}
        l_{{\cal H}} &\colon H
        \rtensor{\beta}{\frakb}{\frakBo} \frakK \to H, & \xi
        \tr \zeta \tl b^{\dagger} &\mapsto \xi
        b^{\dagger}\zeta, \\ r_{{\cal K}} &\colon \frakK \,
        \rtensor{\frakB}{\frakb}{\gamma} K \to K, & b \tr
        \zeta \tl \eta &\mapsto \eta b \zeta,
      \end{aligned}
  \end{align}
  are isomorphisms of
  $C^{*}$-$(\fraka_{1},\ldots,\fraka_{m},\frakb)$-modules
  and
  $C^{*}$-$(\frakbo,\frakc_{1},\ldots,\frakc_{n})$-modules
  $\mathcal{H} \btensor \mathcal{U} \to \mathcal{H}$ and
  $\mathcal{U} \btensor \mathcal{K} \to \mathcal{K}$,
  respectively, natural in $\mathcal{H}$ and $\mathcal{K}$.
\item[Associativity] Let
  $\frakd,\frake_{1},\ldots,\frake_{l}$ be $C^{*}$-bases,
  $\hat{\mathcal{K}}
  =(K,\gamma,\delta_{1},\ldots,\delta_{n},\epsilon)$ a
  $C^{*}$-$(\frakbo,\frakc_{1},\ldots,\frakc_{n}$,
  $\frakd)$-module and
  $\mathcal{L}=(L,\phi,\psi_{1},\ldots,\psi_{l})$ a
  $C^{*}$-$(\frakdo,\frake_{1},\ldots,\frake_{l})$-module. Then
  there exists a canonical isomorphism
  \begin{align} \label{eq:rtp-associative}
    a_{{\cal H},{\cal K},{\cal L}} \colon (\HtensorK) \rtensor{\beta \rt
    \epsilon}{\frakd}{\phi} L \to \beta
  \tr_{\rho_{\gamma}} K {_{\rho_{\epsilon}}\tl} \, \phi
  \to H \rtensor{\beta}{\frakb}{\gamma \lt \phi} (K
  \rtensor{\epsilon}{\frakd}{\phi} L)
\end{align}
which is an isomorphism of
$C^{*}$-$(\fraka_{1},\ldots,\fraka_{m},\frakc_{1},\ldots,\frakc_{n},\frake_{1},\ldots,\frake_{l})$-modules
$(\mathcal{H} \btensor \hat{\mathcal{K}}) \tensor{\frakd}
\mathcal{L} \to \mathcal{H} \btensor (\hat{\mathcal{K}}
\tensor{\frakd} \mathcal{L})$. From now on, we identify the
Hilbert spaces in \eqref{eq:rtp-associative} and denote them
by $\HtensorK \rtensor{\epsilon}{\frakd}{\phi} L$.
\item[Direct sums] Let $\fraka=\cbasesa$ and
  $\frakb=\cbasesb$ be $C^{*}$-bases and let $({\cal
    H}_{i})_{i}$ be a family of
  $C^{*}$-$(\fraka,\frakb)$-modules, where ${\cal H}_{i}=
  (H_{i},\alpha_{i},\beta_{i})$ for each $i$.  Denote by
  $\boxplus_{i} \alpha_{i} \subseteq {\cal L}\big(\frakH,
  \oplus_{i} H_{i}\big)$ the norm-closed linear span of all
  operators of the form $\zeta \mapsto (\xi_{i}\zeta)_{i}$,
  where $(\xi_{i})_{i}$ is contained in the algebraic direct
  sum $ \bigoplus^{\mathrm{alg}}_{i} \alpha_{j}$, and
  similarly define $\boxplus_{i} \beta_{i} \subseteq {\cal
    L}\big(\frakK, \oplus_{i} H_{i}\big)$. Then the triple
  $\boxplus_{i} {\cal H}_{i} := \big(\oplus_{i} H_{i},
  \boxplus_{i} \alpha_{i}, \boxplus_{i} \beta_{i}\big)$ is a
  $C^{*}$-$(\fraka,\frakb)$-module, and for each $j$, the
  canonical inclusions $\iota^{{\cal H}}_{j} \colon H_{j}
  \to \oplus_{i} H_{i}$ and projection $\pi^{{\cal H}}_{j}
  \colon \oplus_{i} H_{i} \to H_{j}$ are morphisms ${\cal
    H}_{j} \to \boxplus_{i} {\cal H}_{i}$ and $\boxplus_{i}
  {\cal H}_{i} \to {\cal H}_{j}$. With respect to these
  maps, $\boxplus_{i} {\cal H}_{i}$ is the direct sum of the
  family $({\cal H}_{i})_{i}$.

  Let $\frakc$ be a $C^{*}$-base and $({\cal K}_{j})_{j}$ a
  family of $C^{*}$-$(\frakbo,\frakc)$-modules, and define
  the direct sum $\boxplus_{j} {\cal K}_{j}$ similarly as
  above.  Then there exist inverse isomorphisms
  $\boxplus_{i,j} ( {\cal H}_{i} \btensor {\cal K}_{j})
  \leftrightarrows (\boxplus_{i} {\cal H}_{i}) \btensor
  (\boxplus_{j} {\cal K}_{j})$ given by
  $(\omega_{i,j})_{i,j} \mapsto \sum_{i,j} (\iota_{i}^{{\cal
      H}} \btensor \iota_{j}^{{\cal K}})(\omega_{i,j})$ and
  $\big((\pi_{i}^{{\cal H}} \btensor \pi_{j}^{{\cal
      K}})(\omega)\big)_{i,j} \mapsfrom \omega$,
  respectively.

  Similar constructions apply to
  $C^{*}$-$\frakb$-modules and $C^{*}$-$\frakbo$-modules.
\end{description}

\subsection{The definition of $C^{*}$-pseudo-multiplicative
  unitaries }

Using the relative tensor product of $C^{*}$-modules
introduced above, we generalize the notion of a
multiplicative unitary as follows.  Let $\frakb=\cbasesb$ be
a $C^{*}$-base, $(H,\hat\beta,\alpha,\beta)$ a
$C^{*}$-$(\frakbo,\frakb,\frakbo)$-module, and $V \colon
\Hsource \to \Hrange$ a unitary satisfying
  \begin{gather} \label{eq:pmu-intertwine}
    \begin{aligned} V(\alpha \lt \alpha) &= \alpha \rt
\alpha, & V(\hbeta \rt \beta) &= \hbeta \lt \beta, &
V(\hbeta \rt \hbeta) &= \alpha \rt \hbeta, & V(\beta \lt
\alpha) &= \beta \lt \beta
    \end{aligned}
  \end{gather}
  in ${\cal L}(\frakK,\Hrange)$. Then all operators in the
  following diagram are well defined,
  \begin{gather} \label{eq:pmu-pentagon} \smalldiagram
    \begin{gathered} \xymatrix@R=15pt@C=20pt{ {\Hone}
\ar[r]^{V \botensor \Id} \ar[d]^{\Id \botensor V} & {\Htwo}
\ar[r]^{\Id \btensor V} & {\Hthree,} \\ {\Hfive} \ar[d]^{\Id
\botensor \Sigma} & & {\Hfour} \ar[u]^{V \btensor \Id} \\
{\Hfourlt} \ar[rr]^{V \botensor \Id} && {\Hfourrt}
\ar[u]^{\Sigma_{23}} } \end{gathered}
  \end{gather} 
where $\Sigma_{23}$ denotes the isomorphism 
\begin{align*}
   \Hfourrt \cong (H {_{\rho_{\alpha}} \tl} \beta)
  {_{\rho_{(\hbeta \lt \beta)}} \tl} \alpha \mycong (H
  {_{\rho_{\hbeta}} \tl} \alpha) {_{\rho_{(\alpha \lt
        \alpha)}} \tl} \beta \cong \Hfour
\end{align*}
given by $(\zeta \tl \xi) \tl \eta \mapsto (\zeta \tl \eta)
\tl \xi$.  We adopt the leg notation \cite{baaj:2} and write
\begin{align*}
  V_{12} &\text{ for } V \botensor \Id,\, V \btensor \Id; &
  V_{23} &\text{ for } V \btensor \Id, \, V \botensor \Id; &
  V_{13} &\text{ for } \Sigma_{23}(V \botensor \Id)(\Id
  \botensor \Sigma).
\end{align*}
\begin{definition} \label{definition:pmu} A {\em
    $C^{*}$-pseudo-multiplicative unitary} is a tuple
  $(\frakb,H,\hbeta,\alpha,\beta,V)$ consisting of a
  $C^{*}$-base $\frakb$, a
  $C^{*}$-$(\frakbo,\frakb,\frakbo)$-module
  $(H,\hbeta,\alpha,\beta)$, and a unitary $V \colon
  \Hsource \to \Hrange$ such that equation
  \eqref{eq:pmu-intertwine} holds and diagram
  \eqref{eq:pmu-pentagon} commutes. We frequently call just
  $V$ a $C^{*}$-pseudo-multiplicative unitary.
\end{definition} 
This definition covers the following special cases:
\begin{remarks} \label{remarks:pmu-cases} Let
  $(\frakb,H,\hat\beta,\alpha,\beta,V)$ be a
  $C^{*}$-pseudo-multiplicative unitary.
  \begin{enumerate}
  \item If $\frakb$ is the trivial $C^{*}$-base
    $(\complex,\complex,\complex)$, then
    $\sHsource \cong H \otimes H \cong \sHrange$, and $V$ is
    a multiplicative unitary.
  \item If we consider $\rho_{\hbeta}$ and $\rho_{\beta}$ as
    representations $\rho_{\hbeta},\rho_{\beta}\colon \frakB
    \to {\cal L}(H_{\alpha}) \cong {\cal
      L}_{\frakB}(\alpha)$, then the map $\alpha
    {_{\rho_{\hbeta}} \tl}\, \alpha \cong \alpha \lt \alpha
    \to \alpha \rt \alpha \cong \alpha \tr_{\rho_{\beta}}
    \alpha$ given by $\omega \mapsto V\omega$ is a
    pseudo-multiplicative unitary on $C^{*}$-modules in the
    sense of \cite{timmermann:hopf}.
  \item Assume that $\frakb=\frakbo$; then $\frakB=\frakBo$
    is commutative.  If $\hbeta=\alpha$, then the
    pseudo-multiplicative unitary in ii) is a
    pseudo-multiplicative unitary in the sense of O'uchi
    \cite{ouchi}.  If additionally $\hbeta=\alpha=\beta$,
    then the unitary in ii) is a continuous field of
    multiplicative unitaries in the sense of Blanchard
    \cite{blanchard}.
  \item Assume that $\frakb$ is the $C^{*}$-base associated
    to a faithful proper KMS-weight $\mu$ on a
    $C^{*}$-algebra $B$ (see \cite[Example
    2.9]{timmer:fiber}).  Then $\mu$ extends to a n.s.f.\
    weight $\tilde \mu$ on $\wclose{\frakB}$, and with
    respect to the canonical isomorphisms $\Hsource \cong H
    \vtensor{\rho_{\hbeta}}{\tilde\mu^{op}}{\rho_{\alpha}}
    H$ and $\Hrange \cong H
    \vtensor{\rho_{\alpha}}{\tilde\mu}{\rho_{\beta}} H$ (see
    \cite[Corollary 2.21]{timmer:fiber}), $V$ is a
    pseudo-multiplicative unitary on Hilbert spaces in the
    sense of Vallin \cite{vallin:2}.
  \end{enumerate}
\end{remarks} 
Let us give some examples and easy constructions:
\begin{examples} \label{examples:pmu}
  \begin{enumerate}
  \item To every locally compact, Hausdorff, second
    countable groupoid with a left Haar system, we shall
    associate a $C^{*}$-pseudo-multiplicative unitary in the
    next subsection.
  \item In \cite{timmer:compact}, a
    $C^{*}$-pseudo-multiplicative unitary is associated to
    every compact $C^{*}$-quantum groupoid.
  \item The {\em opposite} of a
    $C^{*}$-pseudo-multiplicative unitary
    $(\frakb,H,\hat\beta,\alpha,\beta,V)$ is the tuple
    $(\frakb,H,\beta,\alpha,\hat\beta,V^{op})$, where
    $V^{op}$ denotes the composition $\Sigma V^{*} \Sigma
    \colon H \rtensor{\beta}{\frakbo}{\alpha} H
    \xrightarrow{\Sigma} \Hrange \xrightarrow{V^{*}}
    \Hsource \xrightarrow{\Sigma} H
    \rtensor{\alpha}{\frakb}{\hbeta} H$. A tedious but
    straightforward calculation shows that this is a
    $C^{*}$-pseudo-multiplicative unitary.
  \item The {\em direct sum} of a family
    $((\frakb_{i},H^{i},\hbeta_{i},\alpha_{i},\beta_{i},V_{i}))_{i}$
    of $C^{*}$-pseudo-multiplicative unitaries is defined as
    follows. Write
    $\frakb^{i}=(\frakH^{i},\frakB_{i},\frakBo_{i})$ for
    each $i$, put $\frakH:=\bigoplus_{i} \frakH^{i}$,
    $H:=\bigoplus_{i} H^{i}$, denote by
    $\frakB^{(\dagger)}:=\bigoplus_{i}
    \frakB^{(\dagger)}_{i} \subseteq {\cal L}(\frakH)$ the
    $c_{0}$-direct sum of $C^{*}$-algebras, and by
    $\hbeta:=\bigoplus_{i} \hbeta_{i}$,
    $\alpha:=\bigoplus_{i} \alpha_{i}$,
    $\beta:=\bigoplus_{i} \beta_{i} $ the $c_{0}$-direct sum
    in ${\cal L}(\frakH,H)$. Then
    $\frakb:=(\frakH,\frakB,\frakBo)$ is a $C^{*}$-base,
    there exist natural isomorphisms $\Hsource \cong
    \bigoplus_{i} H^{i}
    \rtensor{\hbeta_{i}}{\frakbo_{i}}{\alpha_{i}} H^{i}$ and
    $ \Hrange \cong \bigoplus_{i} H^{i}
    \rtensor{\alpha_{i}}{\frakbo_{i}}{\beta_{i}} H^{i}$
    \cite[Proposition 2.17]{timmer:fiber}, and
    if $V$ denotes the unitary corresponding to
    $\bigoplus_{i} V_{i}$ with respect to these
    isomorphisms, then the tuple
    $(\frakb,H,\hbeta,\alpha,\beta,V)$ is a
    $C^{*}$-pseudo-multiplicative unitary.
  \item The {\em tensor product} of
    $C^{*}$-pseudo-multiplicative unitaries
    $(\frakb,H,\hat\beta,\alpha,\beta,V)$ and
    $(\frakc,K,\hat\delta,\gamma,\delta,W)$ is defined as
    follows. Denote by
    $\frakB^{(\dagger)}\otimes\frakC^{(\dagger)} \subseteq
    {\cal L}(\frakH\otimes\frakK)$ and $\hbeta \otimes
    \hdelta,\alpha \otimes \gamma,\beta \otimes\delta
    \subseteq {\cal L}(\frakH\otimes\frakK,H\otimes K)$ the
    closed subspaces generated by elementary tensor
    products. Then $\frakb\otimes\frakc:=(\frakH\otimes
    \frakK, \frakB\otimes\frakC,\frakBo\otimes\frakCo)$ is a
    $C^{*}$-base, there exist natural isomorphisms
    $(H\otimes K) \rtensor{\hbeta \otimes
      \hdelta}{(\frakb\otimes\frakc)^{\dagger}}{\alpha
      \otimes \gamma} (H \otimes K) \cong (\Hsource) \otimes
    (K \rtensor{\hdelta}{\frakco}{\gamma} K)$ and $(H\otimes
    K) \rtensor{\alpha \otimes
      \gamma}{\frakb\otimes\frakc}{\beta \otimes \delta} (H
    \otimes K) \cong (\Hrange) \otimes (K
    \rtensor{\gamma}{\frakc}{\delta} K)$, and if $U$ denotes
    the unitary corresponding to $V \otimes W$ with respect
    to these isomorphisms, then $(\frakb\otimes\frakc,
    H\otimes K,
    \hbeta\otimes\hdelta,\alpha\otimes\gamma,\beta\otimes\delta,U)$
    is a $C^{*}$-pseudo-multiplicative unitary.
\end{enumerate}
\end{examples}

\subsection{The $C^{*}$-pseudo-multiplicative unitary of a
  groupoid }

\label{subsection:pmu-groupoid}

To every locally compact, Hausdorff, second countable
groupoid with left Haar system, we shall associate a
$C^{*}$-pseudo-multiplicative unitary.  The underlying
pseudo-multiplicative unitary was introduced by Vallin
\cite{vallin:2}, and associated unitaries on $C^{*}$-modules
were discussed in \cite{ouchi,timmermann:hopf}. We focus on
the aspects that are new in the present setting.

Let $G$ be a locally compact, Hausdorff, second countable
groupoid with left Haar system $\lambda$ and associated
right Haar system $\lambda^{-1}$, and let $\mu$ be a measure
on $G^{0}$ with full support. Define measures $\nu,\nu^{-1}$
on $G$ by
\begin{align*} \int_{G} f \intd \nu &:= \int_{G^{0}}
  \int_{G^{u}} f(x) \intd\lambda^{u}(x) \intd\mu(u), &
  \int_{G} f d\nu^{-1} &= \int_{G^{0}} \int_{G_{u}} f(x)
  \intd\lambda^{-1}_{u}(x) \intd\mu(u)
\end{align*}
 for all $f \in C_{c}(G)$. Thus, $\nu^{-1}=i_{*}\nu$, where
 $i\colon G \to G$ is given by $x \mapsto x^{-1}$.  We
 assume that $\mu$ is quasi-invariant in the sense that
 $\nu$ and $\nu^{-1}$ are equivalent, and denote by
 $D:=\intd\nu/\intd\nu^{-1}$ the Radon-Nikodym derivative.

 We identify functions in $C_{b}(G^{0})$ and $C_{b}(G)$ with
 multiplication operators on the Hilbert spaces
 $L^{2}(G^{0},\mu)$ and $L^{2}(G,\nu)$, respectively, and
 let
 \begin{align*}
   \frakK&:=L^{2}(G^{0},\mu), & \frakB&= \frakBo:=
   C_{0}(G^{0}) \subseteq {\cal L}(\frakK), &
   \frakb&:=(\frakK,\frakB,\frakBo), & H &:= L^{2}(G,\nu).
 \end{align*}
 Pulling functions on $G^{0}$ back to $G$ along $r$ or $s$,
 we obtain representations
 \begin{align*}
   r^{*} &\colon C_{0}(G^{0}) \to C_{b}(G)
     \hookrightarrow {\cal L}(H), &
     s^{*} &\colon C_{0}(G^{0}) \to C_{b}(G)
     \hookrightarrow {\cal L}(H).
   \end{align*}
   We define Hilbert $C^{*}$-modules $L^{2}(G,\lambda)$ and
   $L^{2}(G,\lambda^{-1})$ over $C_{0}(G^{0})$ as the
   respective completions of the pre-$C^{*}$-module
   $C_{c}(G)$, the structure maps being given by
 \begin{align*} 
   \langle \xi'|\xi\rangle(u)&= \int_{G^{u}}
   \overline{\xi'(x)}\xi(x) \intd\lambda^{u}(x), & \xi f &=
   r^{*}(f)\xi &  &\text{in the case of } L^{2}(G,\lambda), \\
   \langle \xi'|\xi\rangle(u)&= \int_{G_{u}}
   \overline{\xi'(x)}\xi(x) \intd\lambda^{-1}_{u}(x), & \xi
   f &= s^{*}(f)\xi & &\text{in the case of }
   L^{2}(G,\lambda^{-1})
\end{align*} 
respectively, for all $\xi,\xi' \in C_{c}(G)$, $u \in
G^{0}$, $f \in C_{0}(G^{0})$.

\begin{lemma} \label{proposition:groupoid-factorizations}
    There exist isometric embeddings
    \begin{align*}
      j&\colon L^{2}(G,\lambda) \to {\cal L}(\frakK,H), &
      \hat j &\colon L^{2}(G,\lambda^{-1}) \to {\cal
        L}\big(\frakK,H\big)
    \end{align*}
 such that for all $\xi \in
    C_{c}(G)$, $\zeta \in C_{c}(G^{0})$
  \begin{align*}
    \big(j(\xi) \zeta\big)(x) &= \xi(x)\zeta(r(x)), &
    \big(\hat j(\xi)\zeta\big)(x) &= \xi(x) D^{-1/2}(x)
    \zeta(s(x)).
  \end{align*}
\end{lemma}
\begin{proof}
Let $E:=L^{2}(G,\lambda)$, $\hat
  E:=L^{2}(G,\lambda^{-1})$, and  $\xi,\xi ' \in
  C_{c}(G)$, $\zeta,\zeta' \in C_{c}(G^{0})$.  Then
  \begin{align*} 
    \big\langle j(\xi')\zeta' \big| j(\xi)\zeta\big\rangle
    &= \int_{G^{0}} \int_{G^{u}}
    \overline{\xi'(x)\zeta'(r(x))} \xi(x) \zeta(r(x))
    \intd\lambda^{u}(x) \intd\mu(u) = \big\langle
    \zeta'\big|
    \langle\xi'|\xi\rangle_{E} \zeta\big\rangle,\\
    \big\langle \hat j(\xi')\zeta'\big|\hat j(\xi)\zeta
    \big\rangle &= \int_{G} \overline{\xi'(x)\zeta'(s(x))}
    \xi(x)\zeta(s(x))
    \underbrace{D^{-1}(x)\intd\nu(x)}_{=\intd\nu^{-1}(x)} \\
    &= \int_{G^{0}} \int_{G_{u}} \overline{\zeta'(u)}
    \overline{\xi'(x)}\xi(x)\zeta(u)
    \intd\lambda^{-1}_{u}(x) \intd\mu(u) = \big \langle
    \zeta' \big| \langle\xi'|\xi\rangle_{\hat E}
    \zeta\big\rangle. \qedhere
  \end{align*} 
\end{proof}
Let $\alpha:=\beta:=j(L^{2}(G,\lambda))$ and $ \hbeta :=
\hat j(L^{2}(G,\lambda^{-1}))$. Easy calculations show:
\begin{lemma} \label{lemma:pmu-groupoid-2}
  $(H,\hbeta,\alpha,\beta)$ is a
  $C^{*}$-$(\frakbo,\frakb,\frakbo)$-module,
  $\rho_{\alpha}=\rho_{\beta}=r^{*}$ and
  $\rho_{\hbeta}=s^{*}$, and $j$ and $\hat j$ are
  unitary maps of Hilbert $C^{*}$-modules over $C_{0}(G^{0})
  \cong \frakB$. \qed
\end{lemma}
The Hilbert spaces $\Hsource$ and $\Hrange$ can be described
as follows. We define measures $\nu^{2}_{s,r}$ on $\GsrG$
and $\nu^{2}_{r,r}$ on $\GrrG$ by
 \begin{align*}
   \int_{\GsrG} \!\! f\intd\nu^{2}_{s,r} &:=
   \int_{G^{0}} \int_{G^{u}} \int_{G^{s(x)}} f(x,y)
   \intd\lambda^{s(x)}(y) \intd\lambda^{u}(x) \intd\mu(u), \\
   \int_{\GrrG} \!\! g\intd\nu^{2}_{r,r} &:=
   \int_{G^{0}} \int_{G^{u}}\int_{G^{u}}
   g(x,y)\intd\lambda^{u}(y)\intd\lambda^{u}(x)\intd\mu(u)
 \end{align*} 
 for all $f \in C_{c}(\GsrG)$, $g\in
 C_{c}(\GrrG)$. Routine calculations show:
\begin{lemma}
  There exist unique isomorphisms
  \begin{align*} 
    \Phi_{\hbeta,\alpha} \colon \Hsource &\to
    L^{2}(\GsrG,\nu^{2}_{s,r}), & \Phi_{\alpha,\beta}
    \colon \Hrange &\to L^{2} (\GrrG,\nu^{2}_{r,r})
  \end{align*}
  such that for all $\eta,\xi \in C_{c}(G)$ and $\zeta \in
  C_{c}(G^{0})$,
\begin{align*}
  \Phi_{\hbeta,\alpha}\big(\hat j(\eta) \tr \zeta \tl
  j(\xi)\big)(x,y) &= \eta(x)D^{-1/2}(x) \zeta(s(x))
  \xi(y), \\ \Phi_{\alpha,\beta}\big(j(\eta) \tr \zeta \tl
  j(\xi)\big)(x,y) &= \eta(x) \zeta(r(x)) \xi(y). \qed
\end{align*}
\end{lemma} 
From now on, we identify $\Hsource$ with $L^{2}(\GsrG)$
and $\Hrange$ with $L^{2}(\GrrG)$ via
$\Phi_{\hbeta,\alpha}$ and $\Phi_{\alpha,\beta}$,
respectively, without further notice.
\begin{theorem} \label{theorem:pmu-groupoid} There exists a
  $C^{*}$-pseudo-multiplicative unitary $\pmu$ such that
  $(V\omega)(x,y) = \omega(x,x^{-1}y)$ for all $\omega \in
  C_{c}(\GsrG)$ and $(x,y) \in \GrrG$.
\end{theorem}
\begin{proof} 
  Straightforward calculations show that
  $(H,\hbeta,\alpha,\beta)$ is a
  $C^{*}$-$(\frakbo,\frakb,\frakbo)$-module.

  The homeomorphism $\GrrG \to \GsrG$, $(x,y)
  \mapsto (x,x^{-1}y)$, induces an isomorphism $V_{0} \colon
  C_{c}(\GsrG) \to C_{c}(\GrrG)$ such that
  $(V_{0} \omega)(x,y)=\omega(x,x^{-1}y)$ for all $\omega
  \in C_{c}(\GsrG)$ and $(x,y) \in \GrrG$.  Using
  left-invariance of $\lambda$, one finds that $V_{0}$
  extends to a unitary $V \colon \Hsource \cong
  L^{2}(\GsrG) \to L^{2}(\GrrG) \cong \Hrange$.

  We claim that $V$ is a $C^{*}$-pseudo-multiplicative
  unitary. First, we show that $V(\hbeta \rt \hbeta)=\alpha
  \rt \hbeta$.  For each $\xi,\xi' \in C_{c}(G)$, $\zeta \in
  C_{c}(G^{0})$, and $(x,y) \in \GsrG$,
  \begin{align*} 
    \big(V|\hat j(\xi)\rangle_{1}\hat j(\xi')\zeta\big)(x,y)
    &= \big(|\hat j(\xi)\rangle_{1}\hat j(\xi')\zeta
    \big)(x,x^{-1}y) \\ &= \xi(x)\xi'(x^{-1}y)
    D^{-1/2}(x)D^{-1/2}(x^{-1}y) \zeta(s(y)), \\ \big(
    |j(\xi)\rangle_{1}\hat j(\xi')\zeta\big)(x,y) &=
    \xi(x)\xi'(y) D^{-1/2}(y) \zeta(s(y)).
  \end{align*}
  Using standard approximation arguments and the fact that
  $D(x)D(x^{-1}y)=D(y)$ for $\nu^{2}_{r,r}$-almost all
  $(x,y) \in \GrrG$ (see \cite{hahn} or
  \cite[p. 89]{paterson}), we find that $V(\hbeta \rt \hbeta)
  = [T(C_{c}(\GrrG))] = \alpha \rt \hbeta$, where for
  each $\omega \in C_{c}(\GrrG)$, 
  \begin{align*}
    (T(\omega)\zeta)(x,y) &= \omega(x,y) D^{-1/2}(y)
    \zeta(s(y)) \quad \text{for all } \zeta \in
    C_{c}(G^{0}), \, (x,y) \in \GrrG.
  \end{align*}
  Similar calculations show that the remaining relations in
  \eqref{eq:pmu-intertwine} hold.

  Tedious but straightforward calculations show that diagram
  \eqref{eq:pmu-pentagon} commutes; see also
  \cite{vallin:2}. Therefore, $V$ is a
  $C^{*}$-pseudo-multiplicative unitary.
\end{proof}

\section{Hopf $C^{*}$-bimodules and the legs of a
  $C^{*}$-pseudo multiplicative unitary}

\label{section:legs}

To every regular multiplicative unitary $V$ on a Hilbert
space $H$, Baaj and Skandalis associate two Hopf
$C^{*}$-algebras $(\hA_{V},\hDelta_{V})$ and
$(A_{V},\Delta_{V})$ as follows \cite{baaj:2}.  They show
for every multiplicative unitary $V$, the subspaces
$\hA^{0}_{V}$ and $A^{0}_{V}$ of ${\cal L}(H)$ defined by
\begin{align} \label{eq:legs-intro-algebras}
  \hA^{0}_{V} &:= \{ (\Id \vntimes \omega)(V) \mid \omega
  \in {\cal L}(H)_{*}\}, & A^{0}_{V} &:= \{ (\upsilon
  \vntimes \Id)(V) \mid \upsilon \in {\cal L}(H)_{*}\}
\end{align}
are closed under multiplication. In the regular case,
their norm closures $\hA_{V}$ and $A_{V}$, respectively, are
$C^{*}$-algebras, and
the $*$-homomorphisms $\hDelta_{V} \colon \hA_{V} \to {\cal
  L}(H \otimes H)$ and $\Delta_{V} \colon A_{V} \to {\cal
  L}(H \otimes H)$ given by
\begin{align} \label{eq:legs-intro-comultiplication}
  \hDelta_{V} \colon \ha &\mapsto V^{*}(1 \otimes \ha), &
  \Delta_{V} \colon a &\mapsto V(a \otimes 1)V^{*},
\end{align}
map $\hA_{V}$ to $M(\hA_{V} \otimes \hA_{V}) \subseteq {\cal
  L}(H \otimes H)$ and $A_{V}$ to $M(A_{V} \otimes A_{V})
\subseteq {\cal L}(H \otimes H)$, respectively, and form
comultiplications on $\hA_{V}$ and $A_{V}$. Finally, there
exists a perfect pairing
\begin{align} \label{eq:legs-intro-pairing}
  \hA^{0}_{V} \times A^{0}_{V} \to \complex, \quad \left((\Id
  \vntimes \omega)(V), (\upsilon \vntimes \Id)(V)\right) \to
  (\upsilon \vntimes \omega)(V),
\end{align}
which expresses the duality between $(\hA_{V},\hDelta_{V})$
and $(A_{V},\Delta_{V})$.

Applied to the multiplicative unitary of a locally compact
group $G$, this construction yields the $C^{*}$-algebras
$C_{0}(G)$ and $C^{*}_{r}(G)$ with the comultiplications
$\hDelta \colon C_{0}(G) \to M(C_{0}(G) \otimes C_{0}(G))
\cong C_{b}(G \times G)$ and $\Delta  \colon C^{*}_{r}(G)
\to M(C^{*}_{r}(G) \otimes C^{*}_{r}(G))$  given by
\begin{align} \label{eq:legs-group-comultiplication}
  \hDelta(f)(x,y) &= f(xy) \text{ for all }   f \in
  C_{0}(G), &
  \Delta(U_{x}) &= U_{x} \otimes U_{x} \text{ for all } x
  \in G,
\end{align}
where $U \colon G \to M(C^{*}_{r}(G)), x \mapsto U_{x}$, is
the canonical embedding.

To adapt these constructions to
$C^{*}$-pseudo-multiplicative unitaries, we have to solve
the following problems.

First, we have to find substitutes for the space ${\cal
  L}(H)_{*}$ and the slice maps $\Id \vntimes \omega$ and
$\upsilon \vntimes \Id$ that were used in the definition of
$\hA^{0}_{V}$ and $A^{0}_{V}$. It turns out that for a
$C^{*}$-pseudo-multiplicative unitary, the norm closures
$\hA_{V}$ and $A_{V}$ are easier to define. Therefore, we
first study these algebras, before we introduce the spaces
$\hA^{0}_{V}$ and $A^{0}_{V}$ and the dual pairing. The
spaces $\hA^{0}_{V}$ and $A^{0}_{V}$ carry the structure of
Banach algebras and can be considered as a Fourier algebra
and a dual Fourier algebra for $V$.

The main difficulty, however, is to find a suitable
generalization of the notion of a Hopf $C^{*}$-algebra and,
more precisely, to describe the targets of the
comultiplications $\hDelta_{V}$ and $\Delta_{V}$. For
example, if we replace the multiplicative unitary of a group
$G$ by the $C^{*}$-pseudo-multiplicative unitary of a
groupoid $G$, we expect to obtain the $C^{*}$-algebras
$\hA_{V}=C_{0}(G)$ and $A_{V}=C^{*}_{r}(G)$ with
$*$-homomorphisms $\hDelta$ and $\Delta$ given by the same
formulas as in \eqref{eq:legs-group-comultiplication}.  Then
the target of $\hDelta$ would be $M(C_{0}(\GsrG))$, and
$C_{0}(\GsrG)$ can be identified with the relative tensor
product $C_{0}(G) \rtensor{s^{*}}{C_{0}(G^{0})}{r^{*}}
C_{0}(G)$ of $C_{0}(G^{0})$-algebras 
\cite{blanchard}. But the target of $\Delta$ can not be
described in a similar way, and in general, we need to
replace the balanced tensor product by a fiber product
relative to some base algebra which may even   be
non-commutative. If $G$ is finite, then $C(G)$ and
$C^{*}_{r}(G)=\complex G$ can be considered as modules over
$R=C(G^{0})$ in several ways, and the targets of $\hDelta$
and $\Delta$ can be written using the $\times_{R}$-product
of Takeuchi \cite{takeuchi} in the form $C(G) \times_{R}
C(G)$ and $\complex G \times_{R} \complex G$,
respectively. In the setting of von Neumann algebras, the
targets of the comultiplications can be described using
Sauvageot's fiber product \cite{sauvageot:2,vallin:1}.  For
the setting of $C^{*}$-algebras, a partial solution was
proposed in \cite{timmermann:hopf}, and a general fiber
product construction which suits our purpose was introduced
in \cite{timmer:fiber}.

We proceed as follows. First, we recall the fiber product of
$C^{*}$-algebras over $C^{*}$-bases and systematically study
slice maps and related constructions. These prerequisites
are then used to define Hopf $C^{*}$-bimodules and
associated convolution algebras. Finally, we adapt the
constructions of Baaj and Skandalis to
$C^{*}$-pseudo-multiplicative unitaries and apply them to
the unitary associated to a groupoid.

Throughout this section, let $\frakb=\cbasesb$
be a $C^{*}$-base, $\pmuspace$ a
$C^{*}$-$(\frakbo,\frakb,\frakbo)$-module and $\pmuoperator$
a $C^{*}$-pseudo-multiplicative unitary.

\subsection{The fiber product of $C^{*}$-algebras over $C^{*}$-bases }
\label{subsection:legs-fiber}

In this subsection, we recall the fiber product of
$C^{*}$-algebras over $C^{*}$-bases 
\cite{timmer:fiber}, introduce several new notions of a
morphism of such $C^{*}$-algebras, and show that the fiber
product is also functorial with respect to these generalized
morphisms.  For additional motivation and details, we refer
to \cite{timmer:fiber}; two examples can be found in
subsection \ref{subsection:legs-groupoid}.

Let $\frakb_{1},\ldots,\frakb_{n}$ be $C^{*}$-bases, where
$\frakb_{i}=(\frakK_{i},\frakB_{i},\frakBo_{i})$ for each
$i$.  A {\em (nondegenerate)
  $C^{*}$-$(\frakb_{1},\ldots,\frakb_{n})$-algebra} consists
of a $C^{*}$-$(\frakb_{1},\ldots,\frakb_{n})$-mod\-ule
$(H,\alpha_{1},\ldots,\alpha_{n})$ and a (nondegenerate)
$C^{*}$-algebra $A \subseteq \mathcal{L}(H)$ such that
$\rho_{\alpha_{i}}(\frakBo_{i})A$ is contained in $A$ for
each $i$. We shall only be interested in the cases $n=1,2$,
where we abbreviate $A^{\alpha}_{H}:=(H_{\alpha},A)$,
$A^{\alpha,\beta}_{H}:=(\aHb,A)$.  

We need several natural notions of a morphism. Let ${\cal
  A}=({\cal H},A)$ and ${\cal C}=({\cal K},C)$ be
$C^{*}$-$(\frakb_{1},\ldots,\frakb_{n})$-algebras, where
${\cal H}=(H,\alpha_{1},\ldots,\alpha_{n})$ and ${\cal
  K}=(K,\gamma_{1},\ldots,\gamma_{n})$.  A
$*$-homo\-morphism $\pi \colon A \to C$ is called a {\em
  (semi-)normal morphism} or {\em jointly (semi-)normal
  morphism} from ${\cal A}$ to ${\cal C}$ if $[{\cal
  L}^{\pi}_{(s)}(H_{\alpha_{i}},K_{\gamma_{i}})\alpha_{i}] =
\gamma_{i}$ or $[{\cal L}^{\pi}_{(s)}({\cal H},{\cal
  K})\alpha_{i}] = \gamma_{i}$, respectively, for each $i$,
where
\begin{align*}
  {\cal L}^{\pi}_{(s)}(H_{\alpha_{i}},K_{\gamma_{i}}) &=
  {\cal L}^{\pi}(H,K) \cap {\cal
    L}_{(s)}(H_{\alpha_{i}},K_{\gamma_{i}}), & {\cal
    L}^{\pi}_{(s)}({\cal H},{\cal K}) &= {\cal L}^{\pi}(H,K)
  \cap {\cal L}_{(s)}({\cal H},{\cal K}).
\end{align*}
 One easily verifies that every semi-normal
morphism $\pi$ between $C^{*}$-$\frakb$-algebras
$A_{H}^{\alpha}$ and $C_{K}^{\gamma}$ satisfies
$\pi(\rho_{\alpha}(b^{\dagger}))=\rho_{\gamma}(b^{\dagger})$
for all $b^{\dagger} \in \frakBo$.

We construct a fiber product of $C^{*}$-algebras over 
$C^{*}$-bases as follows.
Let $\frakb$ be a $C^{*}$-base, $A_{H}^{\beta}$  a
$C^{*}$-$\frakb$-algebra, and  $B_{K}^{\gamma}$ a
$C^{*}$-$\frakbo$-algebra. The {\em fiber product} of
$A_{H}^{\beta}$ and $B_{K}^{\gamma}$ is the $C^{*}$-algebra
\begin{align*}
  A \fibre{\beta}{\frakb}{\gamma} B := \big\{ x \in
    \mathcal{L}(\HtensorK) \, \big| \, &x\kbeta{1},
    x^{*}\kbeta{1} \subseteq [ \kbeta{1} B] \text{ as
      subsets of } {\cal L}(K,\HtensorK),  \\ 
    & 
    x\kgamma{2}, x^{*}\kgamma{2} \subseteq
    [\kgamma{2}A] \text{ as subsets of } {\cal
    L}(H,\HtensorK)\big\}.
\end{align*}
If $A$ and $B$ are unital, so is $A
\fibre{\beta}{\frakb}{\gamma} B$, but otherwise, $A
\fibre{\beta}{\frakb}{\gamma} B$ may be degenerate.
Clearly, conjugation by the flip $\Sigma \colon \HtensorK
\to \KtensorH$ yields an isomorphism
  \begin{align*}
    \Ad_{\Sigma} \colon A \fibre{\beta}{\frakb}{\gamma} B
    \to B \fibre{\gamma}{\frakbo}{\beta} A.
  \end{align*}
    If $\fraka,\frakc$ are $C^{*}$-bases,
  $A^{\alpha,\beta}_{H}$ is a
  $C^{*}$-$(\fraka,\frakb)$-algebra and
  $B^{\gamma,\delta}_{K}$ a
  $C^{*}$-$(\frakbo,\frakc)$-algebra, then
\begin{align*}
  A^{\alpha,\beta}_{H} \bfibre B^{\gamma,\delta}_{K}
= ({_{\alpha}H_{\beta}} \btensor {_{\gamma}K_{\delta}},\, A
\fibre{\beta}{\frakb}{\gamma} B)
\end{align*}
is a $C^{*}$-$(\fraka,\frakc)$-algebra, called the {\em fiber
  product} of $A^{\alpha,\beta}_{H}$ and
$B^{\gamma,\delta}_{K}$.

The fiber product construction is functorial with respect to
normal morphisms \cite[Theorem 3.23]{timmer:fiber}, but also
with respect to (jointly) semi-normal morphisms. For the
proof, we slightly modify \cite[Lemma 3.22]{timmer:fiber}.
\begin{lemma}
 \label{lemma:fp-c-morphism}
 Let $\pi$ be a semi-normal morphism of
 $C^{*}$-$\frakb$-algebras $A_{H}^{\beta}$ and
 $C^{\lambda}_{L}$, let $B_{K}^{\gamma}$ be a
 $C^{*}$-$\frakbo$-algebra, and let $I:={\cal L}^{\pi}(H,L)
 \btensor \Id$.
 \begin{enumerate}
 \item $II^{*}I \subseteq I$ and there exists a unique
   $*$-homomorphism $\rho_{I} \colon (I^{*}I)' \to
   (II^{*})'$ such that $\rho_{I}(x)y=yx$ for all $x \in
   (I^{*}I)'$ and $y \in I$.
 \item There exists a linear contraction $j_{\pi}$ from the
   subspace $[\kgamma{2}A] \subseteq {\cal L}(H,\HtensorK)$
   to $[\kgamma{2}C] \subseteq {\cal L}(L, L
   \rtensor{\lambda}{\frakb}{\gamma} K)$ given by
   $|\eta\rangle_{2}a \mapsto |\eta\rangle_{2}\pi(a)$.
 \item $\rho_{I}(\AfibreB) \subseteq \Cl \bfibre \gB$ and
   $\rho_{I}(x)|\eta\rangle_{2} =
   j_{\pi}(x|\eta\rangle_{2})$ for all $x \in \AfibreB$,
   $\eta \in \gamma$.
 \end{enumerate}
\end{lemma}
\begin{proof}
  i) The assertion on $I$ is evident and the assertion on
  $\rho_{I}$ is \cite[Proposition 2.2]{timmer:fiber}.
  
  ii) The existence of $j_{\pi}$ follows from the fact that
  we have
  $(|\eta\rangle_{2}\pi(a))^{*}(|\eta'\rangle_{2}\pi(a'))
  = \pi(a)^{*}\rho_{\lambda}(\eta^{*}\eta')\pi(a') =
  \pi\big((|\eta\rangle_{2}a)^{*}(|\eta'\rangle_{2}a')\big)$
  for all $\eta,\eta' \in \gamma$, $a,a' \in A$.

  iii) Let $J:={\cal L}_{s}^{\pi}(H_{\beta},L_{\lambda})$.
  First, we have  $\rho_{I}(\AfibreB)\kgamma{2} \subseteq
  [\kgamma{2}C]$ because $\rho_{I}(x)|\eta\rangle_{2} =
  j_{\pi}(x|\eta\rangle_{2}) \in [\kgamma{2}C]$ for all $x
  \in \AfibreB$, $\eta \in \gamma$. Indeed, for all $S \in
  J$,
  \begin{align*}
  \rho_{I}(x)|\eta\rangle_{2}S =
  \rho_{I}(x)(S \btensor \Id)|\eta\rangle_{2} = (S \btensor
  \Id) x|\eta\rangle_{2} = j_{\pi}(x|\eta\rangle_{2})S.
  \end{align*}
  Second, $\rho_{I}(\AfibreB)|\lambda\rangle_{1} =
  \rho_{I}(\AfibreB)[(J\btensor \Id)\kbeta{1}] \subseteq [(J
  \btensor \Id)(\AfibreB) \kbeta{1}] \subseteq [(J \btensor
  \Id) \kbeta{1}B] \subseteq [|\lambda\rangle_{1}B]$.
\end{proof}
\begin{theorem}
  Let $\fraka,\frakb,\frakc$ be $C^{*}$-bases, $\phi$ a
  semi-normal morphism of $C^{*}$-$(\fraka,\frakb)$-algebras
  ${\cal A}=A^{\alpha,\beta}_{H}$ and ${\cal
    C}=C^{\kappa,\lambda}_{L}$, and $\psi$ a semi-normal
  morphism of $C^{*}$-$(\frakbo,\frakc)$-algebras ${\cal
    B}=B_{K}^{\gamma,\delta}$ and ${\cal
    D}=D^{\mu,\nu}_{M}$. Then there exists a unique
  semi-normal morphism of $C^{*}$-$(\fraka,\frakc)$-algebras
  $\phi \bfibre \psi \colon {\cal A} \bfibre {\cal B} \to {\cal
    C} \bfibre {\cal D}$ such that
  \begin{align} \label{eq:fp-c-morphism} 
    (\phi \bfibre
    \psi)(x)R = R x \quad \text{for all } x \in \AfibreB
    \text{ and }R \in I_{M}J_{H} + J_{L}I_{K},
  \end{align}
  where $I_{X}={\cal L}^{\phi}(H,L) \btensor \Id_{X}$,
  $J_{Y}= \Id_{Y} \btensor {\cal L}^{\psi}(K,M)$ for
  $X\in\{K,M\},Y\in\{H,L\}$.  If both $\phi$ and $\psi$ are
  normal, jointly semi-normal or jointly normal, then also $\phi
  \bfibre \psi$ is normal, jointly semi-normal or jointly
  normal, respectively.
\end{theorem}
\begin{proof}
  This follows from Lemma \ref{lemma:fp-c-morphism} and
  a similar argument as in the proof of \cite[Theorem
  3.13]{timmer:fiber}.
\end{proof}

Unfortunately, the fiber product need not be associative,
but in our applications, it will only appear as the target
of a comultiplication whose coassociativity will compensate
the non-associativity of the fiber product.

\subsection{Spaces of maps on $C^{*}$-algebras over
  $C^{*}$-bases }
\label{subsection:legs-slice}

To define convolution algebras of Hopf $C^{*}$-bimodules and
generalized Fourier algebras of
$C^{*}$-pseudo-multiplicative unitaries, we need to consider
several spaces of maps on $C^{*}$-algebras over
$C^{*}$-bases.

Let $\fraka=\cbasesa$ and $\frakb=\cbasesb$ be
$C^{*}$-bases, $H$ a Hilbert space, $H_{\alpha}$ a
$C^{*}$-$\fraka$-module, $H_{\beta}$ a
$C^{*}$-$\frakb$-module, and $A \subseteq {\cal L}(H)$ a
$C^{*}$-algebra. We denote by $\alpha^{\infty}$ the space of
all sequences $\eta=(\eta_{k})_{k \in \naturals}$ in
$\alpha$ for which the sum $\sum_{k} \eta_{k}^{*}\eta_{k}$
converges in norm, and put $\|\eta\|:=\|\sum_{k}
\eta_{k}^{*}\eta_{k}\|^{1/2}$ for each $\eta \in
\alpha^{\infty}$.  Similarly, we define
$\beta^{\infty}$. Then standard arguments show that for all
$\eta \in \beta^{\infty}, \eta' \in \alpha^{\infty}$, there
exists a bounded linear map
  \begin{align*}
    \omega_{\eta,\eta'} \colon A \to {\cal
      L}(\frakH,\frakK), \quad T \mapsto \sum_{k \in
      \naturals} \eta_{k}^{*}T\eta'_{k},
  \end{align*}
  where the sum converges in norm and
  $\|\omega_{\eta,\eta'}\|\leq \|\eta\|\|\eta'\|$. We  put
  \begin{align*}
    \Omega_{\beta,\alpha}(A) &:= \{ \omega_{\eta,\eta'} \mid
    \eta \in \beta^{\infty}, \eta' \in \alpha^{\infty} \}
    \subseteq L(A,{\cal L}(\frakH,\frakK),
  \end{align*}
  where $L(A,{\cal L}(\frakH,\frakK))$ denotes the space of
  bounded linear maps from $A$ to ${\cal L}(\frakH,\frakK)$.
  If $\beta=\alpha$, we abbreviate
  $\Omega_{\beta}(A):=\Omega_{\beta,\alpha}(A)$.  It is easy
  to see that $\Omega_{\beta,\alpha}(A)$ is a subspace of
  $L(A,{\cal L}(\frakH,\frakK))$ and that the following
  formula defines a norm on $\Omega_{\beta,\alpha}(A)$:
\begin{align*}
   \|\omega\| &:= \inf \big\{ \|\eta\|\|\eta'\| \big| \eta
    \in \beta^{\infty},\eta' \in \alpha^{\infty},\,
    \omega=\omega_{\eta,\eta'} \big\}  \text{ for all }
    \omega \in \Omega_{\beta,\alpha}(A).
\end{align*}
\begin{lemma}
  $\Omega_{\beta,\alpha}(A)$ is a Banach space.
\end{lemma}
\begin{proof}
  Let $(\omega^{k})_{k}$ be a sequence in
  $\Omega_{\beta,\alpha}(A)$ such that $\|\omega^{k}\| \leq
  4^{-k}$ for all $k \in \naturals$. We show that the sum
  $\sum_{k} \omega^{k}$ converges in norm in
  $\Omega_{\beta,\alpha}(A)$. For each $k \in \naturals$,
  we can choose $\eta^{k}\in \beta^{\infty}$ and
  $\eta'{}^{k} \in \alpha^{\infty}$ such that
  $\omega^{k}=\omega_{\eta^{k},\eta'{}^{k}}$ and
  $\|\eta^{k}\|\|\eta'{}^{k}\| \leq 4^{1-k}$. Without loss
  of generality, we may assume $\|\eta^{k}\| \leq 2^{1-k}$
  and $\|\eta'{}^{k}\|\leq 2^{1-k}$. Choose a bijection $i
  \colon \naturals\times\naturals \to \naturals$ and let
  $\eta_{i(k,n)} = \eta^{k}_{n}$ and $\eta'_{i(k,n)} =
  \eta'{}^{k}_{n}$ for all $k,n \in \naturals$.  Routine
  calculations show that $\eta \in \beta^{\infty}$, $\eta'
  \in \alpha^{\infty}$, and that the sum $\sum_{k}
  \omega^{k}$ converges in norm to
  $\omega_{\eta,\eta'} \in
  \Omega_{\beta,\alpha}(A)$.
\end{proof}
We have the following straightforward result:
\begin{proposition} \label{proposition:slice-star}
  There exists a linear isometry $\Omega_{\beta,\alpha}(A)
  \to \Omega_{\alpha,\beta}(A)$, $\omega \mapsto
  \omega^{*}$, such that $\omega^{*}(a)= \omega(a^{*})^{*}$
  for all $a \in A$ and $(\omega_{\eta,\eta'})^{*} =
  \omega_{\eta',\eta}$ for all $\eta \in \beta^{\infty},\eta'\in
  \alpha^{\infty}$. \qed
\end{proposition}
We can pull back maps of the form considered above via
morphisms as follows:
\begin{proposition} \label{proposition:slice-pullback}
  \begin{enumerate}
  \item Let $\pi$ be a normal morphism of
    $C^{*}$-$\frakb$-algebras $A_{H}^{\alpha}$ and $B_{K}^{
      \gamma}$. Then there exists a linear contraction
    $\pi^{*} \colon \Omega_{\gamma}(B) \to
    \Omega_{\alpha}(A)$ given by $\omega \mapsto \omega
    \circ \pi$.
  \item Let $\pi$ be a jointly normal morphism of
    $C^{*}$-$(\fraka,\frakb)$-algebras
    $A_{H}^{\alpha,\beta}$ and $B_{K}^{
      \gamma,\delta}$. Then there exists a linear
    contraction $\pi^{*} \colon \Omega_{\delta,\gamma}(B)
    \to \Omega_{\beta,\alpha}(A)$ given by $\omega \mapsto
    \omega \circ \pi$. 
  \end{enumerate}
\end{proposition}
\begin{proof}
  We only prove ii), the proof of i) is similar.  Let
  $I:={\cal L}^{\pi}(\aHb,\cKd)$ and
  $\eta \in \delta^{\infty}$, $\eta' \in
  \gamma^{\infty}$. Then there exists a closed separable
  subspace $I_{0} \subseteq I$ such that $\eta_{n} \in
  [I_{0}\beta]$ and $\eta'_{n} \in [I_{0}\alpha]$ for all $n
  \in \naturals$. We may also assume that
  $I_{0}I^{*}_{0}I_{0} \subseteq I_{0}$, and then
  $[I_{0}I_{0}^{*}]$ is a $\sigma$-unital $C^{*}$-algebra
  and has a bounded sequential approximate unit
  $(u_{k})_{k}$ of the form $u_{k}=\sum_{l=1}^{k}
  T_{l}T_{l}^{*}$, where $(T_{l})_{l}$ is a sequence in
  $I_{0}$ \cite[Proposition 6.7]{lance}.  We choose a
  bijection $i \colon \naturals \times \naturals \to
  \naturals$ and let $\xi_{i(l,n)} := T_{l}^{*}\eta_{n} \in
  \beta$ and $\xi'_{i(l,n)} := T_{l}^{*}\eta'_{n} \in
  \alpha$ for all $l,n \in \naturals$. Then the sum
  $\sum_{l} \xi_{i(l,n)}^{*}\xi_{i(l,n)} = \sum_{l}
  \eta_{n}^{*} T_{l}T_{l}^{*} \eta_{n}$ converges to
  $\eta_{n}^{*}\eta_{n}$ for each $n \in \naturals$ in norm
  because $\eta_{n} \in [I_{0}\beta]$. Therefore, $\xi \in
  \beta^{\infty}$ and $\|\xi\|=\|\eta\|$, and a similar
  argument shows that $\xi' \in \alpha^{\infty}$ and
  $\|\xi'\|=\|\eta'\|$. Finally,
  \begin{align*}
    \omega_{\xi,\xi'}(a) = \sum_{l,n} \eta_{n}^{*}
    T_{l} aT_{l}^{*}\eta'_{n} = \sum_{l,n}
    \eta_{n}^{*}\pi(a)T_{l}T_{l}^{*}\eta'_{n} = \sum_{n}
    \eta_{n}^{*} \pi(a) \eta'_{n} =
    \omega_{\eta,\eta'}(\pi(a))
  \end{align*}
  for each $a \in A$, where the sum converges in norm, and
  hence $\omega_{\eta,\eta'} \circ \pi =\omega_{\xi,\xi'}
  \in \Omega_{\beta,\alpha}(A)$ and $\|\omega_{\eta,\eta'}
  \circ \pi\| \leq \|\xi\|\|\xi'\| =
  \|\eta\|\|\eta'\|$.
\end{proof}
For each map of the form considered above, we can form a
slice map as follows.
\begin{proposition} \label{proposition:slice-slice}
 Let $A^{\beta}_{H}$ be a $C^{*}$-$\frakb$-algebra and
$B^{\gamma}_{K}$ a $C^{*}$-$\frakbo$-algebra. 
\begin{enumerate}
\item There exist linear contractions
  \begin{align*}
    \Omega_{\beta}(A) &\to \Omega_{\kbeta{1}}(\AfibreB),
    \ \phi \mapsto
    \phi \ast \Id, &
    \Omega_{\gamma}(B) &\to \Omega_{\kgamma{2}} (\AfibreB),
    \ \psi \mapsto \Id \ast \psi,
  \end{align*}
   such that  for all $\xi,\xi' \in \beta^{\infty}$ and
   $\eta,\eta' \in \gamma^{\infty}$,
   \begin{align*}
     \omega_{\xi,\xi'} \ast \Id &=\omega_{\tilde \xi,\tilde
       \xi'}, \text{ where } \tilde
     \xi_{n}=|\xi_{n}\rangle_{1}, \, \tilde
     \xi'_{n}=|\xi'_{n}\rangle_{1} \text{ for all }n \in
     \naturals, \\
     \Id \ast \omega_{\eta,\eta'} &= \omega_{\tilde
       \eta,\tilde \eta'}, \text{ where } \tilde \eta_{n} =
     |\eta_{n}\rangle_{2}, \tilde \eta'_{n} =
     |\eta'_{n}\rangle_{2} \text{ for all } n \in \naturals.
   \end{align*}
 \item We have $\psi \circ (\phi \ast \Id) = \phi \circ (\Id
   \ast \psi)$ for all $\phi \in \Omega_{\beta}(A)$ and $\psi
   \in \Omega_{\gamma}(B)$.
\end{enumerate}
\end{proposition}
\begin{proof}
  i), ii) Straightforward; see \cite[Proposition
  3.30]{timmer:fiber}.
\end{proof}
Finally, we need to consider fiber products of the linear
maps considered above.  We denote by ``$\hat\otimes$'' the
projective tensor product of Banach spaces.
\begin{proposition} \label{proposition:slice-fiber}
  Let $A^{\alpha,\beta}_{H}$ be a
  $C^{*}$-$(\fraka,\frakb)$-algebra and
  $B^{\gamma,\delta}_{K}$ a
  $C^{*}$-$(\frakbo,\frakc)$-algebra. 
  \begin{enumerate}
  \item There exist linear contractions
    \begin{align*}
      \Omega_{\alpha}(A) \hat\otimes \Omega_{\gamma}(B) &\to
      \Omega_{(\alpha \lt \gamma)}(\AfibreB), \
      \omega \otimes \omega' \mapsto \omega \boxtimes
      \omega':= \omega \circ (\Id \ast \omega'), \\
      \Omega_{\beta}(A) \hat\otimes
      \Omega_{\delta}(B) &\to 
      \Omega_{(\beta \rt \delta)}(\AfibreB), \
      \omega \otimes \omega' \mapsto \omega \boxtimes
      \omega' := \omega' \circ (\omega \ast \Id). 
    \end{align*}
  \item There exist linear contractions
    \begin{align*}
      \Omega_{\alpha,\beta}(A) \hat\otimes
      \Omega_{\gamma,\delta}(B) \to \Omega_{(\alpha \lt
        \gamma), (\beta \rt \delta)}(\AfibreB), \ \omega
      \otimes \omega' \mapsto \omega \boxtimes \omega',  \\
      \Omega_{\beta,\alpha}(A) \hat\otimes
      \Omega_{\delta,\beta}(B) \to \Omega_{(\beta \rt
        \delta),(\alpha \lt \gamma)}(\AfibreB), \ \omega
      \otimes \omega' \mapsto \omega \boxtimes \omega',
    \end{align*}
    such that for all $\xi \in \alpha^{\infty}$, $\xi' \in
    \beta^{\infty}$, $\eta\in \gamma^{\infty}$, $\eta' \in
    \delta^{\infty}$ and each bijection $i \colon
    \naturals\times \naturals \to \naturals$, we have
    $\omega_{\xi,\xi'} \boxtimes \omega_{\eta,\eta'} =
    \omega_{\theta,\theta'}$ and $ \omega_{\xi',\xi}
    \boxtimes \omega_{\eta',\eta} = \omega_{\theta',\theta}$
    where
    \begin{align*}
      \theta_{i(m,n)} &= |\eta_{n}\rangle_{2}\xi_{m} \in
      \alpha \lt \gamma, & \theta_{i(m,n)} &=
      |\xi'_{m}\rangle_{1}\eta'_{n} \in \beta \rt \delta
      &&\text{for all } m,n \in \naturals.
    \end{align*}
  \end{enumerate}
\end{proposition}
\begin{proof}
  The proof of assertion i) is straightforward; we only
  prove the existence of the first contraction in
  ii).  
  Let $\xi,\xi',\eta,\eta',i,\theta,\theta'$ be as above. Then  $\theta \in
  (\alpha \lt \gamma)^{\infty}$ and $\|\theta\| \leq
  \|\xi\|\|\eta\|$ because
  \begin{align*}
    \sum_{k} \theta_{k}^{*}\theta_{k} = \sum_{m,n}
    \xi_{m}^{*}\langle\eta_{n}|_{2}|\eta_{n}\rangle_{2}
    \xi_{m} = \sum_{m,n} \xi_{m}^{*}
    \rho_{\beta}(\eta_{n}^{*}\eta_{n}) \xi_{m} \leq
    \|\eta\|^{2} \sum_{m} \xi_{m}^{*}\xi_{m} \leq
    \|\eta\|^{2} \|\xi\|^{2},
  \end{align*}
  and similarly $\theta' \in (\beta \rt \delta)^{\infty}$
  and $\|\theta'\|\leq\|\xi'\|\|\eta'\|$.  Next, we show
  that $\omega_{\theta,\theta'}$ does not depend on $\xi$
  and $\xi'$ but only on $\omega_{\xi,\xi'} \in \Omega_{\alpha,\beta}(A)$. Let $\zeta'
  \in \frakK$ and $x \in \AfibreB$. Then
  \begin{align*}
    \omega_{\theta,\theta'}(x)\zeta' = \sum_{m,n \in
      \naturals} \xi_{m}^{*} \langle \eta_{n}|_{2}x
    |\xi'_{m}\rangle_{1} \eta'_{n}\zeta',
  \end{align*}
  where the sum converges in norm.   Fix any $n \in
  \naturals$. Then we find  a sequence $(k_{r})_{r}$ in
  $\naturals$ and $\eta''_{r,1},\ldots,\eta''_{r,k_{r}} \in
  \gamma$, $\zeta''_{r,1},\ldots,\zeta''_{r,k_{r}} \in
  \frakK$ such that the sum $\sum_{l=1}^{k_{r}}
  \eta''_{r,l}\zeta''_{r,l}$ converges in norm to
  $\eta'_{n}\zeta'$ as $r$ tends to infinity. But then
  \begin{align*}
    \sum_{m} \xi_{m}^{*}\langle
    \eta_{n}|_{2}x|\xi'_{m}\rangle_{1}\eta'_{n} \zeta' &=
    \lim_{r\to \infty} \sum_{m} \sum_{l=1}^{k_{r}}
    \xi_{m}^{*}\langle
    \eta_{n}|_{2}x|\xi'_{m}\rangle_{1}\eta''_{r,l}
    \zeta''_{r,l}
    \\
    &= \lim_{r\to \infty}  \sum_{l=1}^{k_{r}}\sum_{m}
    \xi_{m}^{*}\langle
    \eta_{n}|_{2}x|\eta''_{r,l}\rangle_{2} \xi'_{m}
    \zeta''_{r,l} \\
    &= \lim_{r\to \infty} \sum_{l=1}^{k_{r}}
    \omega_{\xi,\xi'}(\langle
    \eta_{n}|_{2}x|\eta''_{r,l}\rangle_{2})\zeta''_{r,l}.
  \end{align*}
  Note here that $\langle \eta_{n}|x|\eta''_{r,l}\rangle_{2}
  \in A$. Therefore, the sum on the left hand side only
  depends on $\omega_{\xi,\xi'} \in
  \Omega_{\alpha,\beta}(A)$ but not on $\xi,\xi'$, and since
  $n \in \naturals$ was arbitrary, the same is true for
  $\omega_{\theta,\theta'}(x)\zeta'$. A similar argument
  shows that $\omega_{\theta,\theta'}(x)^{*}\zeta$ depends
  on $\omega_{\eta,\eta'} \in \Omega_{\gamma,\delta}(B)$ but not on $\eta,\eta'$ for each
  $\zeta\in \frakK$.
\end{proof}

\subsection{Concrete Hopf $C^{*}$-bimodules and their
  convolution algebras }
\label{subsection:legs-hopf} 

The fiber product construction leads to the following
generalization of a Hopf $C^{*}$-algebra and of related
concepts.
\begin{definition}
  Let $\frakb=\cbasesb$ be a $C^{*}$-base.  A {\em
    comultiplication} on a
  $C^{*}$-$(\frakbo,\frakb)$-algebra $A^{\beta,\alpha}_{H}$
  is a jointly semi-normal morphism $\Delta$ from
  $A^{\beta,\alpha}_{H}$ to $A^{\beta,\alpha}_{H} \bfibre
  A^{\beta,\alpha}_{H}$ that is coassociative in the sense
  that the following diagram commutes:
    \begin{align*}
      \xymatrix@C=15pt@R=15pt{ A \ar[rrr]^{\Delta}
        \ar[dd]^{\Delta} &&& {\AfibreA} \ar[d]^{\Id
          \bfibre \Delta} \\
        &&& {A \fibre{\alpha}{\frakb}{(\beta \lt \beta)}
          (\AfibreA)}
        \ar@{^(->}[d] \\
        {\AfibreA} \ar[rr]^(0.35){\Delta\bfibre \Id} &&
        {(\AfibreA) \fibre{(\alpha \rt \alpha)}{\frakb}{\beta}
          A} \ar@{^(->}[r] &\mathcal{L}(H
        \rtensor{\alpha}{\frakb}{\beta} H
        \rtensor{\alpha}{\frakb}{\beta} H).  }
    \end{align*}
    A {\em (semi-)normal Hopf $C^{*}$-bimodule over
      $\frakb$} is a $C^{*}$-$(\frakbo,\frakb)$-algebra with
    a jointly (semi-)normal comultiplication. When we speak of
    a Hopf $C^{*}$-bimodule, we always mean a semi-normal
    one.
    A morphism of (semi-)normal Hopf $C^{*}$-bimodules
    $(A^{\beta,\alpha}_{H},\Delta_{A})$,
    $(B^{\delta,\gamma}_{K},\Delta_{B})$ over $\frakb$ is a
    jointly (semi-)normal morphism $\pi$ from
    $A^{\beta,\alpha}_{H}$ to $B^{\delta,\gamma}_{K}$
    satisfying $\Delta_{B} \circ \pi = (\pi \bfibre \pi)
    \circ \Delta_{A}$.

    Let $(A_{H}^{\beta,\alpha},\Delta)$ be a Hopf
    $C^{*}$-bimodule over $\frakb$. 
    A {\em bounded left Haar weight} for
    $(A_{H}^{\beta,\alpha},\Delta)$ is a completely positive
    contraction $\phi \colon A \to \frakB$ satisfying
    \begin{align*}
      \phi(a\rho_{\beta}(b)) &= \phi(a)b, & \phi(\langle
      \xi|_{1}\Delta(a)|\xi'\rangle_{1}) &=
      \xi^{*}\rho_{\beta}(\phi(a))\xi'
    \end{align*}
    for all $a \in A$, $b \in \frakB$,  $\xi,\xi' \in
    \alpha$.
    We call $\phi$ {\em normal} if $\phi \in
    \Omega_{M(\beta)}(A)$, where $M(\beta) = \{ T \in{\cal
      L}(\frakK,H) \mid T\frakBo \subseteq \beta, T^{*}\beta
    \subseteq \frakBo\}$.  Similarly, a {\em bounded right
      Haar weight} for $(A_{H}^{\beta,\alpha},\Delta)$ is a
    completely positive contraction $\psi \colon A \to
    \frakBo$ satisfying 
    \begin{align*}
      \psi(a\rho_{\alpha}(b^{\dagger})) &=
      \psi(a)b^{\dagger}, &
      \psi(\langle\eta|_{2}\Delta(a)|\eta'\rangle_{2}) &=
      \eta^{*}\rho_{\alpha}(\psi(a))\eta'
    \end{align*}
for all $a \in A$, $b^{\dagger} \in
    \frakBo$, $\eta,\eta' \in \beta$.
    We call $\psi$ {\em normal} if $\psi \in
    \Omega_{M(\alpha)}(A)$, where
    $M(\alpha)= \{ S \in {\cal L}(\frakK,H) \mid S\frakB
    \subseteq \alpha, S^{*}\alpha \subseteq \frakB\}$.

    A {\em bounded (left/right) counit} for
    $(A_{H}^{\beta,\alpha},\Delta)$ is a jointly semi-normal
    morphism of $C^{*}$-$(\frakbo,\frakb)$-algebras
    $\epsilon \colon A^{\beta,\alpha}_{H} \to {\cal
      L}(\frakK)^{\frakBo,\frakB}_{\frakK}$ that makes the
    (left/right one of the) following two diagrams commute:
    \begin{align} \label{eq:counit} \xymatrix@C=10pt@R=15pt{
        {A \fibre{\alpha}{\frakb}{\beta} A} \ar[d]_{\epsilon
          \bfibre \Id} && A \ar[ll]_{\Delta}
        \ar[d] \\
        {{\cal L}(\frakK) \fibre{\frakB}{\frakb}{\beta} A}
        \ar[r] & {{\cal L}(\frakK
          \rtensor{\frakB}{\frakb}{\beta} H)}
        \ar[r]^(0.6){\cong} & {{\cal L}(H),}} \qquad
      \xymatrix@C=10pt@R=15pt{ A \ar[rr]^{\Delta} \ar[d] &&
        {A \fibre{\alpha}{\frakb}{\beta} A} \ar[d]^{\Id \bfibre \epsilon} \\
        {{\cal L}(H)} & {{\cal L}(H
          \rtensor{\alpha}{\frakb}{\frakBo} \frakK)}
        \ar[l]_(0.6){\cong} & {A
          \fibre{\alpha}\frakb{\frakBo} {\cal L}(\frakK).}
        \ar[l]}
    \end{align}
\end{definition}
\begin{remark}
  Let $(A_{H}^{\beta,\alpha},\Delta)$ be a Hopf
  $C^{*}$-bimodule over $\frakb$.  Evidently, a completely
  positive contraction $\phi \colon A \to \frakB$ is a
  normal bounded left Haar weight for
  $(A^{\beta,\alpha}_{H},\Delta)$ if and only if $\phi \in
  \Omega_{M(\beta)}(A)$ and $(\Id \ast \phi) \circ
  \Delta = \rho_{\beta} \circ \phi$. A similar remark
  applies to normal bounded right Haar weights.
\end{remark}
Let $(A_{H}^{\beta,\alpha},\Delta)$ be a normal Hopf
$C^{*}$-bimodule over $\frakb$.  Combining Propositions
\ref{proposition:slice-pullback} and
\ref{proposition:slice-fiber}, we obtain for each of the
spaces $\Omega =\Omega_{\alpha}(A), \Omega_{\beta}(A),
\Omega_{\alpha,\beta}(A),\Omega_{\beta,\alpha}(A)$ a map
\begin{align} \label{eq:legs-convolution}
  \Omega \times \Omega \to \Omega, \quad (\omega,\omega')
  \mapsto \omega \ast \omega' :=(\omega \boxtimes \omega')
  \circ \Delta.
\end{align}
\begin{theorem} \label{theorem:legs-convolution}
  Let $(A_{H}^{\beta,\alpha},\Delta)$ be a normal Hopf
  $C^{*}$-bimodule over $\frakb$.  Then
  $\Omega_{\alpha}(A)$, $\Omega_{\beta}(A)$,
  $\Omega_{\alpha,\beta}(A)$, $\Omega_{\beta,\alpha}(A)$
  are Banach algebras with respect to the multiplication
  \eqref{eq:legs-convolution}.
\end{theorem}
\begin{proof}
  It only remains to show that the multiplication is
  associative, but this follows from the coassociativity of
  $\Delta$. 
\end{proof}

\subsection{The legs of a $C^{*}$-pseudo-multiplicative
  unitary }
\label{subsection:legs-pmu}

Let $\frakb=\cbasesb$ be a $C^{*}$-base, $\pmuspace$ a
$C^{*}$-$(\frakbo,\frakb,\frakbo)$-module and $\pmuoperator$
a $C^{*}$-pseudo-multiplicative unitary.  We  associate to
$V$ two algebras and, if $V$ is well behaved, two Hopf
$C^{*}$-bimodules as follows.  Let
\begin{align} \label{eq:pmu-legs} \hA_{V} &:= [
\bbeta{2} V \kalpha{2}] \subseteq {\cal L}(H), & A_{V}
&:= [\balpha{1} V \khbeta{1}] \subseteq {\cal
L}(H),
\end{align} where $\kalpha{2}, \khbeta{1} \subseteq {\cal
L}(H, \Hsource)$ and $\bbeta{2},\balpha{1} \subseteq {\cal
L}(\Hrange, H)$ are defined as in Subsection
\ref{subsection:pmu-rtp}.
\begin{proposition} \label{proposition:pmu-legs} The
following relations hold:
\begin{gather*}
  \begin{aligned} \hA_{V^{op}}&=A_{V}^{*}, & [ \hA_{V}
\hA_{V}] &= \hA_{V}, & [\hA_{V} H] &= H =
[\hA_{V}^{*}H], & [ \hA_{V} \beta] &= \beta =
[ \hA_{V}^{*} \beta],
\end{aligned} \\
[ \hA_{V} \rho_{\hbeta}(\frakB)] = [ \rho_{\hbeta}(\frakB)
\hA_{V}] = \hA_{V} = [ \hA_{V} \rho_{\alpha}(\frakBo)] = [
\rho_{\alpha}(\frakBo) \hA_{V}], \\
\begin{aligned} 
  A_{V^{op}} &=\hA_{V}^{*}, & [ A_{V} A_{V} ] &= A_{V},
  &[A_{V}H] &= H= [A_{V}^{*}H], & [ A_{V} \hbeta] &= \hbeta
  = [ A_{V}^{*}\hbeta],
\end{aligned}\\
[ A_{V} \rho_{\beta}(\frakB)] = [ \rho_{\beta}(\frakB)
A_{V}] = A_{V} = [ A_{V} \rho_{\alpha}(\frakBo)] = [
\rho_{\alpha}(\frakBo) A_{V}].
\end{gather*}
\end{proposition}
\begin{proof} 
  First, we have $\hA_{V^{op}}=[ \bhbeta{2} \Sigma V^{*}
  \Sigma \kalpha{2}] = [ \bhbeta{1}
  V^{*}\kalpha{1}]=A_{V}^{*}$ and $[ \hA_{V} \beta] = [
  \bbeta{2} V \kalpha{2}\beta] = [ \bbeta{2} \kbeta{2}\beta]
  = [ \rho_{\alpha}(\frakBo)\beta] =\beta$ because $V(\beta
  \lt \alpha) = \beta \lt \beta$.  Similarly, one shows that
  $[\hA_{V}^{*}\beta]=\beta$, $A_{V^{op}}=\hA_{V}^{*}$, and
  $[A_{V}\hbeta]=\hbeta=[A_{V}^{*}\hbeta]$.  The remaining
  equations are a particular case of Proposition
  \ref{proposition:rep-legs} in subsection
  \ref{subsection:rep-legs}.
\end{proof} 
Consider the $*$-homomorphisms
\begin{align*} 
  \widehat{\Delta}_{V} &\colon \rho_{\beta}(\frakB)' \to
  {\cal L}\big(\Hsource\big), \ y \mapsto V^{*}(\Id \btensor
  y)V, \\ \Delta_{V} &\colon \rho_{\hbeta}(\frakB)' \to
  {\cal L}\big(\Hrange), \ z \mapsto V(z \botensor
  \Id)V^{*}.
\end{align*}
\begin{proposition} \label{proposition:pmu-delta}
  $\hDelta_{V}$ is a jointly normal morphism of
  $C^{*}$-$(\frakb,\frakbo)$-algebras
  $(\rho_{\beta}(\frakB)')_{H}^{\alpha,\hbeta}$ and
  $\big((\rho_{\beta}(\frakB)
  \rtensor{\hbeta}{\frakbo}{\alpha}
  \rho_{\beta}(\frakB))'\big)^{(\alpha \lt \alpha),(\hbeta
    \rt \hbeta)}_{\Hsource}$, and $\Delta_{V}$ is a jointly
  normal morphism of $C^{*}$-$(\frakbo,\frakb)$-algebras
  $(\rho_{\hbeta}(\frakB)')_{H}^{\beta,\alpha}$ and
  $\big((\rho_{\hbeta}(\frakB)
  \rtensor{\alpha}{\frakb}{\beta}
  \rho_{\hbeta}(\frakB))'\big)^{(\beta \lt \beta),(\alpha
    \rt \alpha)}_{\Hrange}$. Moreover,
  $\hDelta_{V^{op}}=\Ad_{\Sigma} \circ \Delta_{V}$ and
  $\Delta_{V^{op}}=\Ad_{\Sigma} \circ \hDelta_{V}$.
\end{proposition}
\begin{proof}
  We only prove the assertions concerning $\hDelta_{V}$.
  The relation $\Delta_{V^{op}}=\Ad_{\Sigma} \circ
  \hDelta_{V}$ is easily verified. Next,
  $\hDelta_{V}(\rho_{\beta}(\frakB)') \subseteq
  (\rho_{\beta}(\frakB) \botensor \rho_{\beta}(\frakB))'$
  because $V (\rho_{\beta}(\frakB) \botensor
  \rho_{\beta}(\frakB)) = \rho_{\beta}(\frakB)\btensor
  \rho_{\hbeta}(\frakB) \subseteq \Id \btensor
  \rho_{\beta}(\frakB)'$ by \eqref{eq:pmu-intertwine}.  To
  see that $\hDelta_{V}$ is a jointly normal morphism, note
  that $V^{*}\kalpha{1} \subseteq {\cal
    L}^{\hDelta_{V}}\big(H,\Hsource\big)$ because
  $\hDelta(y)V^{*}|\xi\rangle_{{1}} = V^{*}(\Id \rtensorh
  y)|\xi\rangle_{{1}} = V^{*}|\xi\rangle_{{1}} y$ for all $y
  \in \rho_{\hbeta}(\frakB)'$, $\xi \in \alpha$, and that
  $\alpha \lt \alpha = [V^{*} \kalpha{1} \alpha ]$ and
  $\hbeta \rt \hbeta =[ V^{*}\kalpha{1}\hbeta]$ by
  \eqref{eq:pmu-intertwine}.
\end{proof}
Under favorable circumstances,
$((\hA_{V})^{\alpha,\hbeta}_{H},\hDelta_{V})$ and
$((A_{V})_{H}^{\beta,\alpha},\Delta_{V})$ will be concrete
Hopf $C^{*}$-bi\-modules. A sufficient condition,
regularity, will be given in subsection
\ref{subsection:regular}.  Coassociativity of
$\hDelta_{V}$ and $\Delta_{V}$ follows easily from the
commutativity of diagram \eqref{eq:pmu-pentagon}:
\begin{lemma} \label{lemma:pmu-delta-hopf} If
  $\hB \subseteq \rho_{\hbeta}(\frakB)'$ is a
  $C^{*}$-algebra,
  $\rho_{\alpha}(\frakBo)\hB+\rho_{\hbeta}(\frakB) \hB
  \subseteq \hB$ and $\hDelta_{V}(\hB) \subseteq \hB
  \fibre{\hbeta}{\frakbo}{\alpha} \hB$, then
  $(\hB^{\alpha,\hbeta}_{H},\hDelta_{V})$ is a normal Hopf
  $C^{*}$-bimodule over $\frakbo$.  Similarly, if $B
  \subseteq \rho_{\beta}(\frakB)'$ is a $C^{*}$-algebra,
  $\rho_{\beta}(\frakB)B+\rho_{\alpha}(\frakBo)B \subseteq
  B$ and $\Delta_{V}(B) \subseteq B
  \fibre{\alpha}{\frakb}{\beta} B$, then
  $(B^{\beta,\alpha}_{H},\Delta_{V})$ is a normal
  Hopf $C^{*}$-bimodule over $\frakb$.
\end{lemma}
\begin{proof} 
  We only prove the assertion concerning $\hB$; the
  assertion concerning $B$ follows similarly.  Let $\hB
  \subseteq \rho_{\hbeta}(\frakB)'$ be a $C^{*}$-algebra
  satisfying the assumptions and put $\hDelta:=\hDelta_{V}$.
  By Proposition \ref{proposition:pmu-delta}, we only need
  to show that $(\hDelta \bofibre \Id)(\hDelta(\hb))=(\Id
  \bofibre\hDelta)(\hDelta(\hb))$ for all $\hb \in \hB$. But
  this is shown by the following commutative diagram:
  \begin{gather*}
    \xymatrix@C=20pt@R=12pt{ & {\Hone} \ar[d]^{V_{12}} \ar
      `l/0pt[l] `d[lddddd]^{\Id} [ddddd] \ar[rrr]|{(\hDelta
        \bofibre \Id)(\hDelta(\ha))} &&& {\Hone} & \\ &
      {\Htwo} \ar[rrr]|{\Id \btensor \hDelta(\ha)}
      \ar[d]^{V_{23}} &&& {\Htwo} \ar[u]^{V_{12}^{*}} & \\ &
      {\Hthree} \ar[rrr]|{\Id \btensor \Id \btensor \ha}
      \ar[d]^{V_{12}^{*}} &&& {\Hthree} \ar[u]^{V_{23}^{*}}
      & \\ & {\Hfour} \ar[rrr]|{\Id \botensor \Id \btensor
        \ha} \ar[d]^{V_{12}^{*}\Sigma_{23}} &&& {\Hfour}
      \ar[u]^{V_{12}} & \\ & {\Hfourlt}
      \ar[rrr]|{\hDelta(\ha) \botensor \Id}
      \ar[d]^{\Sigma_{23}V_{23}} &&& {\Hfourlt}
      \ar[u]^{\Sigma_{23}V_{12}} & \\ & {\Hone }
      \ar[rrr]|{(\Id \bofibre \hDelta)(\hDelta(\ha))} &&&
      {\Hone} \ar[u]^{V_{23}^{*}\Sigma_{23}} \ar `r/0pt[r]
      `u[ruuuuu]^{\Id} [uuuuu] &} \\[-4.5ex] \qedhere
  \end{gather*}
\end{proof}

Using the maps introduced in subsection
\ref{subsection:legs-slice}, we construct
convolution algebras $\tilde \Omega_{\beta,\alpha}$ and
$\tilde \Omega_{\alpha,\hbeta}$  with homomorphisms onto
dense subalgebras $\hA^{0}_{V} \subseteq \hA_{V}$ and
$A^{0}_{V} \subseteq A_{V}$, respectively, as follows. Let
\begin{align*}
  \tilde \Omega_{\beta,\alpha} &:=
  \Omega_{\beta,\alpha}\big(\rho_{\hbeta}(\frakB)'\big), &
  \tilde \Omega_{\alpha,\hbeta} &:=
  \Omega_{\alpha,\hbeta}(\rho_{\beta}(\frakB)'). 
\end{align*}
\begin{theorem} \label{theorem:pmu-convolution}
  \begin{enumerate}
  \item   There exist linear contractions
    \begin{align*}
      \tilde \Omega_{\beta,\alpha} \hat\otimes \tilde
      \Omega_{\beta,\alpha} \to \Omega_{(\beta \lt
        \beta),(\alpha \rt
        \alpha)}\big((\rho_{\hbeta}(\frakB)
      \rtensor{\alpha}{\frakb}{\beta}
      \rho_{\hbeta}(\frakB))'\big), \quad \omega \otimes
      \omega' \mapsto \omega \boxtimes \omega', \\
      \tilde \Omega_{\alpha,\hbeta} \hat\otimes \tilde
      \Omega_{\alpha,\hbeta} \to \Omega_{(\alpha \lt
        \alpha),(\hbeta \rt
        \hbeta)}\big((\rho_{\beta}(\frakB)
      \rtensor{\hbeta}{\frakbo}{\alpha}
      \rho_{\beta}(\frakB))'\big), \quad \omega \otimes
      \omega' \mapsto \omega \boxtimes \omega',
    \end{align*}
    such that for all $\xi,\xi' \in \beta^{\infty}$,
    $\eta,\eta' \in \alpha^{\infty}$, $\zeta,\zeta' \in
    \hbeta^{\infty}$ and each bijection $i \colon \naturals
    \times \naturals \to \naturals$, we have
    $\omega_{\xi,\eta} \boxtimes \omega_{\xi',\eta'} =
    \omega_{\theta,\theta'}$ and $\omega_{\eta,\zeta}
    \boxtimes \omega_{\eta',\zeta'} =
    \omega_{\kappa,\kappa'}$, where for all $m,n \in
    \naturals$, 
    \begin{align*}
      \theta_{i(m,n)} &=|\xi'_{n}\rangle_{2} \xi_{m} \in
      \beta \lt \beta, & \theta'_{i(m,n)} &=
      |\eta_{m}\rangle_{1} \eta'_{n} \in \alpha \rt \alpha, \\
      \kappa_{i(m,n)} &= |\eta'_{n}\rangle_{2}\eta_{m} \in
      \alpha \lt \alpha, & \kappa'_{i(m,n)} &=
      |\zeta_{m}\rangle_{1} \zeta'_{n} \in \hbeta \rt
      \hbeta.
  \end{align*}
\item The Banach spaces $\tilde \Omega_{\beta,\alpha}$ and
  $\tilde \Omega_{\alpha,\hbeta}$ carry the structure of
  Banach algebras, where the multiplication is given by
  $\omega \ast \omega' = (\omega \boxtimes \omega') \circ
  \Delta_{V}$ and $\omega \ast \omega' = (\omega \boxtimes
  \omega') \circ \hDelta_{V}$, respectively.
\item There exist contractive algebra homomorphisms
  $\hpi_{V} \colon \tilde \Omega_{\beta,\alpha} \to \hA_{V}$
  and $\pi_{V} \colon \tilde \Omega_{\alpha,\hbeta} \to
  A_{V}$ such that for all $\xi \in \beta^{\infty}$, $\eta
  \in \alpha^{\infty}$, $\zeta \in \hbeta^{\infty}$,
  \begin{align*}
    \hpi_{V}(\omega_{\xi,\eta}) &= \sum_{n} \langle
    \xi_{n}|_{2}V|\eta_{n}\rangle_{2}, &
    \pi_{V}(\omega_{\eta,\zeta}) &= \sum_{n} \langle
    \eta_{n}|_{1}V|\zeta_{n}\rangle_{1}.
  \end{align*}
  \end{enumerate}
\end{theorem}
\begin{proof}
 i) This is a slight modification of Proposition
  \ref{proposition:slice-fiber} and follows from similar
  arguments.

  ii) The existence of the multiplication in ii) follows
  from i) and Propositions \ref{proposition:slice-pullback}
  and \ref{proposition:pmu-delta}, and associativity from
  coassociativity of $\Delta_{V}$ and $\hDelta_{V}$ (see the
  proof of  Lemma \ref{lemma:pmu-delta-hopf}).

 iii) This is a special case of
  the more general Proposition
  \ref{proposition:pmu-convolution-rep} which is proven in
  subsection \ref{subsection:rep-banach}.
\end{proof}
If $((\hA_{V})_{H}^{\alpha,\hbeta},\hDelta_{V})$
and $((A_{V})^{\beta,\alpha}_{H},\Delta_{V})$ are Hopf
$C^{*}$-bimodules, they should be thought of as standing in
a generalized Pontrjagin duality. This duality is captured
by a pairing on the dense subalgebras
 \begin{align*}
   \hA^{0}_{V} &:=\hpi_{V}(\tilde \Omega_{\beta,\alpha})
   \subseteq \hA_{V}, & A_{V}^{0} &:= \pi_{V}(\tilde
   \Omega_{\alpha,\hbeta}) \subseteq A_{V}.
 \end{align*}
\begin{definition}
  We call the algebra $\hA_{V}^{0} \subseteq \hA_{V}$,
  equipped with the quotient norm from the surjection
  $\hpi_{V}$, the {\em Fourier algebra} of $V$. Similarly,
  we call the algebra $A_{V}^{0} \subseteq A_{V}$,
  equipped with the quotient norm from the surjection
  $\pi_{V}$, the {\em dual Fourier algebra} of $V$.
\end{definition}
 \begin{proposition}
   \begin{enumerate}
   \item There exists a bilinear map $(\frei |\frei) \colon
     \hA_{V}^{0} \times A_{V}^{0} \to {\cal L}(\frakK)$ such
     that $\omega(\pi_{V}(\upsilon)) = \big(\hpi_{V}(\omega)
     \big| \pi_{V}(\upsilon)\big) =
     \upsilon(\hpi_{V}(\omega))$ for all $\omega \in \tilde
     \Omega_{\beta,\alpha}, \upsilon \in \tilde
     \Omega_{\alpha,\hbeta}$.
     \item This map is nondegenerate in the sense that for
       each $\ha \in \hA^{0}_{V}$ and $a \in A_{V}^{0}$,
       there exist $\ha' \in \hA^{0}_{V}$ and $a' \in
       A_{V}^{0}$ such that $(\ha|a') \neq 0$ and
       $(\ha'|a)\neq 0$.
   \item $(\hpi_{V}(\omega)\hpi_{V}(\omega')|a)= (\omega
     \boxtimes \omega')(\Delta_{V}(a))$ and $(\ha |
     \pi_{V}(\upsilon) \pi_{V}(\upsilon')) = (\upsilon
     \boxtimes \upsilon')(\hDelta_{V}(\ha))$ for all
     $\omega,\omega' \in \tilde \Omega_{\beta,\alpha}$, $a
     \in A^{0}_{V}$, $\upsilon,\upsilon' \in \tilde
     \Omega_{\alpha,\hbeta}$, $\ha \in \hA^{0}_{V}$.
 \end{enumerate}
 \end{proposition}
 \begin{proof}
   i) If $\omega=\omega_{\xi,\xi'}$ and
   $\upsilon=\omega_{\eta,\eta'}$, where $\xi\in \beta^{\infty}$, $\xi',\eta \in
   \alpha^{\infty}$, $\eta' \in \hbeta^{\infty}$, then
   \begin{align*}
     \omega(\pi_{V}(\upsilon)) = \sum_{m,n}
     \xi_{m}^{*}\langle\eta_{n}|_{1} V
     |\eta'_{n}\rangle_{1}\xi'_{m} = \sum_{m,n}
     \eta_{n}^{*}\langle\xi_{m}|_{2} V
     |\xi'_{m}\rangle_{2}\eta'_{n} =
     \upsilon(\hpi_{V}(\omega)).
   \end{align*}

   ii) Evident.

   iii) For all $\omega,\omega',a$ as above,
   $(\hpi_{V}(\omega)\hpi_{V}(\omega')|a) = (\hpi_{V}(\omega
   \ast \omega')|a) = (\omega \ast \omega')(a)  = (\omega \boxtimes \omega)(\Delta_{V}(a))$. The
   second equation follows similarly.
 \end{proof}
 As a consequence of part ii) of the preceding result, we
 obtain the following simple relation between the Fourier
 algebra $\hA^{0}_{V}$ and the convolution algebra
 constructed in Theorem \ref{theorem:legs-convolution}.
\begin{proposition}
  If $((A_{V})^{\beta,\alpha}_{H},\hDelta_{V})$ or
  $((\hA_{V})^{\alpha,\hbeta}_{H},\Delta_{V})$ is a normal
  Hopf $C^{*}$-bimodule, then we have a commutative diagram of
  Banach algebras and homomorphisms
  \begin{align*}
    \xymatrix@C=15pt@R=15pt{ \tilde \Omega_{\beta,\alpha}
      \ar[r]^{\hpi_{V}} \ar@{=}[d] 
       &  \hA_{V}^{0} \\
      \Omega_{\beta,\alpha}(\rho_{\hbeta}(\frakB)')
      \ar[r]^{q} & {\Omega_{\beta,\alpha}(A_{V})}
      \ar[u]_{\hpi}} \quad \text{or} \quad
    \xymatrix@C=15pt@R=15pt{ \tilde \Omega_{\alpha,\hbeta}
      \ar[r]^{\pi_{V}} \ar@{=}[d]
       &  A_{V}^{0} \\
      \Omega_{\alpha,\hbeta}(\rho_{\beta}(\frakB)')
      \ar[r]^{q} & {\Omega_{\alpha,\hbeta}(\hA_{V})}
      \ar[u]_{\pi},}
  \end{align*}
  respectively, where $q$ is the quotient map and
  $\hpi$ or $\pi$  an isometric isomorphism.  \qed
\end{proposition}

\subsection{The legs of the unitary of a groupoid } 

\label{subsection:legs-groupoid}

The general preceding constructions are now applied to the
$C^{*}$-pseudo-multiplicative unitary of a locally compact,
Hausdorff, second countable groupoid $G$ that was
constructed in subsection \ref{subsection:pmu-groupoid}. The
algebras $A_{V}$ and $\hA_{V}$ turn out to be the reduced
groupoid $C^{*}$-algebra $C^{*}_{r}(G)$ and the function
algebra $C_{0}(G)$, respectively, but unfortunately, we can
not determine the Fourier algebras $\hA^{0}_{V}$ and
$A^{0}_{V}$.

 We use the same notation as in subsection
\ref{subsection:pmu-groupoid} and let
\begin{gather*}
  \begin{aligned}
    \frakK&:=L^{2}(G^{0},\mu), & \frakB&= \frakBo:=
    C_{0}(G^{0}) \subseteq {\cal L}(\frakK), &
    \frakb&:=(\frakK,\frakB,\frakBo),
  \end{aligned} \\
  \begin{aligned}
    H &:= L^{2}(G,\nu), &
    \alpha&=\beta:=j(L^{2}(G,\lambda)), & 
    \hbeta := \hat j(L^{2}(G,\lambda^{-1})),
  \end{aligned} \\
  V \colon \Hsource \cong L^{2}(\GsrG,\nu^{2}_{s,r}) \to
  L^{2}(\GrrG,\nu^{2}_{r,r}) \cong \Hrange, \\
  (V\omega)(x,y) = \omega(x,x^{-1}y)  \text{ for all }
  \omega \in C_{c}(\GsrG), \, (x,y) \in \GrrG.
\end{gather*}
 Denote by $m \colon C_{0}(G) \to {\cal L}(H)$ the
representation given by multiplication operators, and
by $L^{1}(G,\lambda)$ the completion of $C_{c}(G)$ with
respect to the norm given by
\begin{align*}
  \|f\| &:= \sup_{u \in G^{0}} \int_{G^{u}}
  |f(u)| \intd\lambda^{u}(x) \quad \text{for all } f\in
  C_{c}(G).
\end{align*}
Then $L^{1}(G,\lambda)$ is a Banach algebra
with respect to the convolution product
\begin{align*}
  (f \ast g)(y) = \int_{G^{r(y)}}
  g(x)f(x^{-1}y)\intd\lambda^{r(y)}(x) \quad \text{for all }
  f,g \in L^{1}(G,\lambda), y \in G,
\end{align*}
and there exists a norm-decreasing algebra homomorphism
$L \colon L^{1}(G,\lambda) \to {\cal L}(H)$ such that 
\begin{align*}
  \big(L(f)\xi\big)(y) = \int_{G^{r(y)}} f(x)D^{-1/2}(x)
  \xi(x^{-1}y) \intd\lambda^{r(y)}(x) \quad \text{for all }
  f,\xi \in C_{c}(G), y \in G.
\end{align*}
For all $\xi,\xi' \in L^{2}(G,\lambda)$ and $\eta \in
L^{2}(G,\lambda),\eta' \in L^{2}(G,\lambda^{-1})$, let
\begin{align*}
 \ha_{\xi,\xi'} &= \langle j(\xi)|_{2}V|j(\xi')\rangle_{2}
  \in \hA^{0}_{V} && \text{and} &
  a_{\eta,\eta'} &= \langle j(\eta)|_{1} V |\hat
  j(\eta')\rangle_{1} \in A^{0}_{V}.
\end{align*}
Routine arguments show that there exists a unique continuous
map
  \begin{align*}
     L^{2}(G,\lambda) \times L^{2}(G,\lambda) &\to
    C_{0}(G), \ (\xi,\xi') \mapsto \overline{\xi} \ast
    \xi'{}^{*},
 \end{align*}
 such that 
    \begin{align*}
      (\overline{\xi} \ast \xi'{}^{*})(x)&= \int_{G^{r(x)}}
      \overline{\xi(y)}\xi'(x^{-1}y) \intd\lambda^{r(x)}(y)
      \quad \text{for all } \xi,\xi' \in C_{c}(G), x\in G.
 \end{align*}
 \begin{lemma} \label{lemma:legs-groupoid} Let $\xi,\xi' \in
   L^{2}(G,\lambda)$ and $\eta,\eta' \in C_{c}(G)$. Then
   $\ha_{\xi,\xi'}=m(\overline{\xi} \ast \xi'{}^{*})$ and
   $a_{\eta,\eta'} = L(\overline{\eta}\eta')$.
\end{lemma}
\begin{proof}
  By continuity, we may assume $\xi,\xi' \in C_{c}(G)$. Then
  for all $\zeta,\zeta' \in C_{c}(G)$,
  \begin{align*} 
    \langle \zeta| \ha_{\xi,\xi'} \zeta'\rangle
    &= \langle \zeta \tl j(\xi)|V(\zeta' \tl j(\xi'))\rangle \\
    &= \int_{G} \int_{G^{r(x)}} \overline{\zeta(x)\xi(y)}
    \zeta'(x)\xi'(x^{-1}y) \intd\lambda^{r(x)}(y)
    \intd\nu(x) = \langle \zeta|m(\overline{\xi} \ast
    \xi'{}^{*}) \zeta'\rangle,
    \\
    \langle \zeta|a_{\eta,\eta'}\zeta'\rangle &= \langle
    j(\eta) \tr \zeta| V(\hat j(\eta') \tr \zeta')\rangle \\
    &= \int_{G} \int_{G^{r(y)}} \overline{\eta(x)\zeta(y)}
    \eta'(x)D^{-1/2}(x) \zeta'(x^{-1}y)
    \intd\lambda^{r(y)}(x) \intd\nu(y) \\ &= \langle
    \zeta|L(\overline{\eta}\eta')\zeta'\rangle. \qedhere \end{align*}
\end{proof}
\begin{remark}
  To extend the formula $a_{\eta,\eta'} =
  L(\overline{\eta}\eta')$ to all $\eta \in
  L^{2}(G,\lambda)$, $\eta' \in L^{2}(G,\lambda^{-1})$, we
  would have to extend the representation $L \colon C_{c}(G)
  \to {\cal L}(H)$ to some algebra $X$ and the pointwise
  multiplication $(\eta,\eta') \mapsto \overline{\eta}\eta'$
  to a map $L^{2}(G,\lambda) \times L^{2}(G,\lambda^{-1})
  \to X$. Note that pointwise multiplication extends to a
  continuous map $L^{2}(G,\lambda) \times L^{2}(G,\lambda)
  \to L^{1}(G,\lambda)$, but in general this is not what we
  need. We expect that the map $L \colon C_{c}(G ) \to
  A^{0}_{V}$ does not extend to an isometric isomorphism of
  Banach algebras $L^{1}(G,\lambda) \to A^{0}_{V}$.
\end{remark}

The algebra $\hA^{0}_{V}$ can be considered as a continuous
Fourier algebra of the locally compact groupoid $G$.
Another Fourier algebra for locally compact groupoids was
defined by Paterson in \cite{paterson:fourier} as
follows. He constructs a Fourier-Stieltjes algebra $B(G)
\subseteq C(G)$ and defines the Fourier algebra $A(G)$ to be
the norm-closed subalgebra of $B(G)$ generated by the set
$A_{cf}(G):=\{ \ha_{\xi,\xi'} \mid \xi \in
L^{2}(G,\lambda)\}$.  The definition of $B(G)$ in
\cite{paterson:fourier} immediately implies that $\|
\hpi_{V}(\omega_{\xi,\xi'}) \|_{B(G)} \leq \|\xi\|\|\xi'\|$
for all $\xi \in \alpha^{\infty}, \xi' \in \beta^{\infty}$
with finitely many non-zero components, whence the following
relation holds:
\begin{proposition}
  The identity on $A_{cf}(G)$ extends to a norm-decreasing
  homomorphism of Banach algebras $\hA_{V}^{0} \to
  A(G)$. \qed
\end{proposition}
Another Fourier space $\tilde {\cal A}(G)$ considered
in \cite[Note after Proposition 13]{paterson:fourier} is
defined as follows. For each $\eta \in L^{2}(G,\lambda)$ and
$u \in G^{0}$, write $\|\xi_{n}(u)\|:=\langle
\xi_{n}|\xi_{n}\rangle(u)^{1/2}$. Denote by $M$ the set of
all pairs $(\xi,\xi')$ of sequences in $L^{2}(G,\lambda)$
such that the supremum $|(\xi,\xi')|_{M}:=\sup_{u,v\in
  G^{0}} \sum_{n}\|\xi_{n}(u)\| \|\xi'_{n}(v)\|$ is finite,
and denote by $\tilde {\cal A}(G)$ the completion of the
linear span of $A_{cf}(G)$ with respect to the norm defined
by
\begin{align*}
  \|\ha\|_{\tilde {\cal A}(G)} = \inf \left\{
    |(\xi,\xi')|_{M} \middle| \ha = \sum_{n}
   \ha_{\xi_{n},\xi'_{n}}\right\}.
\end{align*}
\begin{proposition}
  The identity on $A_{cf}(G)$ extends to a linear
  contraction $\hA^{0}_{V} \to \tilde {\cal A}(G)$.
\end{proposition}
\begin{proof}
  For all $\xi,\xi' \in L^{2}(G,\lambda)^{\infty}$, we
  have
  \begin{align*}
    \|\xi\|^{2} &= \sup_{u \in G^{0}} \sum_{n} \langle
    \xi_{n}|\xi_{n}\rangle(u) = \sup_{u \in G^{0}} \sum_{n}
    \|\xi_{n}(u)\|^{2}, &
    \|\xi'\|^{2} =\sup_{v \in G^{0}} \sum_{n}
    \|\xi_{n}(v)\|^{2},
  \end{align*}
  and therefore
  $|(\xi,\xi')|_{M} = \sup_{u,v\in G^{0}}
  \sum_{n}\|\xi_{n}(u)\| \|\xi'_{n}(v)\|
  \leq\|\xi\|\|\xi'\|$. 
\end{proof}
Let us add that a Fourier algebra for measured groupoids
was defined and studied by Renault \cite{renault:fourier},
and for measured quantum groupoids by Vallin
\cite{vallin:1}. 

Finally, we consider the $C^{*}$-algebras associated to
$V$. Recall that the reduced groupoid $C^{*}$-algebra
$C^{*}_{r}(G)$ is the closed linear span of all operators of
the $L(g)$, where $g \in L^{1}(G,\lambda)$ \cite{renault}.
\begin{theorem} \label{theorem:legs-groupoid}
  Let $V$ be the $C^{*}$-pseudo-multiplicative unitary of a
  locally compact groupoid $G$.  Then
  $((\hA_{V})^{\alpha,\beta}_{H},\hDelta_{V})$ and
  $((A_{V})_{H}^{\hbeta,\alpha},\Delta_{V})$ are Hopf
  $C^{*}$-bimodule and
  \begin{gather*}
    \begin{aligned}
      \hA_{V} &= m(C_{0}(G)), &
      \big(\hDelta_{V}(m(f))\omega\big)(x,y) &= f(xy) \omega(x,y), \\
      A_{V} &= C^{*}_{r}(G), &
      \big(\Delta_{V}(L(g))\omega'\big)(x',y') &=
      \int_{G^{u'}} g(z)D^{-1/2}(z)
      \omega'(z^{-1}x',z^{-1}y') \intd\lambda^{u'}(z)
    \end{aligned}
    \end{gather*}
    for all $f \in C_{0}(G)$, $\omega \in \Hsource$, $(x,y)
    \in \GsrG$ and $g \in C_{c}(G)$, $\omega' \in
    \Hrange$, $(x',y') \in \GrrG$, where $u'=r(x')=r(y')$.
\end{theorem}
\begin{proof}
  The first assertion will follow from Example
  \ref{examples:pmu-regular}  and Theorem
  \ref{theorem:regular-legs} in subsection
  \ref{subsection:regular}. The equations
  concerning $\hA_{V}$ and $A_{V}$ follow directly from
  Lemma \ref{lemma:legs-groupoid}. Let us prove the formulas
  for $\hDelta_{V}$ and $\Delta_{V}$.  For all
  $f,\omega,(x,y)$ as above,
  \begin{multline*} 
    \big( \hDelta_{V}(m(f))\omega\big)(x,y) = \big(V^{*}(
    \Id \tr m(f))V\omega \big) (x,y) \\ = \big(( \Id \tr
    m(f))V\omega \big) (x,xy) = f(xy)(V\omega)(x,xy) =
    f(xy)\omega(x,y),
  \end{multline*}
and for all $g, (x',y'),\omega',u'$ as above,
\begin{align*} 
  \big( \Delta_{V}(L(g))\omega'\big)(x',y') &= \big(V(L(g)
  \tl \Id)V^{*}\omega' \big) (x',y') \\ &=
  \big((L(g) \tl \Id)V^{*}\omega' \big) (x',x'{}^{-1}y') \\
  &= \int_{G^{u'}} g(z) D^{-1/2}(z)
  (V^{*}\omega')(z^{-1}x',x'{}^{-1}y') \intd\lambda^{u'}(z)
  \\ &= \int_{G^{u'}} g(z)
  D^{-1/2}(z)\omega'(z^{-1}x',z^{-1}y) \intd\lambda^{u'}(z).
  \qedhere
  \end{align*}
\end{proof}

\section{Representations of a $C^{*}$-pseudo-multiplicative
  unitary}
\label{section:reps}

Let $G$ be a locally compact group and let $V$ be the
multiplicative unitary on the Hilbert space
$L^{2}(G,\lambda)$ given by formula \eqref{eq:pmu-group}.
Then one can associate to every unitary representation $\pi$
of $G$ on a Hilbert space $K$ a unitary operator $X$ on
$L^{2}(G,\lambda) \otimes K \cong L^{2}(G,\lambda;K)$ such
that $(Xf)(x) = \pi(x)f(x)$ for all $x \in G$ and $f \in
L^{2}(G,\lambda;K)$, and this operator satisfies the
modified pentagon equation
\begin{align} \label{eq:rep-ordinary-pentagon}
 V_{12}X_{13}X_{23}=V_{23}X_{12}. 
\end{align}

For a general multiplicative unitary $V$ on a Hilbert space
$H$, Baaj and Skandalis defined a representation on a
Hilbert space $K$ to be a unitary $X$ on $H \otimes K$
satisfying equation \eqref{eq:rep-ordinary-pentagon},
equipped the class of all such representations with the
structure of a $C^{*}$-tensor category and showed that under
the assumption of regularity, this $C^{*}$-tensor category
is the category of representations of a Hopf $C^{*}$-algebra
$(A_{(u)},\Delta_{(u)})$ (see \cite{baaj:2}).  In the case
where $V$ is the unitary associated to a group $G$ as above,
this category is isomorphic to the category of unitary
representations of $G$, and $A_{(u)}$ is the full group
$C^{*}$-algebra $C^{*}(G)$.

We carry over these definitions and constructions to
$C^{*}$-pseudo-multiplicative unitaries and relate them to
representations of groupoids.
Throughout this section, let $\frakb=\cbasesb$ be a
$C^{*}$-base, $\pmuspace$ a
$C^{*}$-$(\frakbo,\frakb,\frakbo)$-module and $\pmuoperator$
a $C^{*}$-pseudo-multiplicative unitary.

\subsection{The $C^{*}$-tensor category of representations }
\label{subsection:rep-category}

  Let $\cKhd$ be a $C^{*}$-$(\frakb,\frakbo)$-module
and $X \colon K \rbotensor{\hdelta}{\alpha} H \to K
\rbtensor{\gamma}{\beta} H$  an operator satisfying
    \begin{gather} \label{eq:rep-intertwine}
      \begin{aligned} X(\gamma \lt \alpha) &= \gamma \rt
\alpha, & X(\hdelta \rt \beta) &= \hdelta \lt \beta, &
X(\hdelta \rt \hbeta) &= \gamma \rt \hbeta
\end{aligned}
\end{gather} 
as subsets of ${\cal L}(\frakK,\rHrange)$.
Then all operators in the following diagram
are well defined,
\begin{gather} \label{eq:rep-pentagon} \smalldiagram
  \begin{gathered} \xymatrix@R=15pt@C=20pt{ {\rHone}
      \ar[r]^{X \botensor \Id} \ar[d]^{\Id \botensor V} &
      {\rHtwo} \ar[r]^{\Id \btensor V} & {\rHthree,} \\
      {\rHfive} \ar[d]^{\Id \botensor \Sigma} && {\rHfour}
      \ar[u]^{X \btensor \Id} \\ {\rHfourlt} \ar[rr]^{X
        \botensor \Id} && {\rHfourrt} \ar[u]^{\Sigma_{23}}
    }
  \end{gathered}
\end{gather}
where the canonical isomorphism $\Sigma_{23}\colon \rHfourrt \cong (K
{_{\rho_{\gamma}} \tl} \beta) {_{\rho_{(\hdelta \lt \beta)}}
  \tl} \alpha \mycong (K {_{\rho_{\hdelta}} \tl} \alpha)
{_{\rho_{(\gamma \lt \alpha)}} \tl} \beta \cong \rHfour$ is
given by $(\zeta \tl \xi) \tl \eta \mapsto (\zeta \tl \eta)
\tl \xi$.  We again adopt the leg notation \cite{baaj:2} and
write
\begin{align*}
  X_{12} &\text{ for } X \botensor \Id \text{ and } X
  \btensor \Id; & X_{13} &\text{ for } \Sigma_{23}(X
  \botensor \Id)(\Id \botensor \Sigma).
\end{align*}
\begin{definition} \label{definition:pmu-rep} A {\em
    representation} of $V$ consists of a
  $C^{*}$-$(\frakb,\frakbo)$-module $\cKhd$ and a unitary $X
  \colon K \rbotensor{\hdelta}{\alpha} H \to K
  \rbtensor{\gamma}{\beta} H$ such that equation
  \eqref{eq:rep-intertwine} holds and diagram
  \eqref{eq:rep-pentagon} commutes.  We also call $X$ a
  representation of $V$ (on $\cKhd$).  A {\em
    (semi-)morphism} of representations $(\cKhd,X)$ and
  $(\eLhf,Y) $ is an operator $T \in {\cal
    L}_{(s)}(\cKhd,\eLhf)$ satisfying $Y(T \botensor \Id)=(T
  \btensor \Id)X$.  Evidently, the class of all
  representations and (semi-)morphisms forms a category; we
  denote it by $\bfcs\bfrep^{(s)}_{V}$.
\end{definition} 
\begin{examples} \label{examples:rep}
  \begin{enumerate}
  \item Consider the canonical isomorphisms
    \begin{align} \label{eq:rep-canonical-triv} \Phi &\colon
      \frakK \rtensor{\frakBo}{\frakbo}{\alpha} H \to H, \
      b^{\dag} \tr \zeta \mapsto
      \rho_{\alpha}(b^{\dag})\zeta, & \Psi &\colon \frakK
      \rtensor{\frakB}{\frakb}{\beta} H \to H, \ b \tr \zeta
      \mapsto \rho_{\beta}(b)\zeta.
\end{align} 
The composition $\trivrep:=\Psi^{*}\Phi$ is a
representation on ${_{\frakB}\frakK_{\frakBo}}$ which we call
the {\em trivial} representation of $V$.
\item The pair $({_{\alpha}H_{\hbeta}},V)$ is a
  representation which we call the {\em regular}
  representation.
\item Let $(({_{\gamma_{i}}K^{i}_{\hdelta_{i}}},X_{i}))_{i}$
  be a family of representations. Then the operator
    \begin{align*} \boxplus_{i} X_{i} \colon
      \big(\boxplus_{i} K^{i}_{\hdelta_{i}} \big)
      \botensor {_{\alpha}H} \to \big(\boxplus_{i}
      K^{i}_{\gamma_{i}} \big) \btensor {_{\beta}H}
    \end{align*}  corresponding to
    $\bigoplus_{i} X_{i}$ with respect to the
    identifications $\big(\boxplus_{i}
    K^{i}_{\hdelta_{i}} \big) \botensor {_{\alpha}H} \cong
    \bigoplus_{i} \big(K^{i}
    \rtensor{\hdelta^{i}}{\frakbo}{\alpha} H\big)$ and
    $\big(\boxplus_{i} K^{i}_{\gamma_{i}} \big) \btensor
    {_{\beta}H} \cong \bigoplus_{i} \big(K^{i}
    \rtensor{\gamma^{i}}{\frakb}{\beta} H\big)$ is a
    representation on $\boxplus_{i}
    {_{\gamma_{i}}K^{i}_{\hdelta_{i}}}$. We call it the {\em
      direct sum} of $({ X}_{i})_{i}$.
  \item Let $\frakc$ be a $C^{*}$-base, $L_{\lambda}$ a
    $C^{*}$-$\frakc$-module, $(K,\gamma,\hdelta,\kappa)$
    a $C^{*}$-$(\frakb,\frakbo,\frakco)$-module, and $X$ a
    representation on $\cKhd$. If $X(\kappa
    \lt \alpha) = \kappa \lt \beta$, then the operator
    \begin{align*}
      \Id \ctensor X\colon L
      \rtensor{\lambda}{\frakc}{\kappa} K
      \rtensor{\hdelta}{\frakbo}{\alpha} H \to L
      \rtensor{\lambda}{\frakc}{\kappa} K
      \rtensor{\gamma}{\frakb}{\beta} H
    \end{align*}
    is a representation on ${_{\lambda \rt \gamma}(L
      \rtensor{\lambda}{\frakc}{\kappa} K)_{\lambda \rt
        \hdelta}}$, as one can easily check.
\end{enumerate}
\end{examples}

The category of representations admits a
tensor product:
\begin{lemma} 
  Let $(\cKhd,X)$ and $(\eLhf,Y)$ be representations of
  $V$. Then the operator
  \begin{gather*} X \boxtimes Y \colon K
\rtensor{\hdelta}{\frakbo}{\epsilon} L
\rtensor{\widehat\phi}{\frakbo}{\alpha} H
\xrightarrow{\displaystyle Y_{23}} K
\rtensor{\hdelta}{\frakbo}{\epsilon \rt \alpha} (L
\rtensor{\epsilon}{\frakb}{\beta} H)
\xrightarrow{\displaystyle X_{13}} (K
\rtensor{\hdelta}{\frakbo}{\epsilon} L) \rtensor{\gamma \lt
\epsilon}{\frakb}{\beta} H,
  \end{gather*} 
  where $Y_{23}=\Id \botensor Y$ and where $X_{13}$ acts
  like $X$ on the first and last factor of the relative
  tensor product, is a representation of $V$ on $\cKhd
  \botensor \eLhf$.
\end{lemma}
\begin{proof} First, the relations \eqref{eq:rep-intertwine}
for $X$ and $Y$ imply
\begin{align*} X_{13}Y_{23}(\gamma \Lt{\hdelta} \epsilon
  \Lt{\hphi} \alpha) &= X_{13}(\gamma \Lt{\hdelta} (\epsilon
  \Rt{\beta} \alpha)) = (\gamma \Lt{\hdelta} \epsilon)
  \Rt{\beta} \alpha, \\ X_{13} Y_{23}(\hdelta \Rt{\epsilon}
  \hphi \Rt{\alpha} \beta) &= X_{13}(\hdelta \Rt{(\epsilon
    \rt \alpha)} (\hphi \Lt{\epsilon} \beta)) = (\hdelta
  \Rt{\epsilon} \hphi ) \Lt{(\gamma \lt \epsilon)} \beta, \\
  X_{13}Y_{23}(\hdelta \Rt{\epsilon} \hphi \Rt{\alpha}
  \hbeta) &= X_{13}(\hdelta \Rt{(\epsilon \rt \alpha)}
  (\epsilon \Rt{\beta} \hbeta)) = (\gamma \Lt{\hdelta}
  \epsilon) \Lt{(\hdelta \rt \hphi)} \beta.
  \end{align*} 
   If $V$ is an ordinary multiplicative
  unitary, then $Z:=X\boxtimes Y$ satisfies
  $Z_{12}Z_{13}V_{23}=V_{23}Z_{12}$ because the equations
  $Y_{12}Y_{13}V_{23} = V_{23}Y_{12}$,
  $X_{12}X_{13}V_{23}=V_{23}X_{12}$ imply
  $X_{13}Y_{23}X_{14}Y_{24}V_{34} =
  X_{13}X_{14}Y_{23}Y_{24}V_{34} = X_{13}X_{14}V_{34}Y_{23}
  = V_{34}X_{13}Y_{23}$; here, we used the leg notation
  \cite{baaj:2}. A similar calculation applies to the
  general case.
\end{proof} 
The tensor product turns $\bfcs\bfrep_{V}$ and
$\bfcs\bfrep_{V}^{s}$ into ($C^{*}$-)tensor categories,
which are frequently also called ($C^{*}$-)monoidal
categories \cite{doplicher,hayashi,maclane}:
\begin{theorem} \label{theorem:rep-monoidal} The category
  $\bfcs\bfrep_{V}$ carries the structure of a $C^{*
  }$-tensor category and the category $\bfcs\bfrep_{V}^{s}$
  carries the structure of a tensor category, where both
  times
  \begin{itemize}
  \item the tensor product is given by $({X},{Y})
    \mapsto {X} \boxtimes {Y}$ for representations
    and $(S,T) \mapsto S \botensor T$ for morphisms;
  \item the unit is the trivial representation $\trivrep$;
  \item the associativity isomorphism $a_{{X},{ Y},{Z}}
    \colon ({X} \boxtimes {Y}) \boxtimes {Z} \to {X}
    \boxtimes ({Y} \boxtimes {Z})$ is the isomorphism
    $a_{{\cal K},{\cal L},{\cal M}}$ of equation
    \eqref{eq:rtp-associative} for all representations
    $({\cal K},X),({\cal L},Y),({\cal M},Z)$;
  \item the unit isomorphisms $l_{{ X}} \colon \trivrep
    \boxtimes { X} \to { X}$ and $r_{{ X}} \colon { X}
    \boxtimes \trivrep \to { X}$ are the isomorphisms
    $l_{{\cal K}}$ and $r_{{\cal K}}$, respectively, of
    equation \eqref{eq:rtp-unital} for each representation
    $({\cal K},X)$.
  \end{itemize}
\end{theorem}
\begin{proof} 
  Tedious but straightforward.
\end{proof} 

The regular representation tensorially absorbs every other
representation:
\begin{proposition} \label{lemma:rep-absorb} Let $(\cKhd,X)$
  be a representation of $V$. Then $X$ is an
  isomorphism between the representation ${ X}
  \boxtimes V$ and the amplification
  $\Id \btensor V$ on ${_{\gamma \rt \alpha}(K
    \rtensor{\gamma}{\frakb}{\beta} H)_{\gamma \rt
      \hbeta}}$.
\end{proposition}
\begin{proof}
  This follows from commutativity of
  \eqref{eq:rep-pentagon}.
\end{proof}

We denote by $\End(\trivrep)$ the algebra of endomorphisms
of the trivial representation $\trivrep$. This is a
commutative $C^{*}$-algebra, and the category
$\bfcs\bfrep_{V}$ can be considered as a bundle of
$C^{*}$-categories over the spectrum of
$\mathrm{End_{V}}(\trivrep)$ \cite{vasselli}.
\begin{proposition} \label{proposition:rep-end}
  $\End(\trivrep) = \{ b \in M(\frakB) \cap M(\frakBo)
  \subseteq {\cal L}(\frakK) \mid \rho_{\alpha}(b) =
  \rho_{\beta}(b)\}$.
\end{proposition}
\begin{proof}
  First, note that ${\cal L}({_{\frakB}\frakK_{\frakBo}})$
  is equal to $M(\frakB) \cap M(\frakBo) \subseteq {\cal
    L}(\frakK)$.  Let $\Phi$ and $\Psi$ be as in
  \eqref{eq:rep-canonical-triv}.  Then for each $x \in {\cal
    L}({_{\frakB}\frakK_{\frakBo}})$,
  \begin{align*} 
    x \in \End(\trivrep) &\Leftrightarrow \Psi^{*} \Phi (x
    \botensor \Id) = (x \btensor \Id) \Psi^{*}\Phi \\
    &\Leftrightarrow \Phi(x \botensor \Id)\Phi^{*} = \Psi(x
    \btensor \Id)\Psi^{*} \Leftrightarrow \rho_{\alpha}(x) =
    \rho_{\beta}(x). \qedhere
  \end{align*}
\end{proof}
\subsection{The legs of representation operators }
\label{subsection:rep-legs}

To every representation, we associate an algebra and a
space of generalized matrix elements as follows. Given a
representation $X$ on a $C^{*}$-$(\frakb,\frakbo)$-module
$\cKhd$, we put
\begin{align} \label{eq:rep-legs} \hA_{{ X}} &:= [ \bbeta{2}
  X \kalpha{2}] \subseteq {\cal L}(K) &\text{and}& & A_{{
      X}} &:= [ \langle\gamma|_{1} X |\hdelta\rangle_{1} ]
  \subseteq {\cal L}(H),
\end{align} where $\kalpha{2}, |\hdelta\rangle_{1},
\bbeta{2}, \bgamma{1}$ are defined as in subsection
\ref{subsection:pmu-rtp}.
\begin{examples} \label{examples:rep-legs}
  \begin{enumerate}
  \item For the trivial representation
    $({_{\frakB}\frakK_{\frakBo}},\trivrep)$, we have
    $\hA_{\trivrep} = [\beta^{*}\alpha]$ and $A_{\trivrep} =
    [\rho_{\beta}(\frakB)\rho_{\alpha}(\frakBo)]$. The space
    $\hA_{\trivrep}$ is related to the $C^{*}$-algebra
    $\End(\trivrep)$ (see Proposition
    \ref{proposition:rep-end}) as follows: $\End(\trivrep) =
    {\cal L}(\bKbo) \cap (\hA_{\triv_{V}})'$.  This relation
    follows from Proposition \ref{proposition:rep-end} and
    the fact that an element $x \in {\cal L}(\bKbo) =
    M(\frakB)\cap M(\frakBo)$ satisfies
    $\rho_{\alpha}(x)=\rho_{\beta}(x)$ if and only if for
    all $\eta\in\beta$ and $\xi \in \alpha$, the elements
    $\eta^{*}\xi x = \eta^{*}\rho_{\alpha}(x)\xi$ and $x
    \eta^{*}\xi=\eta^{*} \rho_{\beta}(x)\xi$ coincide.
  \item For the regular representation
    $({_{\alpha}H_{\hbeta}},V)$, the definition above is
    consistent with  definition \eqref{eq:pmu-legs}.
  \item If $({ X}_{i})_{i}$ is a family of representations
    and ${ X}=\boxplus_{i} X_{i}$, then $\hA_{{
        X}} \subseteq \prod_{i} \hA_{{ X}_{i}}$ and $A_{{
        X}} = [\bigcup_{i} A_{{ X}_{i}}]$.
  \item If $({_{\lambda \rt \gamma}(L
      \rtensor{\lambda}{\frakc}{\kappa} K)_{\lambda \rt
        \hdelta}}, \Id \ctensor X)$ is the amplification of
    a representation $(\cKhd,X)$ as in Example
    \ref{examples:rep} iv), then $\hA_{(\Id \ctensor X)} =
    \Id \rtensor{\lambda}{\frakc}{\kappa} \hA_{X}$ and
    $A_{(\Id \ctensor X)} = A_{X}$.
 \end{enumerate}
\end{examples}
With respect to tensor products, the definition of $\hA_{X}$
and $A_{X}$ behaves  as follows:
\begin{lemma} \label{lemma:rep-legs-tensor}
Let $(\cKhd,X)$ and $(\eLhf,Y)$ be
representations of $V$. Then 
\begin{align*}
 A_{(X\boxtimes
Y)}&=[A_{X}A_{Y}], & [\hA_{(X \boxtimes Y)}
|\epsilon\rangle_{2}] &= [|\epsilon\rangle_{2} \hA_{X}], &
[\hA_{(X \boxtimes Y)}^{*} |\hdelta\rangle_{1}] &=
[|\hdelta\rangle_{1} A^{*}_{X}].
\end{align*}
\end{lemma}
The proof involves commutative diagrams of a special kind:
\begin{notation} \label{notation:diagrams} We shall
  frequently prove equations for certain spaces of operators
  using commutative diagrams. In these diagrams, the
  vertexes are labelled by Hilbert spaces, the arrows are
  labelled by single operators or closed spaces of
  operators, and the composition is given by the closed
  linear span of all possible compositions of operators.  
\end{notation}
\begin{proof}[Proof of Lemma \ref{lemma:rep-legs-tensor}] 
  The following commutative diagrams shows that
  $A_{(X\boxtimes Y)}=[A_{X}A_{Y}]$:
\begin{gather*} \smalldiagram \xymatrix@C=5pt@R=15pt{ {K
      \rtensor{\hdelta}{\frakbo}{\epsilon} L
      \rtensor{\widehat\phi}{\frakbo}{\alpha} H}
    \ar[rr]^{Y_{23}} && {K
      \rtensor{\hdelta}{\frakbo}{(\epsilon \rt \alpha)} (L
      \rtensor{\epsilon}{\frakb}{\beta} H)} \ar[rr]^{X_{13}}
    \ar[rd]_{\langle\epsilon|_{2}} && {(K
      \rtensor{\hdelta}{\frakbo}{\epsilon} L)
      \rtensor{(\gamma \lt
        \epsilon)}{\frakb}{\beta} H} \ar[d]^{\langle\epsilon|_{2}} \\
    {L \rtensor{\hphi}{\frakbo}{\alpha} H} \ar[r]^{Y}
    \ar[u]^{|\hdelta\rangle_{1}} & {L
      \rtensor{\epsilon}{\frakb}{\beta} H}
    \ar[ru]_{|\hdelta\rangle_{1}}
    \ar[rd]^{\langle\epsilon|_{1}} & & {\rHsource}
    \ar[r]^{X} & {\rHrange} \ar[d]^{\bgamma{1}} \\ H
    \ar[u]^{|\hphi\rangle_{1}} \ar[rr]^{A_{Y}} && H
    \ar[rr]^{A_{X}} \ar[ru]^{|\hdelta\rangle_{1}} && H.  }
  \end{gather*} The relations concerning $\hA_{(X \boxtimes
Y)}$ follow similarly.
\end{proof}
We now collect some general properties of the spaces
introduced above.
\begin{proposition} \label{proposition:rep-legs}
  Let $(\cKhd, X)$ be a representation of $V$.  
  \begin{enumerate}
  \item The space $\hA_{X} \subseteq {\cal L}(K)$ satisfies
  \begin{gather} \label{eq:rep-legs-hax}
    \begin{gathered}
      \begin{aligned}{ }
        [ \hA_{X} \hA_{X}] &= \hA_{X}, &
        [\hA_{X}K]&=K, & [\hA_{X}\gamma] &= [\gamma
        \hA_{\trivrep}], & [\hA_{X}^{*}\hdelta] &= [\hdelta
        \hA_{\trivrep}^{*}],
      \end{aligned} \\
      [ \hA_{X} \rho_{\hdelta}(\frakB)] =
      [ \rho_{\hdelta}(\frakB) \hA_{X}] =
      \hA_{X}= [ \hA_{X} \rho_{\gamma}(\frakBo)]
      = [ \rho_{\gamma}(\frakBo) \hA_{X}],
    \end{gathered}
  \end{gather}
  and if $\hA_{X}=\hA_{X}^{*}$, then
  $(\hA_{X})_{K}^{\gamma,\hdelta}$ is a
  $C^{*}$-$(\frakb,\frakbo)$-algebra.
\item The space $A_{X} \subseteq {\cal L}(H)$ satisfies
  \begin{gather} \label{eq:rep-legs-ax}
    \begin{gathered}
      \begin{aligned}{ }
        [A_{X}\hbeta] &=\hbeta, & [A_{X}\beta] &= [\beta
        \gamma^{*}\hdelta], & [A_{X}^{*}\alpha] &= [\alpha
        \hdelta^{*}\gamma], & [A_{X}A_{V}] &= A_{V},
      \end{aligned} \\
      \begin{aligned}{ }
        [\Delta_{V}(A_{X})\kbeta{2}] &\subseteq
        [\kbeta{2}A_{X}], &
        [\Delta_{V}(A_{X}^{*})\kalpha{1}] &\subseteq
        [\kalpha{1}A_{X}^{*}],
      \end{aligned} \\
      [ A_{X} \rho_{\beta}(\frakB)] = [
      \rho_{\beta}(\frakB) A_{X}] =A_{X} = [
      A_{X} \rho_{\alpha}(\frakBo)] = [
      \rho_{\alpha}(\frakBo) A_{X}].
    \end{gathered}
  \end{gather}
  \end{enumerate}
\end{proposition}
\begin{proof} 
  i) First, $[ \hA_{X} \hA_{X}] = [
  \bbeta{2}\balpha{3} X_{12}\kalpha{3}\kalpha{2}]$
  because the diagram below commutes:
\begin{gather*} \hspace{-0.1cm} \smalldiagram
  \xymatrix@C=4pt@R=15pt{ & K \ar[d]^{\kalpha{2}}
    \ar[rr]_(0.4){ \hA_{X}} \ar `l/0pt[l] [ld]^{\kalpha{2}}
    & & K \ar@{}[dd]|{\scriptstyle\mathrm{(C)}}
    \ar[rd]^(0.65){\kalpha{2}} \ar[rr]_(0.6){\hA_{X}} & & K & \\
    {\rHsource \hspace{-0.6cm}} \ar
    `d[dd]^(0.75){\kalpha{3}} [ddr] & {\rHsource}
    \ar[r]^(0.55){X} \ar[d]^{\kalpha{2}} & {\rHrange}
    \ar[ru]^(0.35){\bbeta{2}} \ar[rd]_(0.35){ \kalpha{2}} &
    & {\rHsource} \ar[r]^(0.45){X} & {\rHrange}
    \ar[u]^{\bbeta{2}} & {\hspace{-0.6cm}\rHrange} \ar
    `u[u]^{\bbeta{2}} [ul] \\ & {\rHfive}
    \ar[r]_(0.75){X_{13}} \ar[d]^{V_{23}^{*}} &&
    {\hspace{-1cm}\rHfour\hspace{-1cm}}
    \ar[ru]_(0.65){\bbeta{3}}
    \ar@{}[d]|(0.4){\scriptstyle\mathrm{(P)}} &
    \ar[r]_(0.25){X_{12}} & {\rHthree} \ar[u]^{\bbeta{3}} &
    \\ & {\rHone} \ar[rrrr]^{ X_{12}} & &&& {\rHtwo}
    \ar[u]^{V_{23}} \ar `r/0pt[r] [ruu]^(0.25){\balpha{3}} &
  }
      \end{gather*} Indeed, cell (C) commutes because for
all $\xi \in \alpha$, $\eta,\eta'\in \beta$, $\zeta \in K$,
      \begin{gather} \label{eq:pmu-diagram-commutes}
        \begin{aligned} |\xi\rangle_{\leg{2}} \langle
          \eta'|_{\leg{2}} (\zeta \tl \eta) &=
          \rho_{\gamma}\big(\eta'{}^{*}\eta\big)\zeta \tl
          \xi = \rho_{(\gamma \lt
            \alpha)}\big(\eta'{}^{*}\eta\big) (\zeta \tl \xi)
          = \langle \eta'|_{\leg{3}} |\xi\rangle_{\leg{2}}
          (\zeta \tl \eta),
        \end{aligned}
      \end{gather} cell (P) is  diagram
\eqref{eq:pmu-pentagon}, and the other cells commute
 by definition of $\hA_{X}$ and because of
\eqref{eq:pmu-intertwine}.  Next, $[
\bbeta{2}\balpha{3}X_{12}\kalpha{3}\kalpha{2}] =
\hA_{X}$ because the following diagram commutes:
\begin{gather*} \smalldiagram \xymatrix@C=30pt@R=15pt{ & K
    \ar `l/0pt[l] [ld]^{{\kalpha{2}}} \ar[r]_{\hA_{X}}
    \ar[d]^{{\kalpha{2}}} & K \\ {\rHsource} \ar
    `d/0pt[d]^{\kalpha{3}} [dr] \ar[r]^(0.55){\rho_{(\hdelta
        \rt \hbeta)}(\frakB)} & {\rHsource} \ar[r]^{ V} &
    {\rHrange} \ar[u]^{\bbeta{2}} \\ & {\rHone} \ar[u]^{
      \balpha{3}} \ar[r]^{V_{12}} & {\rHtwo}
    \ar[u]^{\balpha{3}} }
\end{gather*}
 We  prove some of the other equations in
\eqref{eq:rep-legs-hax}; the remaining ones  follow
  similarly.
\begin{align*}
  [\hA_{X}K]&=[\bbeta{2}X\kalpha{2}K] =
  [\bbeta{2}X(\rHsource)] = [\bbeta{2}(\rHrange)]=K, \\
  [\hA_{X}\gamma]&=[\bbeta{2}X\kalpha{2}\gamma] =[\bbeta{2}
  \kgamma{1}\alpha] = [\gamma \beta^{*}\alpha] = [\gamma
  \hA_{\trivrep}], \\
  [\hA_{X}\rho_{\hdelta}(\frakB)] &= [ \bbeta{2} X
  \kalpha{2} \rho_{\hdelta}(\frakB)] = [ \bbeta{2} X |\alpha
  \frakB\rangle_{2} ] = \hA_{X}, \\
  [ \rho_{\hdelta}(\frakB)\hA_{X}] &= [
  \rho_{\hdelta}(\frakB) \bbeta{2} X \kalpha{2} ] = [
  \bbeta{2} (\rho_{\hdelta}(\frakB) \botensor \Id) X
  \kalpha{2} ] \\ &= [ \bbeta{2} X (\Id \botensor
  \rho_{\beta}(\frakB)) \kalpha{2} ] = [ \bbeta{2} X
  |\rho_{\beta}(\frakB)\alpha\rangle_{2}] = \hA_{X}.
  \end{align*}

  ii) First,
  $[A_{X}\hbeta]=[\langle\gamma|_{1}X|\hdelta\rangle_{1}\hbeta]
  = [\langle\gamma|_{1}|\gamma\rangle_{1} \hbeta] =
  [\rho_{\beta}(\frakB)\hbeta]=\hbeta$ by
  \eqref{eq:rep-intertwine}, and similar calculations show
  that $[A_{X}\beta] = [\beta \gamma^{*}\hdelta]$ and
  $[A_{X}^{*}\alpha] = [\alpha \hdelta^{*}\gamma]$.  Next,
  $[A_{X}A_{V}]=A_{X \boxtimes V} = A_{V}$ by Lemma
  \ref{lemma:rep-legs-tensor}, Lemma \ref{lemma:rep-absorb},
  and Example \ref{examples:rep-legs} iv).  The equations in
  the last line of \eqref{eq:rep-legs-ax} follow from
  similar calculations as for $\hA_{X}$.

  Finally, let us prove the equations in the middle line of
  \eqref{eq:rep-legs-ax}.  Since $A_{X} \subseteq {\cal
    L}(H_{\hbeta}) \subseteq \rho_{\hbeta}(\frakB)'$,
  $\Delta_{V}(A_{X})$ is well defined.  Consider the
  commutative diagram
  \begin{gather*} \smalldiagram \xymatrix@R=15pt@C=20pt{
      {\Hrange} \ar[r]_{V^{*}} \ar[d]^{|\hdelta\rangle_{1}}
      & {\Hsource} \ar[r]_{A_{X} \botensor 1}
      \ar[d]^{|\hdelta\rangle_{1}} & {\Hsource} \ar[r]_{V} &
      {\Hrange.} \\ {\rHfive} \ar[r]^(0.55){V_{23}^{*}}& {\rHone}
      \ar[r]^{X_{12}} & {\rHtwo} \ar[r]^{V_{23}}
      \ar[u]^{\bgamma{1}}& {\rHthree} \ar[u]^{\bgamma{1}} }
      \end{gather*} 
      Since the composition on top is $\Delta_{V}(A_{X})$
      and the composition on the bottom is $X_{12}X_{13}$,
      the following diagram commutes and shows that
      $[\Delta_{V}(A_{X})\kbeta{2}]=[A_{X}\kbeta{2}]$:
      \begin{align*} 
        \smalldiagram \xymatrix@R=15pt@C=35pt{
&&& \\ {H} \ar `u/0pt[u] `r[urrr]_{A_{X}} [rrr]
\ar[r]^{|\hdelta\rangle_{1}} \ar[d]^{\kbeta{2}} &
{\rHsource} \ar[r]^{X} \ar[d]^{\kbeta{3}} & {\rHrange}
\ar[r]^{\bgamma{1}} \ar[d]^{\kbeta{3}} & {H}
\ar[d]^{\kbeta{2}} \\ {\Hrange} \ar `d/0pt[d]
`r[drrr]^{\Delta_{V}(A_{X})} [rrr]
\ar[r]^(0.4){X_{13}|\hdelta\rangle_{1}} & {\rHfour}
\ar[r]^{X_{12}} & {\rHthree} \ar[r]^{\bgamma{1}} & {\Hrange}
\\ &&& }
\end{align*} 
A similar argument shows that $[ \Delta_{V}(A_{X}^{*})
\kalpha{1}] = [ \kalpha{1}A_{X}^{*}]$.
\end{proof}

\begin{definition}
  We call the $C^{*}$-pseudo-multiplicative unitary $ V$
  {\em well-behaved} if $\hA_{X}=\hA_{X}^{*}$ for every
  representation $X$ of $V$ and $\hA_{Y}=\hA_{Y}^{*}$ for
  every representation $Y$ of $V^{op}$.
\end{definition}
\begin{proposition}
  If $V$ is well-behaved, then
  $((\hA_{V})^{\alpha,\hbeta}_{H},\hDelta_{V})$ and
  $((A_{V})^{\beta,\alpha}_{H},\Delta_{V})$ are concrete Hopf
  $C^{*}$-bimodules.
\end{proposition}
\begin{proof}
  By Proposition \ref{proposition:pmu-legs}, the assumption
  implies that $(\hA_{V})^{\alpha,\hbeta}_{H}$ and
  $(A_{V})^{\beta,\alpha}_{H}=(\hA_{V^{op}})^{\beta,\alpha}_{H}$
  are $C^{*}$-algebras, that $\Delta_{V}(A_{V}) \subseteq
  A_{V} \fibre{\alpha}{\frakb}{\beta} A_{V}$ and similarly
  that $\hDelta_{V}(\hA_{V})\subseteq \hA_{V}
  \fibre{\hbeta}{\frakbo}{\alpha} \hA_{V}$. Now, the
  assertion follows from Lemma
  \ref{lemma:pmu-delta-hopf}.
\end{proof}

\subsection{The universal Banach algebra of representations }
 \label{subsection:rep-banach}

Every representation of $V$ induces a representation of the
convolution algebra $\tilde \Omega_{\beta,\alpha}$
introduced in subsection \ref{subsection:legs-pmu}
as follows.
\begin{proposition} \label{proposition:pmu-convolution-rep}
  \begin{enumerate}
  \item Let $(\cKhd,X)$ be a representation of $V$. Then
    there exists a contractive algebra homomorphism
    $\hpi_{X} \colon \tilde \Omega_{\beta,\alpha} \to
    \hA_{X}$ such that
    \begin{align} \label{eq:pmu-convolution-rep}
      \hpi_{X}(\omega_{\xi,\xi'}) = \sum_{n}
      \langle \xi_{n}|_{2}X|\xi'_{n}\rangle_{2} \quad
      \text{for all } \xi \in \beta^{\infty}, \xi' \in
      \alpha^{\infty}.
    \end{align}
  \item Let $(\cKhd,X)$ and $(\eLhf,Y)$ be representations
    of $V$ and let $T \in {\cal L}_{(s)}(\cKhd,\eLhf)$. Then
    $T$ is a (semi-)morphism from $(\cKhd,X)$ to $(\eLhf,Y)$
    if and only if $T$ intertwines $\hpi_{X}$ and $\hpi_{Y}$
    in the sense that $T\hpi_{X}(\omega)=\hpi_{Y}(\omega)T$
    for all $\omega \in \tilde \Omega_{\beta,\alpha}$.
  \end{enumerate}
\end{proposition}
\begin{proof}
  i) The sum on the right hand side of
  \eqref{eq:pmu-convolution-rep} depends on
  $\omega_{\xi,\xi'}$ but not on $\xi,\xi'$ because
  $\eta^{*}\left(\sum_{n}\langle
    \xi_{n}|_{2}X|\xi'_{n}\rangle_{2}\right)\eta' =
  \omega_{\xi,\xi'}(\langle\eta|_{1}X|\eta'\rangle_{1})$ for
  all $\eta \in \gamma,\eta'\in \hdelta$.  Thus, $\hpi_{X}$
  is well defined by \eqref{eq:pmu-convolution-rep}. It is a
  homomorphism because for all $\omega,\omega' \in \tilde
  \Omega_{\beta,\alpha}$, $\eta\in \gamma,\eta' \in
  \hdelta$,
  \begin{align*}
    \eta^{*}\hpi_{X}(\omega)\hpi_{X}(\omega')\eta' &=
    (\omega \boxtimes \omega')(\langle
    \eta|_{1}X_{12}X_{13}|\eta'\rangle_{1})  \\
    &= (\omega \boxtimes \omega')(\langle
    \eta|_{1}V_{23}X_{12}V_{23}^{*}|\eta'\rangle_{1})
    \\
    &=(\omega \boxtimes \omega')\big(V(\langle \eta|_{1}
    X|\eta'\rangle_{1} \rtensor{\hbeta}{\frakbo}{\alpha}
    \Id)V^{*}\big) \\
    & = (\omega \ast \omega')(\langle
    \eta|_{1}X|\eta'\rangle_{1}) = \eta^{*}\hpi_{X}(\omega
    \ast \omega')\eta'.
  \end{align*}
  ii) Straightforward.
\end{proof}
For later use, we note the following formula:
\begin{lemma} \label{lemma:pmu-convolution-hdelta} $\hDelta_{V}(\hpi_{V}(\omega)) = \hpi_{V \boxtimes
    V}(\omega)$ for each $\omega \in \tilde
  \Omega_{\beta,\alpha}$.
\end{lemma}
\begin{proof}
  For all $\xi\in \alpha^{\infty}$ and $\xi' \in
  \beta^{\infty}$, we have
  $\hDelta_{V}(\hpi_{V}(\omega_{\xi',\xi})) = \sum_{n}
  \langle \xi'|_{3} V_{12}V_{23}V_{12}^{*} |\xi\rangle_{3} =
  \sum_{n} \langle \xi'|_{3}V_{13}V_{23}|\xi\rangle_{3} =
  \hpi_{V\boxtimes V}(\omega_{\xi',\xi})$.
\end{proof}
Denote by $\hA_{(u)}$  the separated completion of
$\tilde \Omega_{\beta,\alpha}$ with respect to the seminorm
\begin{align*}
  |\omega| := \sup \{ \|\hpi_{X}(\omega)\| \mid (X \text{ is
    a representation of } V\} \text{ for each } \omega \in
  \tilde \Omega_{\beta,\alpha}
\end{align*}
and by $\hpi_{(u)} \colon \tilde \Omega_{\beta,\alpha} \to
\hA_{(u)}$ the natural map.  
\begin{proposition} \label{proposition:pmu-convolution-banach}
  \begin{enumerate}
  \item There exists a unique algebra structure on
    $\hA_{(u)}$ such that $\hA_{(u)}$ is a Banach algebra
    and $\hpi_{(u)}$ an algebra homomorphism.
  \item For every representation $X$ of $V$, there
    exists a unique algebra homomorphism $\hpi^{(u)}_{X}
    \colon \hA_{(u)} \to \hA_{X}$ such that $\hpi^{(u)}_{X}
    \circ \hpi_{(u)} = \hpi_{X}$.
  \item If $V$ is well-behaved, then the Banach algebra
    $\hA_{(u)}$ carries a unique involution turning it into
    a $C^{*}$-algebra such that $\pi^{(u)}_{X}$ is a
    $*$-homomorphism for every representation $X$ of $V$.
  \end{enumerate}
\end{proposition}
\begin{proof} 
  Assertions i) and ii) follow from routine arguments. Let
  us prove iii).  For each $\omega \in \tilde
  \Omega_{\beta,\alpha}$ and $\epsilon > 0$, choose a
  representation $X_{(\omega,\epsilon)}$ such that
  $\|\hpi_{X_{(\omega,\epsilon)}}(\omega)\| >
  |\omega|-\epsilon$. Let $X :=
  \boxplus_{\omega,\epsilon} X_{(\omega,\epsilon)}$,
  where the sum is taken over all $\omega \in
  \tilde\Omega_{\beta,\alpha}$ and $\epsilon > 0$. Then
  evidently $\hpi^{(u)}_{X} \colon \hA_{(u)} \to \hA_{X}$ is
  an isometric  isomorphism of Banach algebras. We can
  therefore define an involution on $\hA_{(u)}$ such that
  $\hpi^{(u)}_{X}$ becomes a $*$-isomorphism.  Now, let $Y$
  be a representation $V$. Then $X \boxplus Y$ is a
  representation again, and we have a commutative diagram
  \begin{align*} \xymatrix{ & { \hA_{(u)} }
      \ar[ld]_{\hpi^{(u)}_{X}} \ar[d]|{\hpi^{(u)}_{X \boxplus Y}}
      \ar[rd]^{\hpi^{(u)}_{Y}} & \\ {\hA_{X}} & {\hA_{X \boxplus
          Y}} \ar[l]^{p_{X}} \ar[r]_{p_{Y}} & {\hA_{Y},} }
  \end{align*} 
  where $p_{X}$ and $p_{Y}$ are the natural maps. Since
  $\hpi^{(u)}_{X}$ is isometric, so is
  $\hpi^{(u)}_{X\boxplus Y}$. But $\hpi^{(u)}_{X\boxplus Y}$
  also has dense image and therefore is surjective, whence
  $p_{X}$ is injective. Since $\hpi^{(u)}_{X}$, $p_{X}$,
  $p_{Y}$ are $*$-homomorphisms, so is
  $\hpi^{(u)}_{X\boxplus Y}$ and hence also
  $\hpi_{Y}^{(u)}$.
\end{proof}

\subsection{Universal representations and the universal Hopf
$C^{*}$-bimodule }

\label{subsection:rep-universal}

If the unitary $V$ is well-behaved, then the universal
Banach algebra $\hA_{(u)}$ constructed above can be equipped
with the structure of a semi-normal Hopf $C^{*}$-bimodule,
where the comultiplication corresponds to the tensor product
of representations of $V$. The key idea is to identify
$\hA_{(u)}$ with the $C^{*}$-algebra associated to a
 representation that is universal in the following sense.
\begin{definition}  A representation $(\cKhd,X)$ of
  $V$  is {\em universal} if for every representation
  $(\eLhf,Y)$ and every $\xi \in \epsilon$, $\zeta \in L$,
  $\eta \in \hat\phi$, there exists a semi-morphism $T$
  from $(\cKhd,X)$ to $(\eLhf,Y)$ that is a partial
  isometry and satisfies $\xi \in T\gamma$, $\zeta \in TK$,
  $\eta \in T\hat\delta$.
\end{definition}
\begin{remark}
  Evidently, every universal representation is a generator
  \cite{maclane} of $\bfcs\bfrep_{V}^{s}$ in the categorical
  sense.
\end{remark}
We shall use a cardinality argument to show that
$\bfcs\bfrep_{V}^{s}$ has a universal representation.  Given
a topological space $X$ and a cardinal number $c$, let us
say that $X$ is {\em $c$-separable} if $X$ has a dense
subset of cardinality less than or equal to $c$. Let $\omega
:=|\naturals|$.  Let us also say that a {\em
  subrepresentation} of a representation $(\cKhd,X)$ of $V$
is a $C^{*}$-$(\frakb,\frakbo)$-module $\eLhf$ such that $L
\subseteq K$, $\epsilon \subseteq \gamma$, $\hphi \subseteq
\hdelta$, and $X(L \rtensor{\hphi}{\frakbo}{\alpha} H)= L
\rtensor{\epsilon}{\frakbo}{\beta} H$.
\begin{lemma} \label{lemma:pmu-universal}
  Let $(\cKhd, X)$ be a representation, $c,d$ cardinal
  numbers, $K_{0} \subseteq K$, $\gamma_{0} \subseteq
  \gamma$, $\hdelta_{0} \subseteq \hdelta$ $c$-separable
  subsets, and assume that the spaces
  $\frakK,\frakB,\frakBo,\alpha,\beta,\hbeta$ are
  $d$-separable. Put $e:= \omega\sum_{n=0}^{\infty} 
  d^{n}$.  Then there exists a subrepresentation $\eLhf$ of
  $(\cKhd,X)$ such that $\gamma_{0} \subseteq \epsilon$,
  $\hdelta_{0} \subseteq \hphi$, $K_{0} \subseteq L$ and
  such that $L,\gamma,\hphi$ are $e(c+1)$-separable.
\end{lemma}
\begin{proof}
  Replacing $K_{0}$, $\gamma_{0}$, $\hdelta_{0}$ by dense
  subsets, we may assume that each of these sets has
  cardinality less than or equal to $c$. Moreover, replacing
  $\gamma_{0}$ and $\hdelta_{0}$ by larger sets, we may
  assume that $\frakB = [\gamma_{0}^{*}\gamma_{0}]$,
  $\frakBo=[\hdelta_{0}^{*}\hdelta_{0}]$, and $|\gamma_{0}|
  \leq c + \omega d$, $|\hdelta_{0}| \leq c + \omega d$.

  Now, we can choose inductively $K_{n}\subseteq K$,
  $\hdelta_{n}\subseteq \hdelta$, $\gamma_{n}\subseteq
  \gamma$ for  $n =1,2,\ldots$ such that for $n=0,1,2,
  \ldots$, the following conditions hold:
  \begin{enumerate}
  \item $K_{n+1}$ is large enough so that $K_{n} +
    \hdelta_{n}\frakK + \gamma_{n}\frakK \subseteq
    [K_{n+1}]$, but small enough so that
    \begin{align*}
  |K_{n+1}| \leq
    |K_{n}| + \omega d(|\gamma_{n}| +|\hdelta_{n}|)
    \leq \omega d ( | K_{n}|+|\gamma_{n}|+|\hdelta_{n}|);    
    \end{align*}
  \item $\gamma_{n+1}$ is large enough so that
    \begin{gather*}
      \gamma_{n} + \gamma_{n}\frakB +
      \rho_{\hdelta}(\frakBo)\gamma_{n} \subseteq
      [\gamma_{n+1}], \quad K_{n} \subseteq
      [\gamma_{n+1}\frakK], \\
      X\kalpha{2}\gamma_{n} \subseteq
      [|\gamma_{n+1}\rangle_{1} \alpha], \quad X|\hdelta_{n}
      \rangle_{1} \hbeta \subseteq [|\gamma_{n+1}\rangle_{1}
      \hbeta],
    \end{gather*}
    but small enough so that
    \begin{align*}
      |\gamma_{n+1}| \leq |\gamma_{n}|(1 + \omega d) +
      \omega |K_{n}| + \omega d |\gamma_{n}| + \omega d
      |\hdelta_{n}| \leq \omega d (|K_{n}| + |\gamma_{n}| +
      |\hdelta_{n}|);
    \end{align*}
  \item $\hdelta_{n+1}$ is large enough so that
    \begin{gather*}
      \hdelta_{n} + \hdelta_{n} \frakBo +
      \rho_{\gamma}(\frakBo)\hdelta_{n}\subseteq
      [\hdelta_{n+1}], \quad K_{n} \subseteq
      [\hdelta_{n+1}\frakK], \quad X|\hdelta_{n}\rangle_{1}
      \beta \subseteq [\kbeta{2}\hdelta_{n+1}],
    \end{gather*}
    but small enough so that
    \begin{align*}
      |\hdelta_{n+1}| \leq |\hdelta_{n}| ( 1 + \omega d) +
      \omega |K_{n}| + \omega d |\hdelta_{n}| \leq \omega d
      (|K_{n}| + |\hdelta_{n}|).
    \end{align*}
\end{enumerate}
Since $|K_{0}|+|\gamma_{0}|+|\hdelta_{0}|=3 c + 2 \omega d$,
we can conclude inductively that for all $n=0,1,2,\ldots$,
\begin{align*}
  |K_{n+1}| + |\gamma_{n+1}|+|\hdelta_{n+1}| \leq \omega d
  (|K_{n}|+|\gamma_{n}|+|\hdelta_{n}|) \leq (\omega d)^{n+1}
  (c + \omega d).
\end{align*}
Therefore, the spaces $L:=\left[\bigcup_{n}K_{n}\right]$,
$\epsilon := \left[\bigcup_{n} \gamma_{n}\right]$, $\hphi :=
\left[\bigcup_{n} \hdelta_{n}\right]$ are
$e(c+1)$-separable. By construction, $\eLhf$ is a
subrepresentation of $(\cKhd,X)$.
\end{proof}

\begin{proposition} \label{proposition:pmu-universal}
  There exists a universal representation of $V$.
\end{proposition}
\begin{proof}
  Let $d$ and $e$ be as in Lemma \ref{lemma:pmu-universal}.
  Then there exists a set ${\cal X}$ of representations of
  $V$ such that every representation $(\eLhf,Y)$ of $V$,
  where the underlying Hilbert space $L$ is $e$-separable,
  is isomorphic to some representation in ${\cal X}$.  Using
  Lemma \ref{lemma:pmu-universal}, one easily verifies that
  the direct sum $\boxplus_{X \in {\cal X}} X$ is a
  universal representation.
\end{proof}
\begin{theorem}
  Let $V$ be a well-behaved $C^{*}$-pseudo-multiplicative
  unitary and let $(\cKhd,X)$ be a universal representation
  of $V$.
 \begin{enumerate}
 \item The $*$-homomorphism $\hpi^{(u)}_{X} \colon \hA_{(u)}
   \to \hA_{X}$ is an isometric isomorphism.
 \item If $(\eLhf,Y)$ is a representation of $V$, then there
   exists a jointly semi-normal morphism $\hpi_{X,Y}$ of
   $C^{*}$-$(\frakb,\frakbo)$-algebras
   $(\hA_{X})_{K}^{\gamma,\hdelta}$,
   $(\hA_{Y})_{L}^{\epsilon,\hphi}$ such that
   $\hpi^{(u)}_{Y} = \hpi_{X,Y} \circ \hpi^{(u)}_{X}$.
 \item Let $\hDelta_{X}:=\hpi_{X,X \boxtimes X}$. Then
   $((\hA_{X})_{K}^{\gamma,\hdelta},\hDelta_{X})$ is a
   semi-normal Hopf $C^{*}$-bimodule.
 \item $\hpi_{X,V}$ is a morphism of the semi-normal Hopf
   $C^{*}$-bimodules
   $((\hA_{X})_{K}^{\gamma,\hdelta},\hDelta_{X})$ and
   $((\hA_{V})^{\alpha,\hbeta}_{H},\hDelta_{V})$.
 \item Let $(\eLhf,Y)$ be a universal representation of $V$
   and define $\hDelta_{Y}$ similarly as $\hDelta_{X}$. Then
   $\hpi_{X,Y}$ is an isomorphism of the semi-normal Hopf
   $C^{*}$-bimodules
   $((\hA_{X})_{K}^{\gamma,\hdelta},\hDelta_{X})$ and
   $((\hA_{Y})_{L}^{\epsilon,\hphi},\hDelta_{Y})$.
  \end{enumerate}
\end{theorem}
\begin{proof}
  i) Let $\omega \in \tilde \Omega_{\beta,\alpha}$, let
  $(\eLhf,Y)$ be a representation of $V$, and let $\zeta \in
  L$. Since $X$ is universal, there exists a 
  semi-morphism $T$ from $X$ to $Y$ that is a
  partial isometry and satisfies $\zeta \in TL$. Then by
  Proposition \ref{proposition:pmu-convolution-rep}, $\|
  \hpi_{Y}(\omega)\zeta\| = \| \hpi_{Y}(\omega)TT^{*}\zeta\|
  = \| T\hpi_{X}(\omega)T^{*}\zeta\| \leq
  \|\hpi_{X}(\omega)\| \|\zeta\|$. Since $Y$ and $\zeta$
  were arbitrary, we can conclude that $\|\hpi_{(u)}(\omega)\|
  \leq \|\hpi_{X}(\omega)\|$ and hence that
  $\hpi^{(u)}_{X}$ is isometric.

  ii) We have to show that $\hpi_{X,Y}:=\hpi^{(u)}_{Y} \circ
  \hpi^{(u)}_{X}{}^{-1}$ is a jointly semi-normal morphism
  of $C^{*}$-$(\frakb,\frakbo)$-algebras. Let $\xi \in
  \epsilon$ and $\eta \in \hphi$. Since $X$ is universal,
  there exists a semi-morphism $T$ from $X$ to $Y$ such that
  $\xi \in T\gamma$ and $\eta \in T\hdelta$. By Proposition
  \ref{proposition:pmu-convolution-rep}, $\hpi_{Y}(\omega) T
  = \hpi_{X}(\omega)T$ for all $\omega \in \tilde
  \Omega_{\beta,\alpha}$, and hence $T \in {\cal
    L}^{\hpi_{X,Y}}_{s}(\cKhd,\eLhf)$. The claim follows.

  iii) We need to show that $\hDelta_{X}$ is coassociative.
  We shall prove that $(\hDelta_{X} \bofibre \Id) \circ
  \hDelta_{X} = \hpi_{X,X\boxtimes X\boxtimes X}$, and a
  similar argument shows that $(\Id \bofibre \hDelta_{X}) \circ
  \hDelta_{X} = \hpi_{X,X\boxtimes X\boxtimes X}$. Let $S,T$
  be semi-morphisms from $X$ to $X\boxtimes X$. Then $R:=(S
  \botensor \Id) \circ T$ is a semi-morphism from $X$ to $X
  \boxtimes X \boxtimes X$, and a generalization of
  Proposition \ref{proposition:pmu-convolution-rep} ii)
  shows that for each $\omega \in \tilde
  \Omega_{\beta,\alpha}$,
  \begin{align*}
    (\hDelta_{X} \bofibre \Id)(
    \hDelta_{X}(\hpi_{X}(\omega))) \cdot R &= (S \botensor
    \Id) \cdot \hDelta_{X}(\hpi_{X}(\omega) ) \cdot T \\ &= R
    \cdot \hpi_{X}(\omega) = \hpi_{X,X\boxtimes X\boxtimes
      X}(\omega) \cdot R.
  \end{align*}
  Since $S$ and $T$ were arbitrary and $X$ is universal, we
  can conclude that $ (\hDelta_{X} \bofibre \Id) \circ
  \hDelta_{X}(\hpi_{X}(\omega)) = \hpi_{X,X\boxtimes
    X\boxtimes X}(\omega)$.

  iv) We have to show that $(\hpi_{X,V} \bofibre \hpi_{X,V})
  \circ \hDelta_{X} = \hDelta_{V} \circ \hpi_{X,V}$. Let
  $\omega \in \tilde \Omega_{\beta,\alpha}$ and $S \in {\cal
    L}^{\hpi_{X,V}}(K_{\hdelta},H_{\hbeta})$, $T \in {\cal
    L}^{\hpi_{X,V}}(K_{\gamma},H_{\alpha})$. Then $R:=(S
  \botensor T)$ satisfies $R(X \boxtimes X)=(V \boxtimes
  V)R$, and using Lemma \ref{lemma:pmu-convolution-hdelta},
  we find
\begin{align*}
  (\hpi_{X,V} \bofibre \hpi_{X,V})
  (\hDelta_{X}(\hpi_{X}(\omega)) \cdot R &= R \cdot
  \hDelta_{X}(\hpi_{X}(\omega)) \\ &= R \cdot
  \hpi_{X\boxtimes X}(\omega) = \hpi_{V \boxtimes V}(\omega)
  \cdot R =\hDelta_{V}(\hpi_{V}(\omega))\cdot R.
\end{align*}
Since $S$ and $T$ were arbitrary, we can conclude $ (\hpi_{X,V}
\bofibre \hpi_{X,V}) (\hDelta_{X}(\hpi_{X}(\omega)) =
\hDelta_{V}(\hpi_{V}(\omega))$.

v) We have to show that $(\hpi_{X,Y} \bofibre \hpi_{X,Y})
\circ \hDelta_{X} = \hDelta_{Y} \circ \hpi_{X,Y}$.  Let
$\omega \in \tilde \Omega_{\beta,\alpha}$ and $S \in {\cal
  L}^{\hpi_{X,Y}}(K_{\hdelta},H_{\hphi})$, $T \in {\cal
  L}^{\hpi_{X,Y}}(K_{\gamma},H_{\epsilon})$. Then $R:=(S
\botensor T)$ satisfies $R(X \boxtimes X)=(Y \boxtimes Y)R$,
and using Proposition \ref{proposition:pmu-convolution-rep},
we find
\begin{align*}
  (\hpi_{X,Y} \ast \hpi_{X,Y})
  (\hDelta_{X}(\hpi_{X}(\omega)) \cdot R &= R \cdot
  \hDelta_{X}(\hpi_{X}(\omega)) \\ &= R \cdot
  \hpi_{X\boxtimes X}(\omega) = \hpi_{Y\boxtimes
    Y}(\omega) \cdot R = \hDelta_{Y}(\hpi_{Y}(\omega))
  \cdot R.
\end{align*}
Since $S$ and $T$ were arbitrary, we can conclude  $(\hpi_{X,Y}
\ast \hpi_{X,Y}) (\hDelta_{X}(\hpi_{X}(\omega)) =
(\hDelta_{Y}(\hpi_{Y}(\omega)))$.
\end{proof}

\subsection{Corepresentations and $W^{*}$-representations } 
\label{subsection:rep-corep}

The notion of a representation of a
$C^{*}$-pseudo-multiplicative unitary can be dualized so
that one obtains the notion of a corepresentation, and
adapted to $W^{*}$-modules instead of $C^{*}$-modules so
that one obtains the notion of a $W^{*}$-representation. 
We briefly summarize the main definitions
and properties of these concepts.

  A {\em corepresentation} of $V$ consists of a
  $C^{*}$-$(\frakb,\frakbo)$-module $\cKd$ and of a unitary $X
  \colon H \rtensor{\hbeta}{\frakbo}{\gamma} K \to H
  \rtensor{\alpha}{\frakb}{\delta} K$ that satisfies
  $X(\alpha \lt \gamma) = \alpha \rt \gamma$, $X(\beta \lt
  \gamma) = \beta \lt \delta$, $X(\hbeta \rt \delta) =
  \hbeta \lt \delta$ and makes the following diagram
  commute:
\begin{gather} \label{eq:rep-corep-pentagon} \smalldiagram
  \begin{gathered} \xymatrix@R=15pt@C=20pt{ {\crHone}
      \ar[r]^{V_{12}} \ar[d]^{X_{23}} &
      {\crHtwo} \ar[r]^{X_{23}} & {\crHthree,} \\
      {\crHfive} \ar[rr]^{X_{13}} && {\crHfour}
      \ar[u]^{V_{12}} }
  \end{gathered}
\end{gather}
where $V_{12},X_{13},X_{23}$ are defined similarly as in
subsection \ref{subsection:rep-category}.  A {\em
  (semi-)morphism} of corepresentations $(\cKd,X)$ and
$(\eLf,Y)$ is an operator $T \in {\cal L}_{(s)}(\cKd,\eLf)$
satisfying $Y(\Id \botensor T) = (\Id \btensor
T)X$. Evidently, the class of all corepresentations $V$ with
all (semi-)morphisms forms a category
$\bfcs\bfcorep^{(s)}_{V}$.  One easily verifies that there
exists an isomorphism of categories $\bfcs\bfcorep^{(s)}_{V}
\to \bfcs\bfrep^{(s)}_{V^{op}}$ given by $(\cKd,X) \mapsto
(\cKd,\Sigma Y^{*}\Sigma)$ and $T \mapsto T$. Thus, all
constructions and results on representations carry over to
corepresentations. In particular, we can equip
$\bfcs\bfcorep_{V}$ with the structure of $C^{*}$-tensor
category and $\bfcs\bfcorep^{s}_{V}$ with the structure of
a tensor category.

Replacing $\frakb$ by the $W^{*}$-base $\wclose{\frakb}$ and
$C^{*}$-modules by $W^{*}$-modules (see \cite{timmer:fiber})
in definition \ref{definition:pmu-rep}, we obtain the notion
of a $W^{*}$-representation. If we reformulate this notion
using correspondences instead of $W^{*}$-modules, the
definition reads as follows.  A {\em $W^{*}$-representation}
of $V$ consists of a Hilbert space $K$ with two commuting
nondegenerate and normal representations $\hsigma \colon
\frakB \to {\cal L}(K)$, $\sigma \colon \frakBo \to {\cal
  L}(K)$ and a unitary $X \in {\cal L}(\wrHsource,
\wrHrange)$ that satisfies $X(\sigma(b^{\dag}) \tl \Id) =
(\Id \tl \rho_{\alpha}(b^{\dag}))X$, $X(\Id \tl
\rho_{\beta}(b)) = (\hsigma(b) \tl \Id)X$, $X(\Id \tl
\rho_{\hbeta}(b)) = (\Id \tl \rho_{\hbeta}(b))X$ for all
$b^{\dag} \in \frakBo$, $b \in \frakB$ and that makes the
following diagram commute,
  \begin{gather*}
    \xymatrix@R=10pt@C=15pt{ {\wrHone} \ar[r]^{X \tl \Id}
      \ar[d]^{\Id \tl V} & {\wrHtwo} \ar[r]^{\Id \tl V} &
      {\wrHthree,} \\ {\wrHfive} \ar[d]^{\Sigma_{23}}
      && {\wrHfour} \ar[u]^{X \tl \Id} \\
      {\wrHfourlt} \ar[rr]^{X \tl \Id} && {\wrHfourrt} 
\ar[u]^{\Sigma_{23}}}
  \end{gather*} 
  where $\Sigma_{23}$ denotes the isomorphisms that
  exchange the second and the third factor in the iterated
  internal tensor products.  Here, normality of $\sigma,
  \hat \sigma$ means that they extend to the von Neumann
  algebras generated by $\frakB$ and $\frakBo$,
  respectively, in ${\cal L}(\frakK)$. A {\em morphism} of
    $W^{*}$-representations $(K,\sigma,\hat\sigma,X)$ and
  $(L,\tau,\hat\tau,Y)$ is an operator $T \in {\cal L}(K,L)$
  that intertwines $\sigma$ and $\tau$ on one side and
  $\hat\sigma$ and $\hat\tau$ on the other side, and
  satisfies $Y(T \tl \Id)=(T \tl \Id)X$. Evidently, the
  class of all $W^{*}$-representations of $V$ forms a
  category $\bfws\bfrep_{V}$.  One easily verifies that
  there exists a functor $\bfcs\bfrep^{(s)}_{V} \to
  \bfws\bfrep_{V}$ given by $(\cKhd,X) \mapsto
  (K,\rho_{\gamma},\rho_{\hdelta},X)$ and $T \mapsto
  T$. Using a relative tensor product of $W^{*}$-modules
  (see \cite{timmer:fiber}), one can equip $\bfws\bfrep_{V}$
  with the structure of a $C^{*}$-tensor category similarly
  like $\bfcs\bfrep_{V}$ and finds that the functors above
  preserve the tensor product.  Finally, one can consider
  $W^{*}$-corepresentations of $V$ which are defined in a
  straightforward manner.

\subsection{Representations of groupoids and of the
  associated unitaries }
\label{subsection:rep-groupoid}

Let $G$ be  a locally compact, Hausdorff, second countable
groupoid with a left Haar system. Then the
$C^{*}$-tensor category of representations of $G$ is
equivalent to the $C^{*}$-tensor category of
corepresentations of the $C^{*}$-pseudo-multiplicative
unitary associated to $G$, as will be explained now.
We use the notation and
results of subsections \ref{subsection:pmu-groupoid} and
\ref{subsection:legs-groupoid},
\begin{gather*}
  \begin{aligned}
    \frakK&:=L^{2}(G^{0},\mu), & \frakB&= \frakBo:=
    C_{0}(G^{0}) \subseteq {\cal L}(\frakK), &
    \frakb&:=(\frakK,\frakB,\frakBo),
  \end{aligned} \\
  \begin{aligned}
    H &:= L^{2}(G,\nu), &
    \alpha&=\beta:=j(L^{2}(G,\lambda)), & 
    \hbeta &:= \hat j(L^{2}(G,\lambda^{-1})),
  \end{aligned} \\
  V \colon \Hsource \cong L^{2}(\GsrG,\nu^{2}_{s,r}) \to
  L^{2}(\GrrG,\nu^{2}_{r,r}) \cong \Hrange, \\
  (V\omega)(x,y) = \omega(x,x^{-1}y)  \text{ for all }
  \omega \in C_{c}(\GsrG), \, (x,y) \in \GrrG, \\
  \begin{aligned}
    C_{0}(G) &\cong \hA_{V} \subseteq {\cal L}(H), &
    C^{*}_{r}(G) &= A_{V}  \subseteq {\cal L}(H),
  \end{aligned} \\
  \begin{aligned}
    \rho_{\hbeta} &= s^{*} \colon C_{0}(G^{0}) \to C_{b}(G)
    \to {\cal L}(H), & \rho_{\alpha} &= r^{*} \colon
    C_{0}(G^{0}) \to C_{b}(G) \to {\cal L}(H),
  \end{aligned}
\end{gather*}
and fix further notation. Let $X$ be a locally
compact Hausdorff space, $E$ a Hilbert $C^{*}$-module over
$C_{0}(X)$ and $x \in X$. We denote by $\chi_{x} \colon C(X)
\to \complex$ the evaluation at $x$ and by $E_{x}:=E
\tr_{\chi_{x}} \complex$ the fiber of $E$ at $x$; this is
the Hilbert space associated to the sesquilinear form
$(\eta,\eta') \mapsto \langle \eta|\eta'\rangle(x)$ on
$E$. Given an element $\xi \in E$ and an operator $T \in
{\cal L}_{C_{0}(X)}(E)$, we denote by $\xi_{x}:=\xi
\tr_{\chi_{x}} 1 \in E_{x}$ and $T_{x}:=T \tr_{\chi_{x}}
\Id_{\complex} \in {\cal L}(E_{x})$ the values of $\xi$ and
$T$, respectively, at $x$.  Given a locally compact
Hausdorff space $Y$ and a continuous map $p \colon Y \to X$,
the pull-back of $E$ along $p$ is the Hilbert $C^{*}$-module
$p^{*}E:=E \tr_{p^{*}} C_{0}(Y)$ over $C_{0}(Y)$, where
$p^{*} \colon C_{0}(X) \to M(C_{0}(Y))$ denotes the
pull-back on functions.  This pull-back is functorial, that
is, if $Z$ is a locally compact Hausdorff space and $q
\colon Z \to Y$ is a continuous map, then $(p \circ q)^{*}E$
is naturally isomorphic to $q^{*}p^{*}E$.  For $\xi,T$ as
above and all $y \in Y$, we have $(p^{*}\xi)_{y}=\xi_{p(y)}$
and $(p^{*}T)_{y}=T_{y}$.

The first part of the following definition is a special case
of \cite[D\'efinition 4.4]{legall:2}:
\begin{definition}
  A {\em continuous representation of $G$} consists of a
  Hilbert $C^{*}$-module $E$ over $C_{0}(G^{0})$ and a
  unitary $U \in {\cal L}_{C_{0}(G)}(s^{*}E,r^{*}E)$ such
  that $U_{x}U_{y}=U_{xy}$ for all $(x,y) \in \GsrG$.  We
 denote by $\bfcs\bfrep_{G}$ the category of continuous
  representations of $G$, where the morphisms between
  representations $(E,U_{E})$ and $(F,U_{F})$ are all
  operator $T \in {\cal L}_{C_{0}(G^{0})}(E,F)$ satisfying
  $U_{F}\circ s^{*}T=r^{*}T \circ U_{E}$ in ${\cal
    L}_{C_{0}(G)}(s^{*}E,r^{*}F)$.
\end{definition}

The verification of the following result is straightforward:
\begin{proposition}
  \begin{enumerate}
  \item Let $(E,U_{E})$ and $(F,U_{F})$ be continuous
    representations of $G$ and represent $C_{0}(G^{0})$ on
    $F$ by right multiplication operators.  Then $(E \tr
    F)_{x}= E_{x} \otimes F_{x}$ for all $x \in G$, and
    there exists a continuous representation $U_{E}\boxtimes
    U_{F}$ of $G$ on $E \tr F$ such that $(U_{E} \boxtimes
    U_{F})_{x} = (U_{E})_{x} \otimes (U_{F})_{x}$ for all $x
    \in G$.
  \item If $S_{i}$ is a morphism of continuous
    representations $(E_{i},U_{i,E})$ and $(F_{i},U_{i,F})$
    for $i=1,2$, then $S_{1} \tr S_{2}$ is a morphism
    between $(E_{1} \tr E_{2}, U_{1,E} \boxtimes U_{2,E})$
    and $(F_{1} \tr F_{2}, U_{1,F} \boxtimes U_{2,F})$.
  \item The category $\bfcs\bfrep_{G}$ carries the structure
    of a $C^{*}$-tensor category such that 
    \begin{itemize}
    \item the tensor product is given by the constructions
      in i) and ii);
    \item the associativity isomorphism
      $a_{(E_{1},U_{1}),(E_{2},U_{2}),(E_{3},U_{3})}$ is the
      canonical isomorphism $(E_{1} \tr E_{2}) \tr E_{3} \to
      E_{1} \tr (E_{2} \tr E_{3})$ for all
      $(E_{1},U_{1}),(E_{2},U_{2}),(E_{3},U_{3})$;
    \item the unit consists of the Hilbert $C^{*}$-module
      $C_{0}(G^{0})$ and the canonical isomorphism
      $s^{*}C_{0}(G^{0})\cong C_{0}(G) \cong
      r^{*}C_{0}(G^{0})$;
    \item the isomorphisms $l_{(E,U)}$ and $r_{(E,U)}$ are
      the canonical isomorphisms $C_{0}(G^{0}) \tr E \cong E
      \cong E \tr C_{0}(G^{0})$ for each $(E,U)$. \qed
  \end{itemize}
  \end{enumerate}
\end{proposition}
Define $p_{1},p_{2},m \colon \GsrG \to G$ by $p_{1}(x,y)=x$,
$p_{2}(x,y)=y$, $m(x,y)=xy$ and $r_{1},t,s_{2} \colon \GsrG
\to G^{0}$ by $r_{1}(x,y)=r(x), t(x,y)=s(x),
s_{2}(x,y)=s(y)$. Then we have a commutative diagram
\begin{align} \label{eq:rep-groupoid-global} 
    \xymatrix@R=15pt{
      & G^{0} & \\
      G  \ar[ru]^{s} \ar[d]_{r} & {\GsrG} \ar[u]^{t}
      \ar[ld]^{r_{1}} \ar[rd]_{s_{2}}
      \ar[d]_{m} \ar[l]^{p_{1}} \ar[r]_{p_{2}} & G
      \ar[lu]_{r}  \ar[d]^{s} \\ 
      G^{0} & G \ar[l]^{r} \ar[r]_{s} & G^{0}.
    }
\end{align}
\begin{lemma} \label{lemma:rep-groupoid-global}
  Let $E$ be a Hilbert $C^{*}$-module over $C_{0}(G^{0})$
  and $U \in {\cal L}_{C_{0}(G)}(s^{*}E,r^{*}E)$. Then
  $U_{x}U_{y}=U_{xy}$ for all $x,y \in \GsrG$ if and only if
  $m^{*}U$ is equal to the composition
\begin{align*}
  s_{2}^{*}E \xrightarrow{p_{2}^{*}U} t^{*}E
  \xrightarrow{p_{1}^{*}U} r_{1}^{*}E.
\end{align*}
\end{lemma}
\begin{proof}
  $((p_{1}^{*}U)(p_{2}^{*}U))_{(x,y)} = U_{x}U_{y} $,
  $(m^{*}U)_{(x,y)}=U_{xy}$ for all $(x,y) \in \GsrG$.
\end{proof}
We need the following straightforward result which
involves the operators defined in \eqref{eq:ketbra}:
\begin{lemma} \label{lemma:rep-groupoid-iso} Let $p \colon Y
  \to X$ be a continuous map of locally compact Hausdorff
  spaces, let $L$ be a Hilbert space with a nondegenerate
  injective $*$-homomorphism $C_{0}(Y) \hookrightarrow {\cal
    L}(L)$, and let $\gamma$ be a Hilbert $C^{*}$-module
  over $C_{0}(X)$.
  \begin{enumerate}
  \item There exists an isomorphism of $\Phi_{L,\gamma}^{f}
    \colon L \tl f^{*}\gamma \to L {_{f^{*}}\tl} \gamma$ of
    Hilbert spaces given by $\zeta \tl (\xi \tr_{f^{*}} g)
    \mapsto g\zeta {_{f^{*}}\tl} \xi$.
  \item There exists an isomorphism $\Psi_{L,\gamma}^{f}
    \colon f^{*}\gamma \to [r^{f^{*}}_{L}(\gamma)C_{0}(Y)]
    \subseteq {\cal L}(L, L {_{f^{*}}\tl} \gamma)$ of
    Hilbert $C^{*}$--modules over $C_{0}(Y)$ given by $\xi
    \tl_{f^{*}} g \mapsto r^{f^{*}}_{L}(\xi)g$.
  \item For all $U \in {\cal L}_{C_{0}(Y)}(f^{*}\gamma)$ and
    $\omega \in f^{*}\gamma$, we have
    $\Phi^{f}_{L,\gamma}(\Id_{L} \tl
    U)(\Phi^{f}_{L,\gamma})^{*} \Psi^{f}_{L,\gamma}(\omega)=
    \Psi^{f}_{L,\gamma}(U\omega)$ in ${\cal L}(L,L
    {_{f^{*}}\tl}\gamma)$. \qed
  \end{enumerate}
\end{lemma}
To each representation of $G$, we functorially associate a
corepresentation of $V$:
\begin{proposition} \label{proposition:rep-groupoid-rep}
  \begin{enumerate}
  \item Let $(E,U)$ be a continuous representation of
    $G$. Put $K:=\frakK \tl E$ and identify $E$ with the
    subspace $\gamma=\delta:=\{ r_{\frakK}(\xi) \mid \xi
    \in E\} \subseteq {\cal L}(\frakK,K)$ via $\xi \mapsto
    r_{\frakK}(\xi)$. Then $\cKd$ is a
    $C^{*}$-$(\frakb,\frakbo)$-module, we have canonical
    identifications
    \begin{align*}
      H {_{s^{*}}\tl} E &\cong H {_{\rho_{\hbeta}}\tl}
      \gamma \cong \crHsource, & H {_{r^{*}}\tl} E &\cong H
      {_{\rho_{\alpha}}\tl} \delta \cong \crHrange,
    \end{align*}
    and $\cKd$ together with the unitary
    $X:=\Phi^{r}_{H,\gamma}(\Id_{H} \tl
    U)\Phi^{s}_{H,\gamma}$ form a corepresentation $\bfF
    (E,U) := (\cKd,X)$ of $V$.
  \item Let $T$ be a morphism of continuous representations
    $(E,U_{E})$, $(F,U_{F})$ of $G$. Then $\bfF
    T:=\Id_{\frakK} \tl T$ is a morphism of the
    corepresentations $\bfF(E,U_{E})$, $\bfF(F,U_{F})$.
  \item The assignments $(E,U) \mapsto \bfF(E,U)$ and $T
    \mapsto \bfF T$ form a functor
    $\bfF\colon\bfcs\bfrep_{G}\to \bfcs\bfcorep_{V}$.
  \end{enumerate}
\end{proposition}
\begin{proof}
  i) The assertion on $\cKd$ is easily checked. In ${\cal
    L}(H, \crHrange)$, we have
  \begin{align*}
    \Phi^{r}_{H,\gamma}(\Id_{H} \tl
    U)(\Phi^{s}_{H,\gamma})^{*}[\kgamma{2}C_{0}(G)]
    &=\Phi^{r}_{H,\gamma}(\Id_{H} \tl U) r_{H}(s^{*}\gamma)
    \\ &= \Phi^{r}_{H,\gamma} r_{H}(Us^{*}\gamma)
    =\Phi^{r}_{H,\gamma} r_{H}(r^{*}\gamma) =
    [\kgamma{2}C_{0}(G)].
  \end{align*}
  Note here that $r_{H}(\frei)$ denotes the operator defined in
  \eqref{eq:ketbra}, while $r$ denotes the range map of $G$.
  Using the relations $\gamma=\delta$ and
  $\alpha=\beta=[C_{0}(G)\alpha]$, we can conclude
  \begin{align*}
    X[\kgamma{2}\alpha] &= [\kgamma{2}\alpha] =
    [\kalpha{1}\gamma], & X[\kgamma{2}\beta] &=
    [|\delta\rangle_{2}\beta], & X[\khbeta{1}\delta] &=
    X[|\delta\rangle_{2}\hbeta] =
    [|\delta\rangle_{2}\hbeta].
  \end{align*}
  To finish the proof, we have to show that diagram
  \eqref{eq:rep-corep-pentagon} commutes.  We apply Lemma
  \ref{lemma:rep-groupoid-iso} to the maps $p=r_{1},t,s_{2}
  \colon \GsrG \to G^{0}$ in diagram
  \eqref{eq:rep-groupoid-global}, the space $L=\Hsource$ and
  the representation $C_{0}(\GsrG) \to {\cal L}(L)$ given by
  multiplication operators, use the relations $r_{1}^{*} =
  \rho_{(\alpha \lt \alpha)}$ and $ s_{2}^{*} =
  \rho_{(\hbeta \rt \hbeta)}$, and find that $X_{13}X_{23}$
  is equal to the composition
 \begin{gather*}
   L {_{s_{2}^{*}}\tl} \gamma
   \xrightarrow{(\Phi_{L,\gamma}^{s_{2}})^{*}} L \tl
   s_{2}^{*}\gamma \xrightarrow{\Id_{L} \tl p_{2}^{*}U} L
   \tl t^{*}\gamma \xrightarrow{\Id_{L} \tl p_{1}^{*}U} L
   \tl r_{1}^{*}\gamma \xrightarrow{\Phi_{L,\gamma}^{r_{1}}}
   L {_{r_{1}^{*}}\tl} \gamma, 
 \end{gather*}
 which coincides by Lemma \ref{lemma:rep-groupoid-global}
 with $\Phi_{L,\gamma}^{r_{1}}(\Id_{L} \tl
 m^{*}U)(\Phi_{L,\gamma}^{s_{2}})^{*}$.  Since $ V_{12}^{*}
 p_{2}^{*}(f)V_{12} = \hDelta_{V}(f) = m^{*}f$ for all $f \in
 C_{0}(G)$, this composition is equal to
 $V_{12}^{*}X_{23}V_{12}$.

 ii), iii) Straightforward.
\end{proof}
Conversely, we functorially associate to every
corepresentation of $V$ a representation of $G$. In the
formulation of this construction, we apply Lemma
\ref{lemma:rep-groupoid-iso} to $Y=G$ and $L=H$.
\begin{proposition} \label{proposition:rep-groupoid-corep}
  \begin{enumerate}
  \item Let $X$ be a corepresentation of $V$ on a
    $C^{*}$-$(\frakb,\frakbo)$-module $\cKd$.  Then
    $\gamma=\delta$, $X[\kgamma{2}C_{0}(G)] =
    [\kgamma{2}C_{0}(G)]$, and  $\gamma$ together with
    the unitary $U:=\Psi_{H,\gamma}^{r} X
    (\Psi^{s}_{H,\gamma})^{*} \colon s^{*}\gamma \to
    r^{*}\gamma$ form a continuous representation
    $\bfG(\cKd,X):=(\gamma,U)$ of $G$.
  \item If $T$ is a morphism of corepresentations $(\cKd,X)$
    and $(\eLf,Y)$, then the map $\bfG T\colon \gamma \to
    \epsilon$, $\xi \mapsto T\xi$, is a morphism of the
    continuous representations $\bfG(\cKd,X)$,
    $\bfG(\eLf,Y)$.
  \item The assignments $(\cKd,X) \mapsto \bfG(\cKd,X)$,
    $T \mapsto \bfG T$ form a functor $\bfG\colon
    \bfcs\bfcorep_{V} \to \bfcs\bfrep_{G}$.
  \end{enumerate}
\end{proposition}
\begin{proof}
  i) Since $\alpha=\beta$ and $[\kalpha{1}\gamma] =
  X[\kgamma{2}\alpha] =
  X[\kgamma{2}\beta]=[|\delta\rangle_{2}\beta]=[\kbeta{1}\delta]$
  as subsets of ${\cal L}(\frakK,\crHrange)$, we can
  conclude $\gamma = [\rho_{\delta}(\frakB)\gamma] =
  [\balpha{1}\kalpha{1}\gamma] = [\balpha{1}\kbeta{1}\delta]
  = [\rho_{\delta}(\frakB)\delta]=\delta$. The relation
  $X[\kgamma{2}C_{0}(G)] = [\kgamma{2}C_{0}(G)]$ will follow
  from Example \ref{examples:pmu-regular} ii) and
  Proposition \ref{proposition:regular-mult}. Finally, $X =
  \Phi^{r}_{H,\gamma}(\Id_{H} \tl U)\Phi^{s}_{H,\gamma}$,
  and reversing the arguments in the proof of Proposition
  \ref{proposition:rep-groupoid-rep}, we conclude from
  $X_{12}X_{13}=V_{12}^{*}X_{23}V_{12}$ that $p_{1}^{*}U
  \circ p_{2}^{*}U=m^{*}U$. By Lemma
  \ref{lemma:rep-groupoid-global}, $U$ is a representation
  on $\gamma$.

  ii), iii) Straightforward.
\end{proof}
We define an equivalence between $C^{*}$-tensor
categories to be an equivalence of the underlying
$C^{*}$-categories and tensor
categories \cite{maclane}.
\begin{theorem}
  The functors $\bfcs\bfrep_{G}
  \stackrel{\bfG}{\underset{\bfF}{\leftrightarrows}}
  \bfcs\bfcorep_{V}$ extend to an equivalence of
  $C^{*}$-tensor categories.
\end{theorem}
\begin{proof}
  Lemma \ref{lemma:rep-groupoid-iso} iii) implies that the
  functors $\bfF,\bfG$ form equivalences of categories. The
  verification of the fact that they preserve the monoidal
  structure is  tedious but straightforward.
\end{proof}
\begin{remark}
  Let us note that a similar equivalence holds between the
  categories of measurable representations of $G$ and
  $W^{*}$-corepresentations of $V$.
\end{remark}

\section{Regular, proper and \'etale
  $C^{*}$-pseudo-multiplicative unitaries}

\label{section:props}

In this section, we study particular classes of
$C^{*}$-pseudo-multiplicative unitaries. As before, let
$\frakb=\cbasesb$ be a $C^{*}$-base, let $\pmuspace$ be a
$C^{*}$-$(\frakbo,\frakb,\frakbo)$-module, and let
$\pmuoperator$ be
a $C^{*}$-pseudo-multiplicative unitary.

\subsection{Regularity }
\label{subsection:regular}

In \cite{baaj:2}, Baaj and Skandalis showed that the pairs
$(\hA_{V},\hDelta_{V})$ and $(A_{V},\Delta_{V})$ associated
to a multiplicative unitary $V$ on a Hilbert space $H$ form
Hopf $C^{*}$-algebras if the unitary satisfies the
regularity condition $[\langle H|_{2}V|H\rangle_{1}]={\cal
  K}(H)$. This condition was generalized by Baaj in
\cite{baaj:12,baaj:10} and extended to pseudo-multiplicative
unitaries by Enock \cite{enock:10}. 

We now formulate a generalized regularity condition for
$C^{*}$-pseudo-multiplicative unitaries and show that the
pairs $((\hA_{V})^{\alpha,\hbeta}_{H},\hDelta_{V})$ and
$((A_{V})_{H}^{\beta,\alpha},\Delta_{V})$ associated to such
a unitary $V$ in subsection \ref{subsection:legs-pmu} are
concrete Hopf $C^{*}$-bi\-modules if $V$ is regular.  This
regularity condition involves
the space
\begin{align*} C_{V}:= [ \balpha{1} V
\kalpha{2}] \subseteq {\cal L}(H).
\end{align*}
\begin{proposition} \label{proposition:pmu-c} We have
\begin{gather*}
  \begin{aligned}{}
    [C_{V}C_{V}]&=C_{V}, & C_{V^{op}}&=C_{V}^{*}, &
    [C_{V}\alpha]&=\alpha, 
  \end{aligned} \\
  [C_{V}\rho_{\beta}(\frakB)]=[\rho_{\beta}(\frakB)C_{V}]
  = C_{V} = [C_{V}\rho_{\hbeta}(\frakB)] =
  [\rho_{\hbeta}(\frakB)C_{V}].
\end{gather*}
\end{proposition}
\begin{proof} The proof is completely analogous to the proof
  of Proposition \ref{proposition:pmu-legs}; for example,
  the first equation follows from the commutativity of the
  following two diagrams:
  \begin{gather*} \hspace{-0.1cm} \smalldiagram
    \xymatrix@C=8pt@R=14pt{ & H \ar[d]^{ \kalpha{2}}
      \ar[rr]_(0.375){C_{V}} \ar `l/0pt[l] [ld]^{\kalpha{2}}
      & & H
      \ar[rd]^(0.65){\kalpha{2}} \ar[rr]_(0.625){C_{V}} & & H & \\
      {\Hsource \hspace{-0.6cm}} \ar
      `d[dd]^(0.75){\kalpha{2}} [ddr] & {\Hsource}
      \ar[r]^(0.55){V} \ar[d]^{\kalpha{3}} & {\Hrange}
      \ar[ru]^(0.35){\balpha{1}} \ar[rd]+<-10pt,10pt>_(0.45){ \kalpha{3}}
      & & {\Hsource} \ar[r]^(0.45){V} 
      \ar@{<-}[ld]+<10pt,10pt>^(0.45){\balpha{1}} & {\Hrange}
      \ar[u]^{\balpha{1}} & {\hspace{-0.6cm}\Hrange} \ar
      `u[u]^{\balpha{1}} [ul] \\ & {\Hone} \ar[r]_(0.75){
        V_{12}} \ar[d]^{V_{23}} &&
      {\hspace{-1cm}\Htwo\hspace{-1cm}}
       & \ar[r]_(0.25){V_{23}}&
      {\Hthree} \ar[u]^{\balpha{1}} & \\ & {\Hfive}
      \ar[rrrr]^{V_{13}} & &&& {\Hfour} \ar[u]^{ V_{12}} \ar
      `r/0pt[r] [ruu]^(0.25){\balpha{2}} & } \\
    \smalldiagram \xymatrix@C=30pt@R=12pt{ & H \ar `l/0pt[l]
      [ld]^{{\kalpha{2}}} \ar[r]_{C_{V}}
      \ar[d]^{{\kalpha{2}}} & H \\ {\Hsource} \ar
      `d/0pt[d]^{\kalpha{2}} [dr] \ar[r]^(0.55){\rho_{(\hbeta
          \rt \beta)}(\frakB)} & {\Hsource} \ar[r]^{ V} &
      {\Hrange} \ar[u]^{\balpha{1}} \\ & {\Hfive}
      \ar[u]^{\balpha{2}} \ar[r]^{V_{13}} & {\Hfour}
      \ar[u]^{\balpha{2}} } \\[-4.5ex] \qedhere
  \end{gather*}
\end{proof}
\begin{definition} \label{definition:pmu-regular}
  A $C^{*}$-pseudo-multiplicative unitary
  $(\frakb,H,\hbeta,\alpha,\beta,V)$ is {\em
    semi-regular} if $C_{V} \supseteq [\alpha\alpha^{*}]$,
  and {\em regular} if $C_{V}=[\alpha\alpha^{*}]$.
\end{definition}
\begin{examples} \label{examples:pmu-regular}
  \begin{enumerate}
  \item By Proposition \ref{proposition:pmu-c}, $
      V$ is (semi-)regular if and only if $V^{op}$
    is (semi-)regular.
  \item The $C^{*}$-pseudo-multiplicative unitary associated
    to a locally compact Hausdorff groupoid $G$ as in
    Theorem \ref{theorem:pmu-groupoid} is regular. To prove
    this assertion, we use the notation introduced in
    subsection \ref{subsection:pmu-groupoid} and calculate
    that for each $\xi,\xi' \in C_{c}(G)$, $\zeta \in
    C_{c}(G) \subseteq L^{2}(G,\nu)$, $y \in G$,
  \begin{align*}
    \big(\langle j(\xi')|_{1} V |j(\xi)\rangle_{2}
    \zeta\big)(y) &= \int_{G^{r(y)}} \overline{\xi'(x)}
    \zeta(x)\xi(x^{-1}y)
    \intd\lambda^{r(y)}(x),\\
    \big(j(\xi')j(\xi)^{*}\zeta\big)(y) &= \xi'(y)
    \int_{G^{r(y)}} \overline{\xi(x)} \zeta(x)
    \intd\lambda^{r(y)}(x).
  \end{align*}
  Using standard approximation arguments, we find $[
  \balpha{1}V \kalpha{2}]= [S(C_{c}(\GrrG))] =[
  \alpha\alpha^{*}]$, where for each $\omega \in
  C_{c}(\GrrG)$, the operator $S(\omega)$ is given by
  \begin{gather*}
    (S(\omega)\zeta)(y)= \int_{G^{r(y)}} \omega(x,y)
    \zeta(x)d\lambda^{r(y)}(x) \quad \text{for all }
    \zeta\in C_{c}(G), \, y \in G.
  \end{gather*}
  \item In \cite{timmer:compact}, we introduce compact
    $C^{*}$-quantum groupoids and construct for each such
    quantum groupoid a $C^{*}$-pseudo-multiplicative
    unitary that turns out to be regular.
\end{enumerate}
\end{examples}
We shall now deduce several properties of semi-regular and
regular $C^{*}$-pseudo-multiplicative unitaries, using
commutative diagrams as explained in Notation
\ref{notation:diagrams}.
\begin{proposition} \label{proposition:regular-c} If $
    V$ is semi-regular, then $C_{V}$ is a $C^{*}$-algebra.
\end{proposition}
\begin{proof}
  Assume that $ V$ is regular. Then the following two
  diagrams commute, whence $[C_{V}C_{V}^{*}] =
  [\balpha{1}\balpha{1} V_{23} \kalpha{1}\kalpha{2}]=C_{V}$:
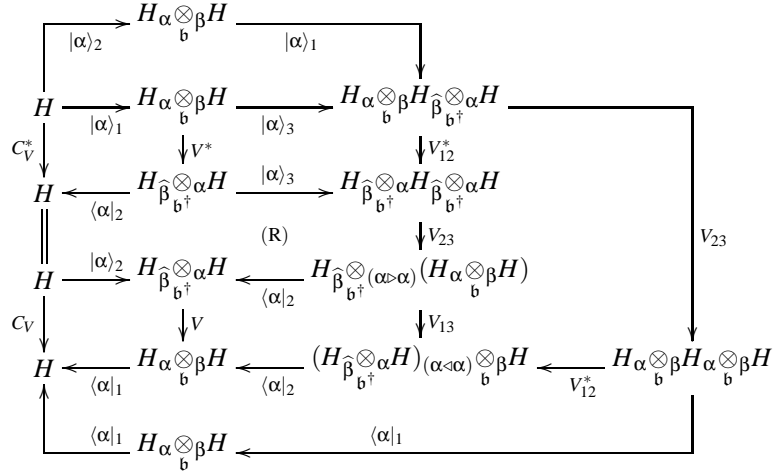
  \begin{align*}
    \smalldiagram \xymatrix@R=10pt@C=25pt{ & \Hrange \ar
      `r/0pt[r]_{\kalpha{1}} [rd] &
      \\
      H \ar `u/0pt[u] [ur]_(0.3){\kalpha{2}}
      \ar[d]_{C_{V}^{*}} \ar[r]_(0.45){\kalpha{1}} &{\Hrange
      } \ar[r]_(0.4){\kalpha{3}} \ar[d]^{V^{*}} & {\Htwo}
      \ar[d]^{ V_{12}^{*}} \ar `r/0pt[r] [rddd]^{V_{23}} &
      \\ H \ar@{=}[d] &{\Hsource} \ar[l]^(0.55){\balpha{2}}
      \ar@{}[rd]|{\scriptstyle\mathrm{(R)} \qquad}
      \ar[r]^(0.4){ \kalpha{3}} &{\Hone} \ar[d]^{V_{23}} &
      \\H \ar[r]^(0.45){\kalpha{2}} \ar[d]_{C_{V}}
      &{\Hsource} \ar[d]^{ V} &{\Hfive} \ar[d]^{V_{13}}
      \ar[l]^(0.6){ \balpha{2}} &\\ H & {\Hrange}
      \ar[l]^(0.55){\balpha{1}} & {\Hfour}
      \ar[l]^(0.6){\balpha{2}} &{\Hthree}
      \ar[l]^(0.4){V_{12}^{*}} \ar `d[d]
      [dll]_(0.6){\balpha{1}} \\ & {\Hrange} \ar
      `l[l]_(0.55){\balpha{1}} [lu] & & }
  \end{align*}
  \begin{align*} 
    \smalldiagram 
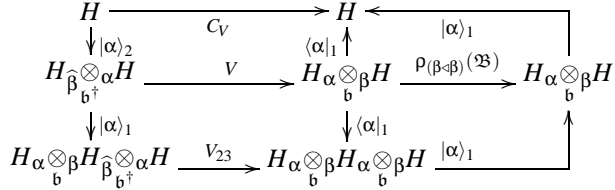
\xymatrix@C=30pt@R=10pt{ H 
      \ar[r]_{C_{V}}  \ar[d]^{\kalpha{2}}
      & H 
      &  \\ {\Hsource}  \ar[r]^(0.55){V} \ar[d]^{\kalpha{1}}
      & {\Hrange}  \ar[u]^{\balpha{1}}
      \ar[r]^{\rho_{(\beta \lt \beta)}(\frakB)} 
      \ar[d]^{\balpha{1}} 
&
      {\Hrange} 
      \ar `u/0pt[u]  [ul]^{{\kalpha{1}}} 
      \\ {\Htwo} \ar[r]^{V_{23}} & {\Hthree}
       \ar `r/0pt[r]^{\kalpha{1}} [ru]  & } 
  \end{align*}
  Now, assume that $ V$ is semi-regular. Then
  cell (R) in the first diagram need not commute, but still
  $[\kalpha{2}\balpha{2}] \subseteq
  [\balpha{2}V_{23}\kalpha{3}]$ and hence
  $[C_{V}C_{V}^{*}]\subseteq[\balpha{1}\balpha{1}V_{23}\kalpha{1}\kalpha{2}]=C_{V}$. 
  A similar argument shows that also $[C_{V}^{*}C_{V}]
  \subseteq C_{V}$, and from Proposition
  \ref{proposition:pmu-legs} and \cite[Lemme
  3.3]{baaj:12}, it follows that $C_{V}$ is a
  $C^{*}$-algebra.
\end{proof}
\begin{proposition} \label{proposition:semi-regular} Assume
  that $C_{V}=C_{V}^{*}$.
  \begin{enumerate}
  \item  Let $(\cKhd,X)$ be a representation of ${
      V}$. Then $(\hA_{X})_{K}^{\gamma,\hdelta}$ is a
    $C^{*}$-$(\frakb,\frakbo)$-algebra.
  \item $(\hA_{V})^{\alpha,\hbeta}_{H}$ is a
    $C^{*}$-$(\frakb,\frakbo)$-algebra and
    $(A_{V})_{H}^{\beta,\alpha}$ a
    $C^{*}$-$(\frakbo,\frakb)$-algebra.
  \end{enumerate}
\end{proposition} 
The proof uses the following central lemma:
\begin{lemma} \label{lemma:semi-regular} Let $(\cKhd,X)$ be
  a representation of $ V$. Then $[ X(1 \botensor
  C_{V})X^{*}\kbeta{2}] = [ \kbeta{2}\hA_{X}^{*}]$.
 \end{lemma}
 \begin{proof} 
    The following diagram commutes and shows that we have $[X(1 \botensor
   C_{V})X^{*}\kbeta{2}] = [
   \kbeta{2}\balpha{2}X^{*}\kbeta{2}]=
   [\kbeta{2}\hA_{X}^{*}]$:
  \begin{gather*} \smalldiagram \xymatrix@R=10pt@C=20pt{ K
\ar[rrr]_{\kbeta{2}} \ar[d]^{\kbeta{2}} &&& {\rHrange}
\ar[ld]^(0.4){\kbeta{3}} \ar[ddd]^{X^{*}} \\ {\rHrange }
\ar[r]^(0.4){ \kalpha{3}} \ar[d]^{X^{*}} & {\rHtwo}
\ar[d]^{X_{12}^{*}} \ar[r]^{V_{23}}
\ar@{}[rdd]|{\scriptstyle \mathrm{(P)}} & {\rHthree}
\ar[dd]^{X_{12}^{*}} \\ {\rHsource} \ar[d]^{1 \botensor C_{V}}
\ar[r]^(0.4){\kalpha{3}} &{\rHone} \ar[d]^{V_{23}} & \\
{\rHsource} \ar[d]^{X} &{\rHfive} \ar[r]^{X_{13}}
\ar[l]^(0.6){\balpha{2}} &{\rHfour}
\ar[lld]^(0.4){\balpha{2}} &
{\rHsource} \ar[d]_{\balpha{2}} \ar[l]_(0.4){\kbeta{3}}
\\ {\rHrange} && & {K} \ar[lll]^{\kbeta{2}} 
}
  \end{gather*} Indeed, cell (P) commutes by
\eqref{eq:rep-pentagon}, and the remaining cells 
because of \eqref{eq:rep-intertwine} or by inspection.
 \end{proof}

\begin{proof}[Proof of Proposition
\ref{proposition:semi-regular}] 
i) By Proposition \ref{proposition:rep-legs}, it suffices to
show that $\hA_{X}=\hA_{X}^{*}$. But by Proposition
\ref{proposition:rep-legs} and Lemma
\ref{lemma:semi-regular}, $\hA_{X}^{*} = [
\rho_{\alpha}(\frakBo)\hA_{X}^{*} ] = [
\bbeta{2}\kbeta{2}\hA_{X}^{*}] = [ \bbeta{2} X(1 \botensor
C_{V})X^{*}\kbeta{2}]$.

ii) Statement i) applied to $(\cKhd,X)=(\aHhb,V)$ yields the
first assertion. The second one follows after replacing
$V$ by $V^{op}$, where we use Propositions
\ref{proposition:pmu-legs} and \ref{proposition:pmu-c}.
\end{proof} 
The main result of this subsection is the following:
\begin{theorem} \label{theorem:regular-legs} If $V$
  is semi-regular, then
  $\big((\hA_{V})^{\alpha,\hbeta}_{H},\hDelta_{V}\big)$
  and $\big((A_{V})_{H}^{\beta,\alpha},\Delta_{V}\big)$
  are normal Hopf $C^{*}$-bi\-modules.
\end{theorem}
\begin{proof} 
  We prove the assertion concerning
  $\big((\hA_{V})^{\alpha,\hbeta}_{H},\hDelta_{V}\big)$; for
  $\big((A_{V})_{H}^{\beta,\alpha},\Delta_{V}\big)$, the
  arguments are similar.  By Proposition
  \ref{proposition:semi-regular},
  $(\hA_{V})_{H}^{\alpha,\hbeta}$ is a
  $C^{*}$-$(\frakb,\frakbo)$-algebra, and by Proposition
  \ref{proposition:rep-legs}, applied to
  $A_{V^{op}}=\hA_{V}^{*}=\hA_{V}$, we have
  $\Delta_{V^{op}}(\hA_{V}) \subseteq (\hA_{V})
  \fibre{\hbeta}{\frakbo}{\alpha} (\hA_{V})$. Now, the claim follows
  from Lemma \ref{lemma:pmu-delta-hopf}.
\end{proof}
We collect several auxiliary results on regular
$C^{*}$-pseudo-multiplicative unitaries.
\begin{proposition}
  \label{proposition:regular-mult} Assume that $V$ is
  regular.
  \begin{enumerate}
  \item Let $(\cKhd,X)$ be a representation of $
      V$. Then $[ X\kalpha{2}\hA_{X}] = [ \kbeta{2}
    \hA_{X}]$ and $[ X |\hdelta\rangle_{1} A_{V}] = [
    \kgamma{1} A_{V}]$.
  \item $[ V\kalpha{2}\hA_{V}] = [
\kbeta{2} \hA_{V}]$ and $[
V\khbeta{1}A_{V}] = [ \kalpha{1}A_{V}]$.
  \end{enumerate}
\end{proposition}
\begin{proof} 
Using Lemma \ref{lemma:semi-regular} and the
relation $\hA_{X}=\hA_{X}^{*}$ (Proposition
\ref{proposition:semi-regular}), we find that $[
X\kalpha{2}\hA_{X}] = [
X\kalpha{2}\balpha{2}X^{*}\kbeta{2}] = [ X(1 \botensor
C_{V})X^{*}\kbeta{2}] = [\kbeta{2}\hA_{X}]$.

Replacing $(\cKhd,X)$ by $(\aHhb,V)$, we obtain the first
equation in ii), and replacing $V$ by $V^{op}$
and using Proposition \ref{proposition:pmu-legs}, we
obtain $[ \Sigma V^{*} \Sigma \kalpha{2} A_{V}^{*} ] = [
\khbeta{2} A_{V}^{*}]$, which yields
$[V\khbeta{1}A_{V}]=[\kalpha{1}A_{V}]$. 

Finally, let us prove the equation
$[X|\hdelta\rangle_{1}A_{V}]=[\kgamma{1}A_{V}]$.  The
following commutative diagram shows that
$[X_{13}|\hdelta\rangle_{1}\kalpha{1}A_{V}] =
[\kalpha{2}\kgamma{1}A_{V}]$:
  \begin{align*} \smalldiagram \xymatrix@R=10pt{ {H}
\ar[r]^{A_{V}} \ar@{=}[d] & {H} \ar[r]^(0.4){\kalpha{1}} &
{\Hrange} \ar[d]^{V^{*}} \ar[r]^(0.4){|\hdelta\rangle_{1}} &
{\rHfive} \ar[d]^{V_{23}^{*}} \ar `r/0pt[r]
`d[rdddd]^{X_{13}} [dddd] & \\ H \ar@{=}[d] \ar[r]^{A_{V}}&
H \ar[r]^(0.4){\khbeta{1}} & {\Hsource}
\ar[r]^(0.4){|\hdelta\rangle_{1}} & {\rHone} \ar[d]^{X_{12}}
& \\ H \ar@{=}[d] \ar[r]^{A_{V}} & H
\ar[r]^(0.4){\khbeta{1}}& {\Hsource}
\ar[r]^(0.4){\kgamma{1}} \ar[d]^{V} & {\rHtwo}
\ar[d]^{V_{23}} &\\ H \ar[r]^{A_{V}} \ar@{=}[d] & H
\ar[r]^(0.4){\kalpha{1}} & {\Hrange}
\ar[r]^(0.4){\kgamma{1}} & {\rHthree} \ar[d]^{X_{12}} & \\ H
\ar[r]^{A_{V}} & H \ar[r]^(0.4){\kgamma{1}}& {\rHrange}
\ar[r]^(0.4){\kalpha{2}}& {\rHfour} & }
  \end{align*} Moreover, also the following diagram
commutes,
\begin{align*} \smalldiagram \xymatrix@C=30pt@R=15pt{ H
    \ar[r]^{\kalpha{1}} \ar[rd]|(0.4){\rho_{\beta}(\frakB)}
    & {\Hrange} \ar[d]^{\balpha{1}}
    \ar[r]^(0.4){X_{13}|\hdelta\rangle_{1}} & {\rHfour}
    \ar[d]^{\balpha{2}} & {\rHrange}
    \ar[l]_(0.4){\kalpha{2}} \ar[ld]|(0.4){\rho_{(\hdelta \lt
        \beta)}(\frakB)} \\ H \ar[u]^{A_{V}} \ar[r]^{A_{V}} &
    H \ar[r]^{X|\hdelta\rangle_{1}} & {\rHrange} & H
    \ar[u]_{\kgamma{1}} \ar[l]_{\kgamma{1}} & H,
    \ar[l]_{A_{V}}}
  \end{align*} and hence
$[X|\hdelta\rangle_{1}A_{V}]=[\balpha{2}X_{13}|\hdelta\rangle_{1}\kalpha{1}]
= [\balpha{2}\kalpha{2}\kgamma{1}A_{V}] =[\kgamma{1}A_{V}]$.
  \end{proof} 
  The last result in this subsection involves the algebras
  $\hA_{\trivrep} = [\beta^{*}\alpha]$ and
  $\hA_{\triv_{V^{op}}}=[\hbeta^{*}\alpha]$
  associated to the trivial representations of $V$
  and $ V^{op}$, respectively.
  \begin{proposition} \label{proposition:pmu-regular-ab} If
    $ V$ is regular, then $[\beta
    \hA_{\trivrep}]=[\alpha \hA_{\trivrep}]$ and $[\hbeta
    A_{\triv_{ V^{op}}}]=[\alpha A_{\triv_{ V^{op}}}]$.
\end{proposition}
\begin{proof} The following diagram commutes
  \begin{gather*} \smalldiagram \xymatrix@R=10pt@C=30pt{
{\frakK} \ar[rr]_{\beta} \ar[d]^{\beta} & & H
\ar[r]_(0.6){\alpha^{*}} \ar[d]^{\kbeta{2}} & {\frakK}
\ar[d]^{\beta} \\ H \ar[r]^(0.4){\kalpha{2}} \ar `d/0pt[d]
`r[drrr]^{\alpha \alpha^{*}} [rrr]& {\Hsource} \ar[r]^{V} &
{\Hrange} \ar[r]^(0.6){\balpha{1}} & H \\ &&& }
\end{gather*} and shows that $[\beta \hA_{\trivrep}]=[\beta
\hA_{\trivrep}^{*}]=[\beta
\alpha^{*}\beta]=[\alpha\alpha^{*}\beta] = [\alpha
\hA_{\trivrep}^{*}]=[\alpha\hA_{\trivrep}]$.  The second
equation follows by replacing $ V$ with $ V^{op}$.
\end{proof}

\subsection{Proper and \'etale $C^{*}$-pseudo-multiplicative
unitaries }

\label{subsection:fixed}

In \cite{baaj:2}, Baaj and Skandalis characterized
multiplicative unitaries that correspond to compact or
discrete quantum groups by the existence of fixed or cofixed
vectors, respectively, and showed that from such vectors,
one can construct a Haar state and a counit on the
associated legs. We adapt some of their constructions to
$C^{*}$-pseudo-multiplicative unitaries as follows.  

Given a $C^{*}$-$\frakb^{(\dagger)}$-module $K_{\gamma}$,
let $M(\gamma)=\{T \in {\cal L}(\frakK,K) \mid
T\frakB^{(\dagger)}\subseteq \gamma, T^{*}\gamma \subseteq
\frakB^{(\dagger)}\}$.
\begin{definition} \label{definition:fixed} A {\em fixed
    element} for $V$ is an operator $\eta \in M(\hbeta) \cap
  M(\alpha) \subseteq {\cal L}(\frakK,H)$ satisfying
  $V|\eta\rangle_{1} = |\eta\rangle_{1}$. A {\em cofixed
    element} for $V$ is an operator $\xi \in M(\alpha) \cap
  M(\beta) \subseteq {\cal L}(\frakK,H)$ satisfying
  $V|\xi\rangle_{2}=|\xi\rangle_{2}$. We denote the set of
  all fixed/cofixed elements for $V$ by
  $\Fix(V)$/$\Cofix(V)$.
\end{definition}
\begin{example} \label{example:fixed-groupoid} Let us
  consider the $C^{*}$-pseudo-multiplicative unitary
  associated to a locally compact, Hausdorff, second
  countable groupoid $G$ in subsection
  \ref{subsection:pmu-groupoid}.  We identify
  $M(L^{2}(G,\lambda))$ in the natural way with the
  completion of the space
\begin{align*}
  \textstyle \Big\{ f \in C(G) \,\Big|\, r\colon \supp f \to
  G \text{ is proper}, \ \sup_{u \in G^{0}} \int_{G^{u}}
  |f(x)|^{2} \intd\lambda^{u}(x) \text{ is finite} \Big\}
\end{align*}
with respect to the norm $f \mapsto \sup_{u \in G^{0}}
\big(\int_{G^{u}} |f(x)|^{2}\intd\lambda^{u}(x)\big)^{1/2}$.
Similarly as in \cite[Lemma 7.11]{timmermann:hopf}, one
easily verifies that
\begin{enumerate}
\item $\eta_{0} \in M(L^{2}(G,\lambda))$ is a fixed element
  for $V$ if and only if for each $u \in G^{0}$,
  $\eta_{0}|_{G^{u} \setminus \{u\}} = 0$ almost everywhere
  with respect to $\lambda^{u}$;
\item $\xi_{0} \in M(L^{2}(G,\lambda))$ is a cofixed element
  for $V$ if and only if $\xi_{0}(x)=\xi_{0}(s(x))$ for all
  $x \in G$.
  \end{enumerate} 
\end{example}
Let us collect some easy properties of fixed and cofixed
elements.
\begin{remarks} \label{remarks:fixed}
  \begin{enumerate}
  \item $\Fix(V)=\Cofix(V^{op})$ and
$\Cofix(V)=\Fix(V^{op})$.
  \item $\Fix(V)^{*} \Fix(V)$ and $\Cofix(V)^{*}\Cofix(V)$
are contained in $M(\frakB) \cap M(\frakBo)$.
  \item The relations $\Fix(V) \subseteq M(\hbeta) \cap
M(\alpha)$ imply $\rho_{\alpha}(\frakBo)\Fix(V) =
\Fix(V)\frakBo \subseteq \hbeta$ and
$\rho_{\hbeta}(\frakB)\Fix(V) =\Fix(V)\frakB \subseteq
\alpha$.  Likewise, we have $\rho_{\beta}(\frakB)\Cofix(V) \subseteq
\alpha$ and $\rho_{\alpha}(\frakBo)\Cofix(V) \subseteq
\beta$.
  \end{enumerate}
\end{remarks}

\begin{lemma} \label{lemma:fixed} Let $\xi,\xi' \in
  \Cofix(V)$ and $\eta,\eta' \in \Fix(V)$. Then
  \begin{align*}
    \langle \xi|_{2} V |\xi'\rangle_{2} &=
\rho_{\alpha}(\xi^{*}\xi') = \rho_{\hbeta}(\xi^{*}\xi'), &
\langle \eta|_{1} V |\eta'\rangle_{1} &=
\rho_{\beta}(\eta^{*}\eta') = \rho_{\alpha}(\eta^{*}\eta').
  \end{align*}
\end{lemma}
\begin{proof} Let $\zeta \in H$. Then $\langle
  \xi|_{2} V |\xi'\rangle_{2} \zeta = \langle
  \xi|_{2}|\xi'\rangle_{2} \zeta =
  \rho_{\alpha}(\xi^{*}\xi')\zeta$ and $(\langle \xi|_{2} V
  |\xi'\rangle_{2})^{*}\zeta = \langle \xi'|_{2}
  |\xi\rangle_{2}\zeta = \rho_{\beta}((\xi')^{*}\xi)
  \zeta$. The second equation follows similarly.
\end{proof}

\begin{proposition} \label{proposition:fixed} 
    \begin{enumerate}
    \item $\rho_{\hbeta}(M(\frakB))\Cofix(V) \subseteq
      \Cofix(V)$ and $\rho_{\beta}(\frakB)\Fix(V) \subseteq
      \Fix(V)$.
\item $[\Cofix(V)\Cofix(V)^{*}\Cofix(V)]= \Cofix(V)$ and
  $[\Fix(V)\Fix(V)^{*}\Fix(V)]= \Fix(V)$.
    \item $[\Cofix(V)^{*}\Cofix(V)]$ and
$[\Fix(V)^{*}\Fix(V)]$ are $C^{*}$-subalgebras of $M(\frakB)
\cap M(\frakBo)$; in particular, they are commutative.
    \end{enumerate}
\end{proposition}
\begin{proof} We only prove the assertions concerning
$\Cofix(V)$; the other assertions follow similarly.

  i) Let $T \in M(\frakB)$ and $\xi \in \Cofix(V)$. Then
$\rho_{\hbeta}(T)\xi \subseteq M(\beta) \cap M(\alpha)$
because $\rho_{\hbeta}(\frakB)\beta \subseteq \beta$ and
$\rho_{\hbeta}(\frakB)\alpha \subseteq \alpha$. The
relation $V(\hbeta \rt \hbeta)=\alpha \rt \hbeta$
furthermore implies
 \begin{align*} V | \rho_{\hbeta}(T)\xi\rangle_{2} = V
\rho_{(\hbeta \rt \hbeta)}(T)|\xi\rangle_{2} = \rho_{(\alpha
\rt \hbeta)}(T)V|\xi\rangle_{2} = \rho_{(\alpha \rt
\hbeta)}(T) |\xi\rangle_{2} =
|\rho_{\hbeta}(T)\xi\rangle_{2}.
  \end{align*}

  ii) Using i) and the relation $\Cofix(V) \subseteq
  M(\hbeta)$, we find that
  \begin{align*} [\Cofix(V)\Cofix(V)^{*}\Cofix(V)]
&\subseteq [\Cofix(V) M(\frakBo)] \\ &=
[\rho_{\hbeta}(M(\frakBo))\Cofix(V)] \subseteq \Cofix(V).
  \end{align*}
  Therefore, $[\Cofix(V)^{*}\Cofix(V)]$ is a $C^{*}$-algebra
  and $\Cofix(V)$ is a Hilbert $C^{*}$-module over
  $[\Cofix(V)^{*}\Cofix(V)]$. Now, \cite[p.\ 5]{lance}
  implies that
  the inclusion above is an equality.

  iii) This follows from ii) and  Remark
  \ref{remarks:fixed} ii).
\end{proof}

\begin{definition} \label{definition:compact} We call the
  $C^{*}$-pseudo-multiplicative unitary $V$ {\em \'etale} if
  $\eta^{*}\eta =\Id_{\frakK}$ for some $\eta \in \Fix(V)$,
  {\em proper} if $\xi^{*}\xi = \Id_{\frakK}$ for some $\xi
  \in \Cofix(V)$, and {\em compact} if it is proper and
  $\frakB,\frakBo$ are unital.
\end{definition}
\begin{example}
  The $C^{*}$-pseudo-multiplicative unitary associated to a
  locally compact, Hausdorff, second countable groupoid $G$
  (Theorem \ref{theorem:pmu-groupoid}) is
  \'etale/proper/compact if and only if $G$ is
  \'etale/proper/compact. This follows from similar
  arguments as in \cite[Theorem 7.12]{timmermann:hopf}.
\end{example}
\begin{remarks} \label{remarks:compact} 
  \begin{enumerate}
  \item By Remark \ref{remarks:fixed}, $V$ is \'etale/proper
    if and only if $V^{op}$ is proper/\'etale.
  \item If $V$ is proper, then $\Id_{H} \in \hA_{V}$; if $V$
    is \'etale, then $\Id_{H} \in A_{V}$. This follows
    directly from Lemma \ref{lemma:fixed}.
  \end{enumerate}
\end{remarks}

The first main result of this subsection shows how one can
construct a counit on
$((\hA_{V})^{\alpha,\hbeta}_{H},\hDelta_{V})$ from a fixed
element for $V$.
\begin{theorem} \label{theorem:proper-counit} Let
  $V$ be an \'etale $C^{*}$-pseudo-multiplicative unitary.
  \begin{enumerate}
  \item There exists a unique  contractive homomorphism
    $\hepsilon \colon \hA_{V} \to \hA_{\triv}$ such that
    $\hpi_{\triv} = \hepsilon \circ \hpi_{V} \colon \tilde
    \Omega_{\beta,\alpha} \to \hA_{\triv}$.
  \item Assume that $V$ is regular.  Then $\hepsilon$ is a
    jointly normal morphism from
    $(\hA_{V})^{\alpha,\hbeta}_{H}$ to
    $(\hA_{\triv})^{\frakB,\frakBo}_{\frakK}$ and a bounded
    counit for
    $((\hA_{V})^{\alpha,\hbeta}_{H},\hDelta_{V})$. 
  \end{enumerate}
\end{theorem}
\begin{proof} Choose an $\eta_{0} \in \Fix(V)$ with
  $\eta_{0}^{*}\eta_{0}=\Id_{\frakK}$ and define
  $\hepsilon \colon \hA_{V} \to {\cal L}(\frakK)$  by
  $\hat a \mapsto \eta_{0}^{*} \hat a \eta_{0}$. Then
  $\hepsilon$ is contractive.  For all $\xi
  \in\alpha,\eta\in \beta,\zeta \in \frakK$,
  \begin{align*}  \langle \eta|_{2} V
    |\xi\rangle_{2} \eta_{0} \zeta = \langle \eta|_{2} V
    (\eta_{0} \tr \xi\zeta) = \langle\eta|_{2} (\eta_{0} \tr
    \xi\zeta) = \eta_{0} (\eta^{*}\xi)\zeta,
  \end{align*}
  and hence
  $\hpi_{V}(\omega)\eta_{0}=\eta_{0}\hpi_{\triv}(\omega)$
  for all $\omega \in \tilde \Omega_{\beta,\alpha}$. In
  particular,
  $\hepsilon(\hpi_{V}(\omega))=\eta_{0}^{*}\hpi_{V}(\omega)\eta_{0}
  =\eta_{0}^{*}\eta_{0} \hpi_{\triv}(\omega) =
  \hpi_{\triv}(\omega)$.

  Assume that $V$ is regular. Then $\hepsilon$ is a morphism
  as claimed because by construction, $\hepsilon$ is a
  $*$-homomorphism, $\eta_{0}^{*} \in {\cal
    L}^{\hepsilon}(\aHhb,{_{\frakB}\frakK_{\frakBo}})$, and
  $[\eta_{0}^{*}\alpha] \supseteq
  [\eta_{0}^{*}\eta_{0}\frakB] = \frakB$ and
  $[\eta_{0}^{*}\hbeta] \supseteq
  [\eta_{0}^{*}\eta_{0}\frakBo]=\frakBo$.  It remains to
  show that diagram \eqref{eq:counit} commutes. Clearly,
  $(\hepsilon \bofibre \Id)(x) =\langle \eta_{0}|_{1} x
  |\eta_{0}\rangle_{1}$ and $(\Id \bofibre \hepsilon)(x) =
  \langle \eta_{0}|_{2}x|\eta_{0}\rangle_{2}$ for all $x \in
  (\hA_{V}) \fibre{\hbeta}{\frakbo}{\alpha} (\hA_{V})$.  Now, the
  left square in diagram \eqref{eq:counit} commutes because
  for all $\hat a \in \hA_{V}$,
  \begin{align*} 
    \langle \eta_{0}|_{1}\hDelta_{V}(\hat
    a)|\eta_{0}\rangle_{1} &= \langle \eta_{0}|_{1} V^{*}(1
    \btensor \hat a)V|\eta_{0}\rangle_{1} = \langle
    \eta_{0}|_{1} (1 \btensor \hat a)|\eta_{0}\rangle_{1} =
    \rho_{\beta}(\eta_{0}^{*}\eta_{0})\hat a = \hat a.
  \end{align*} 
  To see that the left square in diagram \eqref{eq:counit}
  commutes, let $\eta \in \beta,\xi \in \alpha$ and consider
  the following diagram:
  \begin{gather*}\smalldiagram \xymatrix@C=17pt@R=20pt{ {H}
      \ar[d]^{|\eta_{0}\rangle_{2}} \ar[r]^{|\xi\rangle_{2}}
      & {\Hsource} \ar[d]^{|\eta_{0}\rangle_{2}} \ar[r]^{V}
      \ar@{}[dr]|{\mathrm{(*)}} & {\Hrange}
      \ar[rd]+<-2pt,7pt>_(0.4){|\eta_{0}\rangle_{2}}
      \ar[rr]^{\Id} && {\Hrange}
      \ar@{<-}[ld]+<2pt,7pt>^(0.4){\langle\eta_{0}|_{2}}
      \ar[r]^{\langle \eta|_{2}} & H \\ {\Hsource}
      \ar[r]^(0.4){|\xi\rangle_{3}} \ar `d/0pt[d]
      `r[drrrrr]^{\hDelta(\langle
        \eta|_{2}V|\xi\rangle_{2})} [rrrrr] & {\Hone}
      \ar[r]^(0.7){V_{13}V_{23}} &&
      {\hspace{-2cm}\Hfour\hspace{-2cm}} &
      \ar[r]^(0.3){\langle\eta|_{3}} & {\Hsource}
      \ar[u]^{\langle\eta_{0}|_{2}} \\ &&&&& }
  \end{gather*} The lower cell commutes by Lemma
  \ref{lemma:pmu-convolution-hdelta}, cell (*) commutes
  because $V_{23}|\eta_{0}\rangle_{2} =
  |\eta_{0}\rangle_{2}$, and the other cells commute as
  well. Since $\eta \in \beta$ and $\xi \in \alpha$ were
  arbitrary, the claim follows.
\end{proof}
As an example, we consider the unitary associated to a
groupoid (subsection \ref{subsection:pmu-groupoid}) .
\begin{proposition} \label{proposition:proper-counit-groupoid}
  Let $G$ be a locally compact, Hausdorff, second countable
  groupoid and let $\pmuoperator$ be the associated
  $C^{*}$-pseudo-multiplicative unitary.
  \begin{enumerate}
  \item Let $G$ be \'etale. Then $V$ is \'etale,
    $\hA_{V}\cong C_{0}(G)$, $\hA_{\triv} \cong
    C_{0}(G^{0})$, and $\hepsilon \colon \hA_{V} \to
    \hA_{\triv}$ is given by the restriction of functions on
    $G$ to functions on $G^{0}$.
  \item Let $G$ be proper. Then $V^{op}$ is \'etale,
    $A_{V}=\hA_{V^{op}} = C^{*}_{r}(G)$, and for each $f
    \in C_{c}(G)$, the operator $\hepsilon(L(f)) \in {\cal
      L}(L^{2}(G^{0},\mu))$ is given by
    \begin{align*}
      (\hepsilon(L(f))\zeta)(u) = \int_{G^{u}}
      f(x)D^{-1/2}(x)\zeta(s(x)) \intd \lambda^{u}(x) \quad
      \text{for all } \zeta \in L^{2}(G^{0},\mu), \ x\in G.
    \end{align*}
  \end{enumerate}
\end{proposition}
\begin{proof}
   For all $\xi,\xi' \in C_{c}(G)$, $\zeta \in
  L^{2}(G^{0},\mu)$ and $u \in G^{0}$, we have by Lemma
  \ref{lemma:legs-groupoid}
  \begin{align*}
    (\hepsilon(m(\bar \xi \ast \xi'{}^{*}))\zeta)(u) =
    (\hepsilon(\ha_{\xi,\xi'})\zeta)(u) &=
    (j(\xi)^{*}j(\xi')\zeta)(u) \\ &= \int_{G^{u}}
    \overline{\xi(x)}\xi'(x)\zeta(u) \intd\lambda^{u}(x) \\ &=
    ( \bar \xi \ast \xi'{}^{*})(u) \zeta(u), \\
    (\hepsilon(L(\bar \xi \xi'))\zeta)(u) =
    (\hepsilon(a_{\xi,\xi'})\zeta)(u) &= (j(\xi)^{*}\hat
    j(\xi')\zeta)(u) \\ &= \int_{G^{u}}
    \overline{\xi(x)}\xi'(x)D^{-1/2}(x)\zeta(s(x))
    \intd\lambda^{u}(x).  \qedhere
  \end{align*}
\end{proof}
The second main result of this subsection shows how one can
construct a Haar weight on
$((\hA_{V})^{\alpha,\hbeta}_{H},\hDelta_{V})$ from a cofixed
element for $V$.
\begin{theorem} \label{theorem:proper-haar} Let $V$ be a
  proper regular $C^{*}$-pseudo-multiplicative unitary.
  Then there exists a normal bounded left Haar weight $\phi$
  for $((\hA_{V})^{\alpha,\hbeta}_{H},\hDelta_{V})$.
\end{theorem}
\begin{proof} Choose $\xi_{0} \in \Cofix(V)$ with
  $\xi_{0}^{*}\xi_{0} = \Id_{\frakK}$. By Proposition
  \ref{proposition:pmu-legs} and Remark \ref{remarks:fixed}
  i), $[\xi_{0}^{*} \hA_{V} \xi_{0}] = [\xi_{0}^{*}
  \rho_{\alpha}(\frakBo)\hA_{V}
  \rho_{\alpha}(\frakBo)\xi_{0}] \subseteq
  [\beta^{*}\hA_{V}\beta] \subseteq \frakBo$.  Hence, we can
  define a completely positive map $\phi \colon \hA_{V} \to
  \frakBo$ by $\hat a \mapsto \xi_{0}^{*}\hat a \xi_{0}$,
  and $\phi \in \Omega_{M(\alpha)}(\hA_{V})$. For all $\hat
  a \in \hA_{V}$, $(\Id \ast \phi)(\hDelta_{V}(\hat a)) =
  \langle \xi_{0}|_{2} V^{*}(\Id \btensor \hat
  a)V|\xi_{0}\rangle_{2} = \langle \xi_{0}|_{2} (\Id \btensor
  \hat a) |\xi_{0}\rangle_{2} =
  \rho_{\alpha}(\xi_{0}^{*}\hat a \xi_{0})$ .
\end{proof}
As an example, we again consider the
 unitary associated to a
 groupoid.
\begin{proposition}
  Let $G$ be a locally compact, Hausdorff, second countable
  groupoid and let $\pmuoperator$ be the associated
  $C^{*}$-pseudo-multiplicative unitary.
  \begin{enumerate}
  \item Let $G$ be proper. Then $V$ is proper, $\hA_{V}
    \cong C_{0}(G)$, and the map $\phi \colon \hA_{V} \to
    C_{0}(G^{0})$ given by $(\phi(f))(u)=\int_{G^{u}} f(x)
    \intd \lambda^{u}(x)$ is a normal bounded left Haar
    weight for
    $((\hA_{V})^{\alpha,\hbeta}_{H},\hDelta_{V})$.
  \item Let $G$ be \'etale. Then $V^{op}$ is proper and
    there exists a normal bounded left and right Haar weight
    $\phi$ for $((A_{V})_{H}^{\beta,\alpha},\Delta_{V})$
    given by $L(f) \mapsto f|_{G^{0}}$ for all $f \in
    C_{c}(G)$.
  \end{enumerate}
\end{proposition}
\begin{proof}
  This follows from Theorem \ref{theorem:proper-haar} and
  similar calculations as in
  \ref{proposition:proper-counit-groupoid}.
\end{proof}

\def\cprime{$'$}

\end{document}